\documentclass[a4paper,11pt]{amsart}

\usepackage{amscd}
\usepackage{amsmath}
\usepackage{amsthm}
\usepackage{graphicx}
\usepackage{txfonts}
\usepackage{dsfont}
\usepackage{mathabx}
\usepackage{tikz-cd}
\usepackage[titletoc]{appendix}
\usepackage{soul}
\usepackage{graphics}

\usepackage{xcolor} 
\usepackage{colortbl} 
\definecolor{cellGreen}{HTML}{A8F4A8}    
\definecolor{cellRed}{HTML}{FF9999}      
\definecolor{cellBlue}{HTML}{99CCFF}  
\definecolor{cellLightPurple}{HTML}{E3CFF2}
\definecolor{cellLightgreen}{HTML}{C8F0C8} 
\definecolor{cellYellowRemaining}{HTML}{FFFEB3} 
\definecolor{cellBBlue}{HTML}{AEE6F6}     
\definecolor{cellReddishPink}{HTML}{F8BBD6}

\usepackage{geometry}
\usepackage{hyperref}
\usepackage{tikz}

\theoremstyle{definition}
\newtheorem{Defn}{Definition}[section]

\theoremstyle{plain}
\newtheorem{Lemma}[Defn]{Lemma}
\newtheorem{propo}[Defn]{Proposition}
\newtheorem{theorem}[Defn]{Theorem}
\newtheorem{Corollary}[Defn]{Corollary}

\theoremstyle{plain}
\newtheorem{THM}{Theorem}

\newtheorem{COR}[THM]{Corollary}

\theoremstyle{remark}
\newtheorem{remark}[Defn]{Remark}
\newtheorem{Example}[Defn]{Example}

\newtheorem{notation}[Defn]{Notation}

\DeclareMathOperator{\nat}{\mathbb{N}}
\DeclareMathOperator{\prim}{\mathbb{P}}
\DeclareMathOperator{\AC}{C^{\bullet}_{d}}
\DeclareMathOperator{\ACk}{C_d}
\DeclareMathOperator{\ACPN}{D^{\bullet}_{d,p^N}}
\DeclareMathOperator{\ACPNk}{D_{d,p^N}}%
\DeclareMathOperator{\ACPNQ}{D^{\bullet}_{d^\prime,p^N}}
\DeclareMathOperator{\ACPNQk}{D_{d^\prime,p^N}}
\DeclareMathOperator{\gr}{\mathsf{gr}}
\DeclareMathOperator{\f}{\mathcal{F}}

\title{On integral extensions between the abelianization functor and its symmetric powers}
\author{Minkyu Kim}
\address{KIAS, Seoul, South Korea}
\email{kimminq@kias.re.kr}
\urladdr{https://minq92.github.io/Gaeul-Autumn/}
\author{Christine Vespa}
\address{Aix Marseille Univ, CNRS, I2M, Marseille, France.}
\email{christine.vespa@univ-amu.fr}
\urladdr{https://www.i2m.univ-amu.fr/perso/christine.vespa/}

\date{}
\begin{document}

\begin{abstract}
    This paper aims to study Ext-groups between certain functors defined on the category of finitely generated free groups.
    Rational Ext-groups between the abelianization functor and its symmetric powers are known, and are almost always equal to zero. 
    Recently, using homotopical methods, Arone constructed an explicit bounded complex whose cohomology corresponds to the integral Ext-groups between the abelianization functor and its symmetric powers. The cohomology of this complex is far from being trivial.
    Using this complex, we explicitly calculate some of these Ext-groups. More precisely, we compute $ \mathrm{Ext}^{1}, \mathrm{Ext}^{2},  \mathrm{Ext}^{d-1}$ and  $\mathrm{Ext}^{d-2}$ between the abelianization functor and its $d$-th symmetric power.  We further explain how Arone's complex can be obtained from an explicit projective resolution of the abelianization functor.
 We compare our results with the computation of Ext-groups between functors from finitely generated free {\it abelian} groups, obtained by Franjou and Pirashvili. In particular, we obtain that the composition with the abelianization functor induces an isomorphism for the $\mathrm{Ext}^{1}$ considered in this paper.
  
\end{abstract}

\maketitle
Functors on the category of free groups arise  in various contexts, including, stable homology of automorphism groups of free groups (see \cite{DV2015, D2019, KawazumiVespa}), higher Hochschild homology (see \cite{PV}), compactly supported rational cohomology of configuration spaces of
points on wedges of spheres (see \cite{GH}) and the study of Jacobi diagrams in handlebodies (see \cite{KatadaI, KatadaII, V2025}).
Calculating Ext-groups between functors on free groups is an interesting problem motivated in part by the computation of the stable homology of automorphism groups of free groups.

Let $\mathsf{gr}$ be the category of finitely-generated free groups, $\mathsf{Ab}$ be the category of abelian groups, $\f(\gr; \mathds{Z})$ be the category of functors from $\gr$ to $\mathsf{Ab}$, $\mathfrak{a} \in \f(\gr; \mathds{Z})$ be the abelianization functor and $T^c$ be the $c$th tensor power functor for abelian groups. In \cite{vespa2018extensions}, the second author computes $ \mathrm{Ext}^{*}_{\mathcal{F}(\mathsf{gr};\mathds{Z})} ( T^c \circ \mathfrak{a}, T^{d} \circ \mathfrak{a} )$ and deduces, from this result, several computations of rational Ext-groups between functors such as $S^d \circ \mathfrak{a}$ and $\Lambda^d \circ \mathfrak{a}$, where $S^d$ is the $d$-th symmetric power and $\Lambda^d$ the $d$-th exterior power.  In particular, she obtains that 
\begin{align} \label{202508111534}
   \mathrm{Ext}^{i}_{\mathcal{F}(\mathsf{gr};\mathds{Z})} (\mathfrak{a}, S^{d} \circ \mathfrak{a} )\otimes \mathds{Q}= \begin{cases}
            \mathds{Q} & \mathrm{~if~} d=1 \mathrm{~and~} i=0
             , \\
            0  & \mathrm{otherwise}.
        \end{cases} 
\end{align}

In \cite{arone2025polynomial}, Arone uses homotopical methods in order to compute integral Ext-groups between functors on $\gr$. In particular, he gives an explicit bounded complex whose cohomology corresponds to $ \mathrm{Ext}^{*}_{\mathcal{F}(\mathsf{gr};\mathds{Z})} (\mathfrak{a}, S^{d} \circ \mathfrak{a} )$. The following table, given in \cite{arone2025polynomial}, presents some examples of $\mathrm{Ext}^{i}_{\mathcal{F}(\mathsf{gr};\mathds{Z})}(\mathfrak{a}, S^{d}\circ \mathfrak{a})$, obtained in part with the help of a computer. 

\begin{table}[h]
\[\scalebox{0.88}{
\begin{tabular}{|c|c|c|c|c|c|c|c|c|c|c|c|}
\hline
$d ~\backslash~ i$ & 0 & 1 & 2 & 3 & 4 & 5 & 6 & 7 & 8 &$\cdots$ \\
\hline
1 & $\mathds{Z}$ & 0 & 0 & 0 & 0 & 0 & 0 & 0&0& $\cdots$ \\
\hline
2 & 0 & \cellcolor{cellGreen}$\mathds{Z}/2$ & 0 & 0 & 0 & 0 & 0  & 0&0&$\cdots$\\
\hline
3 & 0 & \cellcolor{cellLightgreen}$\mathds{Z}/3$ & \cellcolor{cellLightPurple}$\mathds{Z}/2$ & 0 & 0 & 0 & 0 &0 & 0 & $\cdots$ \\
\hline
4 & 0 & \cellcolor{cellYellowRemaining}$\mathds{Z}/2$ & \cellcolor{cellReddishPink}$\mathds{Z}/3$ & \cellcolor{cellBlue}$\mathds{Z}/2$ & 0 & 0 & 0 & 0 & 0 &$\cdots$ \\ 
\hline
5 & 0 & \cellcolor{cellYellowRemaining}$\mathds{Z}/5$ & \cellcolor{cellRed}$\mathds{Z}/2$ & \cellcolor{cellBBlue}0 & \cellcolor{cellBlue}$\mathds{Z}/2$ & 0 & 0 & 0 & 0 &$\cdots$ \\
\hline
6 & 0 & \cellcolor{cellYellowRemaining}0 & \cellcolor{cellRed}$\mathds{Z}/10$ & $\mathds{Z}/6$ & \cellcolor{cellBBlue}0 & \cellcolor{cellBlue}$\mathds{Z}/2$ & 0  & 0 & 0 &$\cdots$\\
\hline
7 & 0 & \cellcolor{cellYellowRemaining}$\mathds{Z}/7$ & \cellcolor{cellRed}0 & $\mathds{Z}/2$ & $\mathds{Z}/6$ & \cellcolor{cellBBlue}0 & \cellcolor{cellBlue}$\mathds{Z}/2$  & 0 & 0 &$\cdots$ \\
\hline
8 & 0 & \cellcolor{cellYellowRemaining}$\mathds{Z}/2$ & \cellcolor{cellRed}$\mathds{Z}/7$ & $\mathds{Z}/2$ & $\mathds{Z}/2$ & $\mathds{Z}/2$ & \cellcolor{cellBBlue}0 & \cellcolor{cellBlue}$\mathds{Z}/2$& 0 & $\cdots$ \\
\hline
9 & 0 & \cellcolor{cellYellowRemaining}$\mathds{Z}/3$ & \cellcolor{cellRed}$\mathds{Z}/2$ & 0 & $\mathds{Z}/2$ & $\mathds{Z}/6$ & $\mathds{Z}/2$ & 0\cellcolor{cellBBlue}& $\mathds{Z}/2$\cellcolor{cellBlue}& $\cdots$ \\
\hline
$\vdots$ & $\vdots$ & \cellcolor{cellYellowRemaining}$\vdots$ & \cellcolor{cellRed}$\vdots$ & $\vdots$ & $\vdots$ & $\vdots$ & $\vdots$  & $\vdots$ &\cellcolor{cellBBlue}$\ddots$& \cellcolor{cellBlue}$\ddots$ \\
\hline
\end{tabular}
}
\]
\caption { $\mathrm{Ext}^{i}_{\mathcal{F}(\mathsf{gr};\mathds{Z})} ( \mathfrak{a}, S^{d} \circ \mathfrak{a} )$} \label{20250801529}
\end{table}

Comparing this table with (\ref{202508111534}) shows that the integral $ \mathrm{Ext}$-groups between $\mathfrak{a}$ and  $S^{d}\circ \mathfrak{a}$ are far more intricate than their rational counterparts. The situation over $\mathds{Z}$ retains additional torsion phenomena making the integral case significantly more complex.

One of the primary motivation of this work was to fully understand the pattern of this table.

In this paper, we give a complete computation of the Ext-groups corresponding to the colored cells in the table. More precisely, denoting by $\nat$ the set of non-negative integers, $\nat^*$ the set of positive integers and
$\prim$ the set of prime numbers, we obtain the following results:

\begin{THM}[Theorem \ref{202507171119}] \label{202508081422}
    For $d$ a positive integer, we have:
    \begin{align*}
  \mathrm{Ext}^{1}_{\mathcal{F}(\mathsf{gr};\mathds{Z})} ( \mathfrak{a}, S^{d} \circ \mathfrak{a} )= 
        \begin{cases}
            \mathds{Z}/ p & \mathrm{~if~} d = p^l \mathrm{~for~} p \in \prim \mathrm{~and~}  l \in \nat^*
             , \\
            0  & \mathrm{otherwise}.
        \end{cases}
    \end{align*}
\end{THM}

\begin{THM}[Theorem \ref{202507041404}] \label{202508081440}
For $d$ a positive integer, we have:
    \begin{align*}
    \mathrm{Ext}^{2}_{\mathcal{F}(\mathsf{gr};\mathds{Z})} ( \mathfrak{a}, S^{d} \circ \mathfrak{a})= \bigoplus_{J_d} \mathds{Z}/p
    \end{align*}
    where $J_d =\{p \in \prim \mid \exists n \in \nat \exists m \in \nat^*
 , d=p^n (p^m+1) \}$.
\end{THM}

\begin{THM}[Theorem \ref{202507311555K}, Theorem \ref{202507111139}] \label{202508141030}
    Let $d$ be an integer such that $d \geq 2$. We have:
    \begin{align*}
        \mathrm{Ext}^{d-1}_{\mathcal{F}(\mathsf{gr};\mathds{Z})} ( \mathfrak{a}, S^{d} \circ \mathfrak{a}) = \mathds{Z}/2,
    \end{align*}
    \begin{align*}
        \mathrm{Ext}^{d-2}_{\mathcal{F}(\mathsf{gr};\mathds{Z})} ( \mathfrak{a}, S^{d} \circ \mathfrak{a}) =
        \begin{cases}
            \mathds{Z}/3 & d \in \{3,4 \}; \\
            0 & \mathrm{otherwise} .
        \end{cases}
    \end{align*}
    
\end{THM}

Let $\mathsf{ab}$ be the category of finitely-generated free abelian groups, $\f(\mathsf{ab}; \mathds{Z})$ be the category of functors from $\mathsf{ab}$ to the category of abelian groups and $I:\mathsf{ab}\to \mathsf{Ab} $ the inclusion functor. 

Comparing Theorem \ref{202508081422} with $\mathrm{Ext}^{i}_{\mathcal{F}(\mathsf{ab};\mathds{Z})} ( I , S^{d} )$ computed in \cite{Franjou-Pira}, we obtain the following corollary:

\begin{COR}[Corollary \ref{202508041451}] \label{202508081423}
   The functor $- \circ \mathfrak{a}:\mathcal{F}(\mathsf{ab};\mathds{Z})\to \mathcal{F}(\mathsf{gr};\mathds{Z}) $ induces an isomorphism:
  $$\mathrm{Ext}^{1}_{\mathcal{F}(\mathsf{ab};\mathds{Z})} ( I , S^{d} )\xrightarrow{\cong}\mathrm{Ext}^{1}_{\mathcal{F}(\mathsf{gr};\mathds{Z})} ( \mathfrak{a} , S^{d} \circ \mathfrak{a} ).$$ 
\end{COR}

For $d = p^{l}$, we give, in Proposition \ref{202508041518}, an explicit extension of $I$ by $S^d$ which generates $\mathrm{Ext}^{1}_{\mathcal{F}(\mathsf{ab};\mathds{Z})} ( I , S^{d} ) \cong \mathds{Z}/ p$.
The extension, which we denote by $E_d$, is constructed by using a refinement of the generator-relation description of $\mathsf{gr}$ given by \cite{Pira}.
Furthermore, for $p=2$, the extension $E_d$ evaluated at $n \in \nat$ induces a nonsplit extension of modules over the integral general linear group of degree $n$.
This is explained in Remark \ref{202508281012k}.

Note that the analogue of Corollary \ref{202508081423} is not true for higher Ext-groups, as explained in the beginning of Appendix \ref{202508141211k}. It would be interesting to have a conceptual explanation of the isomorphism obtained in Corollary \ref{202508081423}. 

The main results are proved by computing the cohomology of an explicit cochain complex, denoted by $\AC$. This complex is isomorphic to the complex introduced   first by Arone in \cite{arone2025polynomial} and satisfies $\mathrm{Ext}^{i}_{\mathcal{F}(\mathsf{gr};\mathds{Z})} ( \mathfrak{a} , S^{d} \circ \mathfrak{a} ) \cong H^i(\AC)$. 
In this paper, we define $\AC$ via an explicit projective resolution of $\mathfrak{a}$ which was previously used by the second author in \cite{DPV,KawazumiVespa}.
This resolution, given in (\ref{202507301119}), is induced by the \textit{normalized} bar resolution of a free group. By Proposition \ref{202508271137}, the cochain complex $\AC$ is bounded since the functor $S^{d} \circ \mathfrak{a}$ is polynomial. 
In Theorems \ref{202507171119}, \ref{202507041404}, \ref{202507111139} and \ref{202507311555K}, we also give explicit cocycles of $\AC$ generating the corresponding Ext-groups.

In this paper, we go further by computing $\mathrm{Ext}^{i}_{\mathcal{F}(\mathsf{gr};\mathds{Z})} ( T^c \circ \mathfrak{a} , S^{d} \circ \mathfrak{a} )$ by applying a K\"unneth formula to the main results.

To better illustrate our method, we also give the computation of $\mathrm{Ext}^{i}_{\mathcal{F}(\mathsf{gr};\mathds{Z})} (T^c \circ \mathfrak{a}, \Lambda^{d} \circ \mathfrak{a} )$. 
This calculation is much easier than that of $\mathrm{Ext}^{i}_{\mathcal{F}(\mathsf{gr};\mathds{Z})} ( T^c \circ \mathfrak{a} , S^{d} \circ \mathfrak{a} )$.
In \cite[Thm 4.2]{vespa2018extensions}, the second author computes $\mathrm{Ext}^{i}_{\mathcal{F}(\mathsf{gr};\mathds{Z})} (T^c \circ \mathfrak{a}, \Lambda^{d} \circ \mathfrak{a} )\otimes \mathds{Q}$.
In \cite[Corollary 11.4 (2)]{arone2025polynomial}, Arone extends this result to integer coefficients and proves that
\begin{align*}
        \mathrm{Ext}^{i}_{\mathcal{F}(\mathsf{gr};\mathds{Z})} ( T^c \circ \mathfrak{a}, \Lambda^d \circ \mathfrak{a}) =
        \begin{cases}
            \mathds{Z}[R(d,c)]& i =d-c,  \\
            0 & \mathrm{otherwise} ,
        \end{cases}
    \end{align*}
where $R(d,c)$ is the set of ordered partitions of $d$ into $c$ parts.
In Appendix \ref{202508151448}, we give another proof of this result by using the projective resolution of the functor $\mathfrak{a}$ and a K\"unneth formula. 

To prove Theorem \ref{202508081422} and Corollary \ref{202508081423}, we start with explicitly computing $1$-coboundaries and $1$-cocycles. This leads to $H^1(\AC)=\mathds{Z}/g$ where $g$ is the greatest common divisor of the binomial coefficients $\binom{d}{k}$ for $1\leq k\leq d-1$. We deduce Theorem \ref{202508081422} using classical arithmetic arguments based on Kummer's theorem, recalled in Appendix \ref{202507071237}.

To obtain Theorem \ref{202508081440}, we need to compute $H^2(\AC)$. This computation represents the most technically demanding part of this paper and requires the most substantial effort. To compute  $H^2(\AC)$, we apply the universal coefficient theorem. Since $H^2 ( \AC)$ is a finite group, we can detect the $p$-torsion part of $\mathrm{Ext}^2 ( \mathfrak{a}, S^{d} \circ \mathfrak{a})$ by computing $H^1(\AC \otimes \mathds{Z}/p^{N})$ for all $N\in \nat^*$. To compute $H^1(\AC \otimes \mathds{Z}/p^{N})$, we introduce, in section \ref{202507071356}, the auxiliary complex $\ACPN$ whose cohomology is isomorphic to the cohomology of $\AC\otimes \mathds{Z}/p^{N}$. The differentials of $\ACPN$ closely resemble those of $\AC$, with binomial coefficients replaced by their $p$-factors.
We compute $H^1(\ACPN)$ using some classical arithmetic results given in Appendix \ref{202507071237}.

Theorem \ref{202508141030} follows by a direct computation of the corresponding coboundaries and cocycles.

\subsection{Organisation of the paper}
In section \ref{202508141214k}, we present some recollection on functors on free groups. In Section \ref{202507071042}, we introduce the cochain complex $\AC$. In Section \ref{202508081420}, we prove Theorem \ref{202508081422} and Corollary \ref{202508081423}. Sections \ref{202508081426}, \ref{202507041421} and \ref{202507041422} are devoted to the proof of Theorem \ref{202508081440}. In Section \ref{202507041421} we compute $H^1(\AC \otimes \mathds{Z}/p^{N})$ for all $N\in \nat^*$ and Theorem \ref{202508081440} is proved in Section \ref{202507041422}. Section \ref{202508081503} concerns the computation of $\mathrm{Ext}^{d-1}_{\mathcal{F}(\mathsf{gr};\mathds{Z})} ( \mathfrak{a}, S^{d} \circ \mathfrak{a} )$ and $\mathrm{Ext}^{d-2}_{\mathcal{F}(\mathsf{gr};\mathds{Z})} ( \mathfrak{a}, S^{d} \circ \mathfrak{a} )$. In Section \ref{202508081128}, we explain how to obtain $\mathrm{Ext}^{i}_{\mathcal{F}(\mathsf{gr};\mathds{Z})} ( T^c \circ \mathfrak{a} , S^{d} \circ \mathfrak{a} )$ from the calculation of $\mathrm{Ext}^{i}_{\mathcal{F}(\mathsf{gr};\mathds{Z})} ( \mathfrak{a} , S^{d} \circ \mathfrak{a} )$, and we give some examples. In Appendix \ref{202507071237}, we recall some arithmetic results on binomial coefficients that will be useful in this paper. Appendix \ref{202508141211k} concerns $\mathrm{Ext}^{i}_{\mathcal{F}(\mathsf{ab};\mathds{Z})} ( I, S^{d} )$. Appendix \ref{202508151448} concerns $\mathrm{Ext}^{*}_{\mathcal{F}(\mathsf{gr};\mathds{Z})} (T^c \circ \mathfrak{a}, \Lambda^d \circ \mathfrak{a})$.

\subsection{Notation and Conventions}

\begin{enumerate}
\item
$\nat$ denotes the set of non-negative integers and $\nat^*$ the positive integers;
\item 
$\prim$ denotes the set of prime numbers.
\end{enumerate}

\section{Recollections on functors on free groups} 
\label{202508141214k}

In this section, we give an overview of the notations and results on functors on free groups used in this paper. For more details, we refer the reader to \cite{DPV}, for example.

\subsection{Generalities on functors on $\gr$} \label{202508211615}

Let $\mathsf{gr}$ be the category of finitely-generated free groups and group homomorphisms. This category is essentially small, with skeleton labeled by $\nat$, where $n\in \nat$ corresponds to the free group $\mathds{Z}^{\ast n}$ of rank $n$. For clarity, we will sometimes denote the object $n$ by $\mathds{Z}^{\ast n}$. The object $0=\{1\}$ is a null-object in $\mathsf{gr}$. 

We denote by $\f(\gr; \mathds{Z})$  the category of functors from $\gr$ to $\mathsf{Ab}$, the category of abelian groups. This category is abelian.

A functor $M: \gr \to \mathsf{Ab}$ is said to be \textit{reduced} if $M(0)=0$. Every functor $F: \gr \to \mathsf{Ab}$ can be decomposed naturally as a direct sum $F=F(0) \oplus \overline{F}$ where $F(0)$ is the constant functor equal to $F(0)$ and $\overline{F}$ is a reduced functor.

Let $P_n: \gr \to \mathsf{Ab}$ be the functor $\mathds{Z}[\gr(n,-)]$; $\{P_n, n \in \mathbb{N}\}$ is a set of projective generators of the category $\f(\gr; \mathds{Z})$. 
By the Yoneda lemma, for $F: \gr \to \mathsf{Ab}$, we have $\text{Hom}_{\f(\gr; \mathbb{Z})}(P_n,F)\cong F(\mathds{Z}^{\ast n})$. Note that $P_0=\mathds{Z}$ where $\mathds{Z}$ is the constant functor on $\gr$ equal to $\mathds{Z}$.

Let $\mathsf{ab}$ be the category of finitely-generated free abelian groups and group homomorphisms. Abelianization of groups $G \mapsto G/[G,G]$ induces a functor $\mathfrak{a}: \gr \to \mathsf{ab}$. Composing this functor with the inclusion functor $I: \mathsf{ab} \to \mathsf{Ab}$, we obtain a functor $\gr \to \mathsf{Ab}$. By abuse of notation we also denote this functor by $\mathfrak{a}$ and call it the abelianization functor. Note that this functor is reduced.

Since $\gr$ is a pointed category with finite coproducts, by \cite{hartl2015polynomial}, we can define the notions of cross-effects and polynomial functors in $\f(\gr; \mathds{Z})$. Recall that for $F\in \f(\gr; \mathds{Z})$, its $n$th cross-effect, for $n\in \nat^*$, is the functor $\textrm{cr}_nF: \gr^{n}\to \mathsf{Ab}$ defined by:
$$\textrm{cr}_nF(G_1, \ldots, G_n)=\textrm{Ker}\Big(F(G_1 \ast \ldots \ast G_n) \to \bigoplus_{i=1}^n F(G_1 \ast \ldots \ast \hat{G_i} \ast \ldots \ast G_n) \Big)$$
where the map is induced by the morphisms $$G_1 \ast \ldots \ast G_n \xrightarrow{\textrm{Id}_{G_1 \ast\ldots \ast G_{i-1}}\ast \epsilon_i\ast \textrm{Id}_{G_{i+1} \ast\ldots \ast G_{n}}} G_{1} \ast \ldots \ast \hat{G_i} \ast \ldots \ast G_n$$
where $\epsilon_i: G_i \to 0$.

For $F \in \f(\gr; \mathds{Z})$, using the decomposition $F=F(0) \oplus \overline{F}$ recalled above, the fact that  $\textrm{cr}_nF=\textrm{cr}_n\overline{F}$ for $n \in \nat^*$ and \cite[Prop. 3.4]{hartl2015polynomial}, we obtain the following natural isomorphisms:
\begin{align}\label{202508071131}
  F(G_1 \ast \ldots \ast G_n) \cong F(0) \oplus \overset{n}{\underset{k=1}{\bigoplus}} \ \underset{1 \leq i_1<\ldots < i_k\leq n}{\bigoplus}  \textrm{cr}_kF(G_{i_1}, \ldots, G_{i_k}).
\end{align}
In particular, for $G_1= \ldots=G_n=\mathds{Z}$, we obtain:
\begin{align}\label{202603261717}
  F(\mathds{Z}^{\ast n}) \cong F(0) \oplus \overset{n}{\underset{k=1}{\bigoplus}} \binom{n}{k} \textrm{cr}_kF(\mathds{Z}, \ldots, \mathds{Z}).
\end{align}
A functor $F \in \f(\gr; \mathds{Z})$ is said to be \textit{polynomial} of degree $n$ if $\textrm{cr}_{n+1}F=0$ and $\textrm{cr}_{n}F \neq 0$.\\
The abelianization functor $\mathfrak{a}$ is polynomial of degree $1$. 

Let $T^d: \mathsf{ab} \to  \mathsf{Ab}$ be the $d$-th tensor power functor defined on an object $G$ of $\mathsf{ab}$ by $T^d(G)=G^{\otimes d}$ and $S^d: \mathsf{ab} \to  \mathsf{Ab}$ the $d$-th symmetric power functor given on an object $G$ of $\mathsf{ab}$ by $S^d(G)=(G^{\otimes d})_{\mathfrak{S}_d}$ where $\mathfrak{S}_d$ is the symmetric group on $d$ letters.
Composing $\mathfrak{a}$ with $T^d$ or $S^d$, we obtain a polynomial functor of degree $d$.

 Denote by $\overline{P_1}$ the reduced part of $P_1$ i.e. $P_1\cong \mathds{Z} \oplus \overline{P_1}$ since $P_1(0)= \mathds{Z}$. The functors $\overline{P_1}^{\otimes d}$ are related to the cross-effects by the following proposition:
 
 \begin{propo}\label{202603261841}
     For $d\in \nat^*$ and $F \in \f(\gr; \mathds{Z})$, we have a natural isomorphism:
     \begin{align} \label{202603041014}
   \mathrm{Hom}_{\mathcal{F}(\mathsf{gr};\mathds{Z} )}(\overline{P_1}^{\otimes d} , F) \cong \text{cr}_dF(\mathds{Z},\ldots, \mathds{Z}) .
\end{align}
 \end{propo}
 \begin{proof}
     First, note that for $G \in \gr$, we have $\mathrm{Hom}_{\gr}(\mathds{Z}^{\ast d}, G)\simeq \mathrm{Hom}_{\gr}(\mathds{Z}, G)^{\times d}$. We deduce that $P_d\cong (P_1)^{\otimes d}$.
     Using the decomposition $P_1\cong \mathds{Z} \oplus \overline{P_1}$ we obtain: 
      \begin{align} \label{202603261802}
      P_d \cong (P_1)^{\otimes d} \cong (\mathds{Z} \oplus \overline{P_1})^{\otimes d} \cong  \mathds{Z} \oplus \overset{d}{\underset{k=1}{\bigoplus}} \binom{d}{k} (\overline{P_1})^{\otimes k}.
      \end{align}
     We prove the statement by induction on $d$. For $d=1$, by the Yoneda lemma, we have $\text{Hom}_{\f(\gr; \mathbb{Z})}(P_1,F)\cong F(\mathds{Z})$ and using the decomposition $P_1\cong P_0 \oplus \overline{P_1}$, we have:
     $$\text{Hom}_{\f(\gr; \mathbb{Z})}(P_1,F)\cong \text{Hom}_{\f(\gr; \mathbb{Z})}(P_0,F)\oplus \text{Hom}_{\f(\gr; \mathbb{Z})}(\overline{P_1},F) \cong F(0) \oplus \text{Hom}_{\f(\gr; \mathbb{Z})}(\overline{P_1},F).$$ So we obtain the isomorphism $F(\mathds{Z})\cong F(0) \oplus \text{Hom}_{\f(\gr; \mathbb{Z})}(\overline{P_1},F)$. Comparing this relation with (\ref{202603261717}), we deduce that $\text{Hom}_{\f(\gr; \mathbb{Z})}(\overline{P_1},F)\simeq \textrm{cr}_{1}F(\mathbb{Z})$.

     Assume that the statement holds for some $d \geq 1$. We show that the statement holds for $d+1$. By the Yoneda lemma, $\text{Hom}_{\f(\gr; \mathbb{Z})}(P_{d+1},F)\cong F(\mathds{Z}^{\ast d+1})$ and using (\ref{202603261802}), we have: 
     $$\text{Hom}_{\f(\gr; \mathbb{Z})}(P_{d+1},F)\cong F(0) \oplus \overset{d}{\underset{k=1}{\bigoplus}} \binom{d+1}{k} \text{Hom}_{\f(\gr; \mathbb{Z})}(\overline{P_1}^{\otimes k}, F) \oplus \text{Hom}_{\f(\gr; \mathbb{Z})}(\overline{P_1}^{\otimes d+1},F).$$
     Comparing with (\ref{202603261717}), and using the induction hypothesis we obtain that: $$\text{Hom}_{\f(\gr; \mathbb{Z})}(\overline{P_1}^{\otimes d+1},F) \simeq \text{cr}_{d+1}F(\mathds{Z},\ldots, \mathds{Z}).$$
     This complete the induction.
 \end{proof}

\subsection{Projective resolutions of the functor $\mathfrak{a}$ } \label{202507301126}

In this section, we recall the projective resolutions of the abelianization functor which will be used in Section \ref{202507071042}.

By \cite[Proposition 5.1]{JP} (see also \cite{DV2015,  vespa2018extensions}), the exact sequence in $\mathcal{F}(\mathsf{gr};\mathds{Z})$:
\begin{align} \label{202507291730}
 P_\bullet: \qquad    \ldots \to P_{n+1} \xrightarrow{d_n} P_n \to \ldots \to P_2 \xrightarrow{d_1} P_1
\end{align}
is a projective resolution of the abelianization functor $\mathfrak{a}: \mathsf{gr} \to \mathsf{Ab}$.
The natural transformation $d_n: P_{n+1} \to P_{n}$ is given on a group $G \in \mathsf{gr}$ by the linear map $\mathds{Z}[G^{n+1}] \to \mathds{Z}[G^n]$ such that:
$$d_n([g_1, \ldots, g_{n+1}])=[g_2, \ldots, g_{n+1}]+$$
$$\overset{n}{\underset{i=1}{\sum}}(-1)^{i}[g_1, \ldots, g_{i-1}, g_ig_{i+1}, g_{i+2}, \ldots, g_{n+1}]+(-1)^{n+1}[g_1, \ldots, g_n]$$
for all $(g_1, \ldots, g_{n+1}) \in G^{n+1}$. This resolution is obtained from the bar resolution where we truncate the part of degree $0$. Here, we employ the known computation of the group homology of the free group:  $H_i(\mathds{Z}^{*n})= \begin{cases}
            \mathds{Z}^{\oplus n} & \mathrm{if\ } i=1, \\
            0& \mathrm{if\ } i>1.
        \end{cases}$

The Ext-group $\mathrm{Ext}^*_{\mathcal{F}(\mathsf{gr};\mathds{Z} )}( \mathfrak{a}, F)$ is the cohomology of the cochain complex $\text{Hom}_{\f(\gr; \mathbb{Z})}(P_\bullet ,F)$ which is isomorphic to
\begin{align} \label{202507291735}
  F(\mathds{Z}) \xrightarrow{\delta^0} F(\mathds{Z}^{*2} ) \xrightarrow{\delta^1} \cdots \rightarrow F(\mathds{Z}^{*n}) \xrightarrow{\delta^{n-1}} F(\mathds{Z}^{*n+1}) \rightarrow \cdots  
\end{align}
with $\delta^{n-1}=F(a^n)+ \sum_{k=1}^n(-1)^{k}F(b_k^n)+(-1)^{n+1} F(c^n)$.
Here, the morphisms 
$$a^n, b_k^n, c^n : \mathds{Z}^{*n} \to \mathds{Z}^{*n+1}$$ are given by the following formulas where $\{x_1, \ldots, x_n\}$ is a free generating set of $\mathds{Z}^{*n}$:

\begin{align*}
        a^n(x_i) {=} x_{i+1};  
 \qquad
        b^n_k(x_i) {=}
        \begin{cases}
            x_i & \mathrm{if\ } i<k, \\
            x_kx_{k+1}& \mathrm{if\ } i=k ,\\
            x_{i+1} & \mathrm{if\ } i>k;
        \end{cases}
        \qquad c^n(x_i)=x_i.
    \end{align*}

Using the \textit{normalized} bar resolution, we obtain the following variant of the resolution (\ref{202507291730}) (see \cite[Section 4]{DPV}, \cite{KawazumiVespa}):
\begin{align} \label{202507301119}
    \ldots \overline{P_1}^{\otimes (n+1)} \xrightarrow{d_n} \overline{P_1}^{\otimes n} \to \ldots \to \overline{P_1}^{\otimes 2} \xrightarrow{d_1} \overline{P_1} .
\end{align}
The Ext-groups $\mathrm{Ext}^*_{\mathcal{F}(\mathsf{gr};\mathds{Z} )}( \mathfrak{a}, F)$ are isomorphic to the cohomology of the complex $\text{Hom}_{\f(\gr; \mathbb{Z})}(\overline{P_1}^{\otimes \bullet} ,F)$ which is, by Proposition \ref{202603261841}, isomorphic to the following subcomplex of (\ref{202507291735}):
\begin{align} \label{202507291740}
  \overline{F}(\mathds{Z}) \xrightarrow{\delta^0} \text{cr}_2F(\mathds{Z}, \mathds{Z}) \xrightarrow{\delta^1} \cdots \rightarrow \text{cr}_nF(\mathds{Z},\cdots, \mathds{Z}) \xrightarrow{\delta^{n-1}} \text{cr}_{n+1}F(\mathds{Z},\cdots, \mathds{Z}) \rightarrow \cdots  .
\end{align}
The advantage of this complex lies in the fact that, for $F$ a polynomial functor, this complex is bounded whereas it is not the case of the complex (\ref{202507291735}). We deduce the following proposition:
\begin{propo} \label{202508271137}
     Let $F \in \mathcal{F}(\mathsf{gr};\mathds{Z} )$ be a polynomial functor of degree $d$, then:
    $$\mathrm{Ext}^i_{\mathcal{F}(\mathsf{gr};\mathds{Z} )}( \mathfrak{a}, F)=0\quad  \textrm{for}\quad  i\geq d.$$
\end{propo}

\begin{Example}
    The functor 
$S^d \circ \mathfrak{a}$ is polynomial of degree $d$. The application of Proposition \ref{202508271137} to 
$F=S^d \circ \mathfrak{a}$ yields $\mathrm{Ext}^i_{\mathcal{F}(\mathsf{gr};\mathds{Z} )}( \mathfrak{a}, S^d \circ \mathfrak{a})=0$ for $i\geq d$.
\end{Example}

\section{Recollections on Arone's complex} \label{202507071042}

Let $d \in \nat^*$.
In \cite{arone2025polynomial}, Arone gives an explicit cochain complex over $\mathds{Z}$ computing $\mathrm{Ext}^{\bullet}_{\mathcal{F}(\mathsf{gr};\mathds{Z})} ( \mathfrak{a}, S^{d} \circ \mathfrak{a} )$.
In this section, we recover Arone's cochain complex algebraically using the projective resolution recalled in (\ref{202507301119}). More precisely, we consider the complex $\AC$ obtained from (\ref{202507291740}) for $F=S^{d} \circ \mathfrak{a}$. This complex satisfies $H^i(\AC)= \mathrm{Ext}^{i}_{\mathcal{F}(\mathsf{gr};\mathds{Z})} ( \mathfrak{a}, S^{d} \circ \mathfrak{a} )$ and is isomorphic to Arone's complex (see Proposition \ref{202507311546}).

Let $k \in \nat$.
We denote by $\Theta_{k+1}$ the intersection of the interior of the standard simplex scaled by $d$ and $\nat^{k+1}$, in other words:
$$\Theta_{k+1}=\{(\alpha_1, \ldots, \alpha_{k+1})\in \nat^{k+1} \ | \ \sum_{i=1}^{k+1}\alpha_i=d \textrm{\ and\ } \forall i ~(\alpha_i>0)  \} .$$
We can compute the cross-effects of the functor $F=S^{d} \circ \mathfrak{a}$ using the decompositions (\ref{202508071131})
to obtain:
$$cr_{k+1}(S^{d} \circ \mathfrak{a})(\mathds{Z},\ldots, \mathds{Z})=\mathds{Z}\{ e_1^{\alpha_1}e_2^{\alpha_2} \ldots e_{k+1}^{\alpha_{k+1}} | (\alpha_1, \ldots, \alpha_{k+1}) \in \Theta_{k+1}\} \cong \mathds{Z}  [\Theta_{k+1}] ,$$
where $\{e_1, \ldots, e_{k+1}\}$ is the canonical basis of the abelian group $\mathfrak{a}(k+1)$.

By the definition of the differential $\delta^{k}: cr_{k+1}(S^{d} \circ \mathfrak{a})(\mathds{Z},\ldots, \mathds{Z}) \to cr_{k+2}(S^{d} \circ \mathfrak{a})(\mathds{Z},\ldots, \mathds{Z}) $ given in Section  \ref{202507301126}, we have:
\begin{align*}
  \delta^{k}(e_1^{\alpha_1}e_2^{\alpha_2} \ldots e_{k+1}^{\alpha_{k+1}})=&e_2^{\alpha_1}e_3^{\alpha_2} \ldots e_{k+2}^{\alpha_{k+1}}+ \sum_{j=1}^{k+1}(-1)^j e_1^{\alpha_1} \ldots e_{j-1}^{\alpha_{j-1}}(e_j+e_{j+1})^{\alpha_j}e_{j+2}^{\alpha_{j+1}} \ldots e_{k+2}^{\alpha_{k+1}}\\
  +&(-1)^{k+2}e_1^{\alpha_1}e_2^{\alpha_2} \ldots e_{k+1}^{\alpha_{k+1}}  .
\end{align*}
By expanding and simplifying $(e_j+e_{j+1})^{\alpha_j}$, we obtain:
\begin{align*}
  \delta^{k}(e_1^{\alpha_1}e_2^{\alpha_2} \ldots e_{k+1}^{\alpha_{k+1}})=& \sum_{j=1}^{k+1}(-1)^j e_1^{\alpha_1} \ldots e_{j-1}^{\alpha_{j-1}}\left(  \sum_{\beta=1}^{\alpha_j-1} \binom{\alpha_j}{\beta}e_j^\beta e_{j+1}^{\alpha_j-\beta} \right)e_{j+2}^{\alpha_{j+1}} \ldots e_{k+2}^{\alpha_{k+1}} , \\
  =&\sum_{j=1}^{k+1}\sum_{\beta=1}^{\alpha_j-1} (-1)^j  \binom{\alpha_j}{\beta} e_1^{\alpha_1} \ldots e_{j-1}^{\alpha_{j-1}} e_j^\beta e_{j+1}^{\alpha_j-\beta} e_{j+2}^{\alpha_{j+1}} \ldots e_{k+2}^{\alpha_{k+1}}
\end{align*}

Let $n_i=\sum_{j=1}^{i} \alpha_j$. By the condition $\alpha_j>0$, we have:
$$0<n_1=\alpha_1<n_2< \ldots < n_{k}<n_{k+1}=d, $$
so $\Theta_{k+1}\cong \{(n_1, \ldots, n_k)\in \nat^k\ |\ 0<n_1<\ldots< n_k<d\}$ if $k >0$; and $\Theta_1 \cong \ast$ otherwise.
Therefore, the complex (\ref{202507291740}) for $F=S^{d} \circ \mathfrak{a}$ can be described in the following way, corresponding, up to a sign (see Proposition \ref{202507311546}), to the description given by Arone in \cite{arone2025polynomial}:           
for $d \in \nat^*$, we have
\begin{align*}
   \AC :\quad  \ACk^{0}  \xrightarrow{\delta^0} \ACk^{1} \xrightarrow{\delta^1}\ACk^{2} \xrightarrow{\delta^2} \cdots \to \ACk^{d-1} \xrightarrow{\delta^{d-1}} 0 \to \cdots 
\end{align*}
where $\ACk^k$ is the free abelian group of rank $\binom{d-1}{k}$ with basis $\langle n_{1}n_{2} \cdots n_{k} \rangle$ where $n_j \in \nat$ such that $0 < n_{1} < n_{2} < \cdots < n_{k} < d$.
In particular, we denote by $\langle ~ \rangle$ the generator of $\ACk^0$. In other words $\ACk^k=\mathds{Z} [ \Theta_{k+1}] \cong cr_{k+1}(S^{d} \circ \mathfrak{a})(\mathds{Z},\ldots, \mathds{Z})$.
Then the differential $\delta^{k} : \ACk^{k} \to \ACk^{k+1}$ is characterized by
\begin{align} \label{202507011523}
    \delta^{k} ( \langle n_1n_2 \cdots n_{k} \rangle ) {:=} \sum^{k+1}_{j=1} \sum^{n_{j}-1}_{m = n_{j-1} +1 } (-1)^{j} \binom{n_{j}-n_{j-1}}{m-n_{j-1}} \langle n_1 n_2 \cdots n_{j-1} m n_{j} \cdots n_{k} \rangle 
\end{align}
where $n_{0} = 0$ and $n_{k+1} = d$. 

We have the following general lemma:
\begin{Lemma} \label{202510071626}
    If $d > 1$, then for $i \in \nat$, $H^i(\AC)$ is a finite group.
\end{Lemma}
\begin{proof}
    By definition, for all $i \in \nat$, $\ACk^i$ are finitely generated free abelian groups, so $H^i(\AC)$ are finitely generated abelian groups. Moreover by \cite[Theorem 4.2]{vespa2018extensions} (see also \cite[Lemma 11.7]{arone2025polynomial}) $\mathrm{Ext}^i_{\mathcal{F}(\mathsf{gr};\mathds{Z})} ( \mathfrak{a}, S^{d} \circ \mathfrak{a})\otimes \mathbb{Q}=0$ for $d>1$ and all $i$. So $H^i(\AC) \cong \mathrm{Ext}^i_{\mathcal{F}(\mathsf{gr};\mathds{Z})} ( \mathfrak{a}, S^{d} \circ \mathfrak{a})$ are finite groups.
\end{proof}

The first component $\ACk^{1}$ is the free abelian group with basis $\{ \langle n_{1}\rangle\ |\ 0 < n_{1}< d\}$.
It can be identified with $\mathds{Z}^{d-1}$ by using the relation:
\begin{align*}
    \sum^{d-1}_{i=1} a_i \langle i \rangle = (a_1,a_2, \cdots, a_{d-1}) .
\end{align*}
We have $\delta^0( \langle~\rangle)= -\left( \binom{d}{1}, \binom{d}{2}, \cdots, \binom{d}{d-1} \right) \in \mathds{Z}^{d-1}$, and the group of $1$-coboundaries $B^1 (\AC)$ is generated by 
\begin{align*}
    \left( \binom{d}{1}, \binom{d}{2}, \cdots, \binom{d}{d-1} \right) \in \mathds{Z}^{d-1} .
\end{align*}

The description of the $1$-cocycles of $\AC$ given in the following lemma plays an important role in the rest of the paper.

\begin{Lemma} \label{202507101541}
   A $1$-cochain $\sum^{d-1}_{i=1} a_i \ \langle i\rangle \in \ACk^{1}$ lies in the group of $1$-cocycles $Z^{1} (\AC )$ if and only if $(a_1, \ldots, a_{d-1}) \in \mathds{Z}^{d-1}$ satisfies the following system of equations:
\begin{align} \label{202506301221}
    \binom{k}{r} a_k = \binom{d-r}{d-k} a_r, \quad 0<r<k<d .
\end{align}
\end{Lemma}
\begin{proof}
    By the definition of $\delta^1$ we have:
$$\delta^1(\langle i \rangle)= - \sum^{i-1}_{m=1}  \binom{i}{m} \langle m ~i \rangle +  \sum^{d-1}_{m=i+1}  \binom{d-i}{m-i} \langle i ~m \rangle.$$
    Hence, we have
    $$\delta^{1}(\sum^{d-1}_{i=1} a_i \ \langle i\rangle)=\sum^{d-1}_{i=1} a_i \left(- \sum^{i-1}_{m=1}  \binom{i}{m} \langle m ~i \rangle +  \sum^{d-1}_{m=i+1}  \binom{d-i}{m-i} \langle i ~ m \rangle \right).$$
    The coefficient of $\langle r ~ k \rangle$ for $0<r<k<d$ in $\delta^{1}(\sum^{d-1}_{i=1} a_i \ \langle i\rangle)$ is
    $$- a_k\binom{k}{r} +a_r \binom{d-r}{k-r}.$$
    Since $\{ \langle r ~ k \rangle\ | \  0<r<k<d \}$ is a basis of $\ACk^2$, we deduce the statement using the relation $\binom{d-r}{k-r}=\binom{d-r}{d-k}$.
\end{proof}

In \cite{arone2025polynomial}, Arone defines a complex $D^\bullet_d$ where the term in degree $k$ is the free abelian group with basis $\langle n_{0}n_{1} \cdots n_{k-1} \rangle$ for $n_j \in \nat$ such that $0 < n_{0} < n_{1} < \cdots < n_{k-1} < d$. We have the following result: 
\begin{propo}\label{202507311546}
   Arone's complex $D^\bullet_d$ is isomorphic to $\AC$ assigning :
$$\langle n_{0}n_{1} \cdots n_{k-1} \rangle \mapsto (-1)^k\langle n_{0}n_{1} \cdots n_{k-
1} \rangle.$$
\end{propo}

\section{On $\mathrm{Ext}^{1}_{\mathcal{F}(\mathsf{gr};\mathds{Z})} ( \mathfrak{a}, S^{d} \circ \mathfrak{a} )$} \label{202508081420}

In this section, we give a computation of $\mathrm{Ext}^{1}_{\mathcal{F}(\mathsf{gr};\mathds{Z})} ( \mathfrak{a}, S^{d} \circ \mathfrak{a} )$ with an explicit extension that generates this group.

\subsection{Computation of $\mathrm{Ext}^{1}_{\mathcal{F}(\mathsf{gr};\mathds{Z})} ( \mathfrak{a}, S^{d} \circ \mathfrak{a} )$}
    The aim of this section is to prove the following theorem extending the computations given in \cite{arone2025polynomial}, corresponding to the first column of Table \ref{20250801529}, given in the introduction.
\begin{theorem} \label{202507171119}
    For $d \in \nat^*$, we have:
    \begin{align*}
  \mathrm{Ext}^{1}_{\mathcal{F}(\mathsf{gr};\mathds{Z})} ( \mathfrak{a}, S^{d} \circ \mathfrak{a} )= H^1 ( \AC) \cong
        \begin{cases}
            \mathds{Z}/ p & \mathrm{~if~} d = p^l \mathrm{~for~} p \mathrm{~a~prime~and~}  l \in \nat^*
             , \\
            0  & \mathrm{otherwise}.
        \end{cases}
    \end{align*}

    In the first case, a generator is given by $$\left( \frac{1}{p} \binom{p^{l}}{1}, \frac{1}{p} \binom{p^{l}}{2}, \cdots, \frac{1}{p} \binom{p^{l}}{p^{l}-1} \right)\in \mathds{Z}^{p^l-1}.$$
\end{theorem}

\begin{proof}
    Recall that $H^1 ( \AC)=Z^1 (\AC)/B^1 (\AC)$, and $Z^1 (\AC)$ is the group of solutions in $\mathds{Z}^{d-1}$ of the system of equations (\ref{202506301221}) given in Lemma \ref{202507101541}.
    Let $g$ be the greatest common divisor of $\binom{d}{k}$ where $0 < k < d$.
    It is easy to show that $( \frac{1}{g} \binom{d}{1} , \cdots, \frac{1}{g} \binom{d}{d-1} )$ satisfies the equations (\ref{202506301221}), so $( \frac{1}{g} \binom{d}{1} , \cdots, \frac{1}{g} \binom{d}{d-1} ) \in Z^1 (\AC)$.
    We now prove that the group $Z^1 ( \AC)$ is generated by that element.   
    Let $(a_1, \cdots, a_{d-1}) \in Z^1 ( \AC)$. 
    It suffices to show that there exists $b \in \mathds{Z}$ such that $a_k = \frac{1}{g}\binom{d}{k}b$ for $0 < k < d$.
    Considering the equations for $r=1$, we obtain:
    \begin{align} \label{202507071130}
    ka_k = \binom{d-1}{d-k}a_1 ,\quad 1<k<d.
    \end{align}
    By the identity $d \binom{d-1}{d-k} = k \binom{d}{k}$, we obtain 
        \begin{align} \label{202506301230}
            \binom{d}{k}a_1 = da_k .
        \end{align}
    By using B\'ezout's identity, we choose $\lambda_1,\cdots,\lambda_{d-1} \in \mathds{Z}$ such that $\sum^{d-1}_{k=1} \lambda_{k} \binom{d}{k} = g$.
    Then we obtain
    $$
    ga_1 = \sum^{d-1}_{k=1} \lambda_k \binom{d}{k} a_1 = \sum^{d-1}_{k=1} \lambda_{k} d a_k = d b ,
    $$
    where $b = \sum^{d-1}_{k=1} \lambda_{k} a_k$. By (\ref{202506301230}), we see that $a_k = \frac{1}{d} \binom{d}{k} a_1 = \frac{1}{g} \binom{d}{k} b$ for $0 < k < d$.
    
    Since $B^1 ( \AC)$ is generated by $( \binom{d}{1} , \cdots, \binom{d}{d-1} )$, we obtain that $H^1 ( \AC) \cong \mathds{Z}/ g$.
    Thus, the statement follows from some arithmetic argument given in Corollary \ref{202506301226}.
\end{proof}

\subsection{Comparison with $\mathrm{Ext}^{1}_{\mathcal{F}(\mathsf{ab};\mathds{Z})} ( I , S^{d} )$}
Let $I:\mathsf{ab}\to \mathsf{Ab} $ be the inclusion functor.
In this section, we compare the group $ \mathrm{Ext}^{1}_{\mathcal{F}(\mathsf{gr};\mathds{Z})} ( \mathfrak{a}, S^{d} \circ \mathfrak{a} )$ with $\mathrm{Ext}^{1}_{\mathcal{F}(\mathsf{ab};\mathds{Z})} ( I , S^{d} )$ for  $d\in \nat^*$.

We have the following lemma:
\begin{Lemma} \label{202508111732}
    For $F,G\in\mathcal{F}(\mathsf{ab};\mathds{Z})$, the functor $- \circ \mathfrak{a}:\mathcal{F}(\mathsf{ab};\mathds{Z})\to \mathcal{F}(\mathsf{gr};\mathds{Z}) $ induces an injective map:
    $$\mathrm{Ext}^{1}_{\mathcal{F}(\mathsf{ab};\mathds{Z})} ( F , G )\hookrightarrow\mathrm{Ext}^{1}_{\mathcal{F}(\mathsf{gr};\mathds{Z})} ( F \circ \mathfrak{a} , G \circ \mathfrak{a} ).$$ 
\end{Lemma}
\begin{proof}
    Let $E$ be an extension of $F$ by $G$ in $\mathcal{F}(\mathsf{ab};\mathds{Z})$. Since the functor $- \circ \mathfrak{a}$ is exact (see \cite[Proposition 2.7]{PV}) , the image of $E$ by  $- \circ \mathfrak{a}$ is an extension of $ F \circ \mathfrak{a}$ by $G \circ \mathfrak{a}$. Assume that the image of $E$ by  $- \circ \mathfrak{a}$ is split in $\mathcal{F}(\mathsf{gr};\mathds{Z})$. Since the functor $ \mathfrak{a}$ is full and essentially surjective then $E$ is split in $\mathcal{F}(\mathsf{ab};\mathds{Z})$.
    Hence, the map in the statement is injective.
    \end{proof}

    Note that the injective map obtained in the previous lemma is not an isomorphism in general. For example, for $\Lambda^{2}$ the second exterior power, we have $\mathrm{Ext}^{1}_{\mathcal{F}(\mathsf{ab};\mathds{Z})} ( I , \Lambda^{2} )=0$ by \cite[Cor. 2.3]{Franjou-Pira} whereas $\mathrm{Ext}^{1}_{\mathcal{F}(\mathsf{gr};\mathds{Z})} ( \mathfrak{a} , \Lambda^2 \circ \mathfrak{a} )=\mathds{Z}$ by  \cite[Thm 4.2]{vespa2018extensions} (see also Proposition \ref{202508281110}).

    As a corollary of Theorem \ref{202507171119}, we obtain the following result:
\begin{Corollary} \label{202508041451}
   The functor $- \circ \mathfrak{a}:\mathcal{F}(\mathsf{ab};\mathds{Z})\to \mathcal{F}(\mathsf{gr};\mathds{Z}) $ induces an isomorphism:
  $$\mathrm{Ext}^{1}_{\mathcal{F}(\mathsf{ab};\mathds{Z})} ( I , S^{d} )\xrightarrow{\cong}\mathrm{Ext}^{1}_{\mathcal{F}(\mathsf{gr};\mathds{Z})} ( \mathfrak{a} , S^{d} \circ \mathfrak{a} ).$$ 
  In particular, for $d = p^{l}$, the extension in $\mathcal{F}(\mathsf{gr};\mathds{Z})$:
\begin{align*}
    0 \to S^{d} \circ \mathfrak{a} \to E_{d}\circ \mathfrak{a} \to \mathfrak{a} \to 0
\end{align*}
where $E_d$ is the functor given in Proposition  \ref{202508111122k}, is a generator of   $\mathrm{Ext}^{1}_{\mathcal{F}(\mathsf{gr};\mathds{Z})} (\mathfrak{a} , S^{d} \circ \mathfrak{a}) \cong \mathds{Z}/ p$.
\end{Corollary}
\begin{proof}
    By Lemma  \ref{202508111732}, we have an injective map:
   $$\mathrm{Ext}^{1}_{\mathcal{F}(\mathsf{ab};\mathds{Z})} ( I , S^{d} )\hookrightarrow \mathrm{Ext}^{1}_{\mathcal{F}(\mathsf{gr};\mathds{Z})} ( \mathfrak{a} , S^{d} \circ \mathfrak{a} ).$$
    By Theorem \ref{202507171119} and Theorem \ref{202508111718k}, for $d=p^l$, this can be regarded as an injective map between $\mathds{Z}/p$, so that it gives an isomorphism.
\end{proof}

\begin{remark}
    Note that the analogue of Corollary \ref{202508041451} is not true for higher Ext-groups. To see this, we recommend the reader to compare Table \ref{20250801529}, given in the introduction, and Table \ref{20250801525}, given in Appendix \ref{202508141211k}.
\end{remark}

\section{Computation of $\mathrm{Ext}^{2}_{\mathcal{F}(\mathsf{gr};\mathds{Z})} ( \mathfrak{a}, S^{d} \circ \mathfrak{a} )$} \label{202508081426}

This section concerns the following theorem extending the computations given in \cite{arone2025polynomial} Table 2 (second column):

\begin{theorem}\label{202507041404}
For $d \in \nat^*$, we have:
    \begin{align*}
    \mathrm{Ext}^{2}_{\mathcal{F}(\mathsf{gr};\mathds{Z})} ( \mathfrak{a}, S^{d} \circ \mathfrak{a})= H^2 ( \AC) \cong \bigoplus_{J_d} \mathds{Z}/p .
    \end{align*}
    where $J_d =\{p \in \prim \mid \exists n \in \nat \exists m \in \nat^*
 , d=p^n (p^m+1) \}$.
 Furthermore, if $d= p^n(p^m +1)$ for $p \in J_d$,  $n \in \nat$ and $m \in \nat^*$, then the following 2-cochain gives a generator of the $p$-torsion part:
 $$
 -\sum^{p^n-1}_{k=1} \frac{1}{p}\binom{p^n}{k} \langle k ~p^n \rangle + \sum^{p^{n+m}-1}_{l=1} \frac{1}{p} \binom{p^{n+m}}{l} \langle p^n ~ (p^n+l) \rangle \in \ACk^{2} .
 $$
\end{theorem}
Note that the set $J_d$ may contain more than one prime. For example, $J_{12} = \{ 2,3, 11 \}$ by $12=2^2(2+1)=3(3+1)=11+1$; and $J_{30} = \{ 3,5,29\}$ by $30=3(3^2+1)=5(5+1)=29+1$. We give in the following table the computation of $H^2 ( \AC)$ for $d\leq 36.$

\begin{minipage}{0.3\textwidth}
\footnotesize
\begin{tabular}{|c|c|}
\hline
$d$ & $H^2 ( \AC)$ \\
\hline
1 & $0$ \\
2 & $0$ \\
3 & $\mathds{Z}/2$ \\
4 & $\mathds{Z}/3$ \\
5 & $\mathds{Z}/2$ \\
6 & $\mathds{Z}/2 \oplus \mathds{Z}/5$ \\
7 & $0$ \\
8 & $\mathds{Z}/7$ \\
9 & $\mathds{Z}/2$ \\
10 & $\mathds{Z}/2 \oplus \mathds{Z} /3$ \\
11 & $0$ \\
12 & $\mathds{Z}/2 \oplus \mathds{Z}/3 \oplus \mathds{Z}/11$ \\
\hline
\end{tabular}
\end{minipage}%
\hspace{1.2cm}
\begin{minipage}{0.3\textwidth}
\footnotesize
\begin{tabular}{|c|c|}
\hline
$d$ & $H^2 ( \AC)$ \\
\hline
13 & $0$ \\
14 & $\mathds{Z}/13$ \\
15 & $0$ \\
16 & $0$ \\
17 & $\mathds{Z}/2$ \\
18 & $\mathds{Z}/2 \oplus \mathds{Z}/17$ \\
19 & $0$ \\
20 & $\mathds{Z}/2 \oplus \mathds{Z}/19$ \\
21 & $0$ \\
22 & $0$ \\
23 & $0$ \\
24 & $\mathds{Z}/2 \oplus \mathds{Z}/23$ \\
\hline
\end{tabular}
\end{minipage}%
\hspace{0.1cm}
\begin{minipage}{0.3\textwidth}
\footnotesize
\begin{tabular}{|c|c|}
\hline
$d$ & $H^2 ( \AC)$ \\
\hline
25 & $0$ \\
26 & $\mathds{Z}/5$ \\
27 & $0$ \\
28 & $\mathds{Z}/3$ \\
29 & $0$ \\
30 & $\mathds{Z}/3 \oplus \mathds{Z} /5 \oplus \mathds{Z}/29$ \\
31 & $0$ \\
32 & $\mathds{Z}/31$ \\
33 & $\mathds{Z}/2$ \\
34 & $\mathds{Z}/2$ \\
35 & $0$ \\
36 & $\mathds{Z}/2 \oplus \mathds{Z}/3$ \\
\hline
\end{tabular}
\end{minipage}

\vspace{.5cm}
We present the strategy of the proof of Theorem \ref{202507041404}.
To this end, we recall that by Lemma \ref{202510071626}, $H^2 ( \AC) = \mathrm{Ext}^2 ( \mathfrak{a}, S^{d} \circ \mathfrak{a})$ is a finite group.
For $p \in \prim$, we can detect the $p$-torsion part of $H^2 ( \AC)$ by computing $\mathrm{Tor} ( H^2 (\AC), \mathds{Z}/p^{N}), ~ N \in \nat^\ast$. 
The Tor groups are closely related to $H^1(\AC \otimes \mathds{Z}/p^{N})$ via the universal coefficient theorem:
\begin{align} \label{202507171059}
    0 \to H^1(\AC) \otimes \mathds{Z}/p^{N} \to  H^1(\AC \otimes \mathds{Z}/p^{N}) \to \mathrm{Tor} ( H^2 (\AC), \mathds{Z}/p^{N}) \to 0 .
\end{align}
Note that the group $H^1(\AC)$ is computed in Theorem \ref{202507171119}.
Section \ref{202507041421} is devoted to the computation of $H^1(\AC \otimes \mathds{Z}/p^{N})$. The proof of Theorem \ref{202507041404}, based on (\ref{202507171059}), is postponed to Section \ref{202507041422}.

\section{Computation of $H^1(\AC \otimes \mathds{Z}/p^{N})$} \label{202507041421}
Using Section \ref{202507071042}, for $p\in \prim$ and $(d,N) \in (\nat^*)^2$, we have:

\begin{align} \label{202507071158}
   \AC\otimes \mathds{Z}/p^{N} :\quad  \ACk^{0}\otimes \mathds{Z}/p^{N}  \to \ACk^{1} \otimes \mathds{Z}/p^{N}  \to \cdots \to \ACk^{d-1} \otimes \mathds{Z}/p^{N} \to  0 \to \cdots 
\end{align}
where the differentials are given by $\delta^k \otimes \mathrm{id}_{\mathds{Z}/p^N}$.
For $k \in \nat$, $(\AC\otimes \mathds{Z}/p^{N})^k$ is the free $\mathds{Z}/p^{N}$-module generated by $\langle n_{1}n_{2} \cdots n_{k} \rangle$ where $n_j \in \nat$ such that $0 < n_{1} < n_{2} < \cdots < n_{k} < d$.

Analogously to Lemma \ref{202507101541}, we can prove that  a $1$-cochain $\sum^{d-1}_{i=1} a_i \ \langle i\rangle \in \ACk^1 \otimes \mathds{Z}/p^{N}$,  lies in the group of $1$-cocycles $Z^{1} (\AC \otimes \mathds{Z}/p^{N} )$ if and only if $(a_1, \ldots, a_{d-1}) \in (\mathds{Z}/p^{N})^{d-1}$ satisfies the following system of equations:
\begin{align}
    \binom{k}{r} a_k = \binom{d-r}{d-k} a_r, ~ 0<r<k<d .
\end{align}
So the group of $1$-cocycles $Z^{1} (\AC\otimes \mathds{Z}/p^{N} )$ is the group of solutions of this system in $(\mathds{Z}/p^{N})^{d-1}$.

The authors of this paper do not have a direct method for computing the solutions of this system. In order to compute the cohomology of the complex (\ref{202507071158}), we introduce the auxiliary  complex $\ACPN$ in Section \ref{202507071356}. To define this complex, we consider an automorphism of $ \ACk^k  \otimes \mathds{Z}/p^{N}$ obtained using Theorem \ref{202507021304}. So the cohomology of $\ACPN$ is isomorphic to the cohomology of $\AC\otimes \mathds{Z}/p^{N}$.
The cohomology of $\ACPN$ is simpler to compute, since we can use the arithmetic result, Proposition \ref{202507081602}, given in Appendix \ref{202507071237}.

\subsection{The cochain complex $\ACPN$} \label{202507071356}
For $p\in \prim$ and $(d,N)\in (\nat^*)^2$, we will define the complex $\ACPN$.
The differentials of $\ACPN$ closely resemble those of $\AC$, with binomial coefficients replaced by their $p$-factors
(see Proposition \ref{202507041001}).
We define the complex $\ACPN$ by using the following isomorphism:
\begin{Defn}  \label{202507132227}   
   For $k\in \nat$, we consider the map $\varphi^{k} : \ACk^k  \otimes \mathds{Z}/p^{N} \to \ACk^k  \otimes \mathds{Z}/p^{N} $  defined by
    \begin{align*}
        \varphi^{k} ( \langle n_1 n_2 \cdots n_k \rangle ) {:=}  \left( \prod^{k+1}_{i=1}  \tilde{F}_{p,N}(n_i- n_{i-1} ) \right) \cdot \langle n_1 n_2 \cdots n_k\rangle
    \end{align*}
    where we put $n_0 = 0$ and $n_{k+1} = d$; and $\tilde{F}_{p,N}(n_i- n_{i-1} )$ is defined in Definition \ref{202507040934}.
\end{Defn}

Since $\tilde{F}_{p,N}(n_i- n_{i-1} )$ is coprime to $p$, $\varphi^{k}$ is an isomorphism.
\begin{Defn} 
\label{202603031045}
    We define the cochain complex $\ACPN$ whose $k$-th component is given by $\ACPNk^{k} {:=} \ACk^k  \otimes \mathds{Z}/p^{N}$.
    The $k$-the differential is defined by
    $$\delta^{k}_{p} {:=} \varphi^{k+1} \circ ( \delta^{k} \otimes \mathrm{id}_{\mathds{Z}/p^{N}} ) \circ (\varphi^{k})^{-1} .$$
\end{Defn}
By definition of $\ACPN$, we have the following commutative diagram:
$$
\begin{tikzcd}
    \ACPNk^{k-1} \ar[r, "\delta^{k-1}_{p}"] & \ACPNk^{k} \ar[r, "\delta^{k}_{p}"] & \ACPNk^{k+1} \\
    \ACk^{k-1} \otimes \mathds{Z}/p^{N} \ar[r, "\delta^{k-1}\otimes\mathrm{id}_{\mathds{Z}/p^N}"'] \ar[u, "\varphi^{k-1}"] & \ACk^{k} \otimes \mathds{Z}/p^{N}  \ar[r, "\delta^{k}\otimes\mathrm{id}_{\mathds{Z}/p^N}"'] \ar[u, "\varphi^{k}"] & \ACk^{k+1} \otimes \mathds{Z}/p^{N} \ar[u, "\varphi^{k+1}"]
\end{tikzcd}
$$
In the following proposition, we give an explicit description of the differential in $\ACPN$.
\begin{propo} \label{202507041001}
    The differential of $\ACPN$ is computed as follows:
    \begin{align*}
        \delta^{k}_{p} ( \langle n_1n_2 \cdots n_{k} \rangle ) = \sum^{k+1}_{j=1} \sum^{n_{j}-1}_{m = n_{j-1} +1 } (-1)^{j} \binom{n_{j}-n_{j-1}}{m-n_{j-1}}_{p} \langle n_1 n_2 \cdots n_{j-1} m n_{j} \cdots n_{k} \rangle 
    \end{align*}
    where $\binom{n_{j}-n_{j-1}}{m-n_{j-1}}_{p}$ is the $p$-factor of the binomial coefficient $\binom{n_{j}-n_{j-1}}{m-n_{j-1}}$ defined in Definition \ref{202507040924}.
\end{propo}
\begin{proof}
By (\ref{202507011523}), Definitions \ref{202507132227} and \ref{202603031045}, we have 
$$\delta^{k}_{p} ( \langle n_1n_2 \cdots n_{k} \rangle )=\sum^{k+1}_{j=1} \sum^{n_{j}-1}_{m = n_{j-1} +1 } (-1)^{j}\  C_{m,j} \ \binom{n_{j}-n_{j-1}}{m-n_{j-1}}  \langle n_1 \cdots n_{j-1} m n_{j} \cdots n_{k} \rangle $$
where 
\begin{align*}
  C_{m,j}&= \dfrac{\left( \prod^{j-1}_{i=1}\tilde{F}_{p,N}(n_i- n_{i-1} ) \right) \tilde{F}_{p,N}(m- n_{j-1} )  
\tilde{F}_{p,N}(n_j- m ) \left( \prod^{k+1}_{i=j+1}  \tilde{F}_{p,N}(n_i- n_{i-1} ) \right) }{\prod^{k+1}_{i=1}  \tilde{F}_{p,N}(n_i- n_{i-1} )} , \\
&=\dfrac{\tilde{F}_{p,N}(m- n_{j-1} )  
\tilde{F}_{p,N}(n_j- m )}{\tilde{F}_{p,N}(n_j- n_{j-1} ) }.
\end{align*}
We deduce the statement using Theorem \ref{202507021304}.
\end{proof}

By definition, we have a cochain complex isomorphism $\varphi^\bullet : \AC \otimes \mathds{Z}/p^{N} \stackrel{\cong}{\to} \ACPN$. In particular, we have 
\begin{align} \label{202507171124}
H^\bullet (\AC \otimes \mathds{Z}/p^{N} ) \cong H^\bullet (\ACPN) .
\end{align}
Hence, we can replace the computation of $H^1 (\AC \otimes \mathds{Z}/p^{N} )$ with that of $H^1 (\ACPN)$.
In fact, the coefficients of the differentials in the complex $\ACPN$, computed in  Proposition \ref{202507041001}, are more tractable using elementary arithmetic lemmas.

\subsection{Computation of $H^1 (\ACPN)$}
To state the theorem that computes $H^1 (\ACPN)$, we first introduce the following notation.
\begin{notation}
    \label{202507141154}
    For $p \in \prim$, we denote
    \begin{align*}
        A(p) &{:=} \{ p^{n} ( p^{m} + 1) \in \nat^* \mid n \in \nat \mathrm{~and~} m \in \nat^* \} , \\
        B(p) &{:=} \{ p^{n} \in \nat^* \mid n \in \nat^*\}
    \end{align*}
\end{notation}

\begin{theorem} \label{202507091101}
    For $d, N \in \nat^\ast$ and $p \in \prim$, we have
    \begin{align*}
        H^1 ( \ACPN) \cong 
        \begin{cases}
            \mathds{Z}/p  & d \in A(p) \amalg B(p) , \\
            0 & \mathrm{otherwise} .
        \end{cases}
    \end{align*}
    Furthermore, for $n = v_{p} (d)$ the $p$-adic valuation, we have a generator as follows.
    \begin{itemize}
        \item If $d \in A(p)$, then $( 0, \cdots, 0, a_{p^n}= p^{N-1} , 0, \cdots, 0)$, where $p^{N-1}$ is in the $p^n$-th coordinate, gives a generator.
        \item If $d \in B(p)$, then $(\frac{1}{p} \binom{p^n}{1}_{p}, \frac{1}{p} \binom{p^n}{2}_{p}, \cdots, \frac{1}{p} \binom{p^n}{p^n-1}_{p})$ gives a generator.
    \end{itemize}
\end{theorem}

The proof of this theorem proceeds in two steps.
In Section \ref{202507141041} we study the case where $d$ is coprime to $p$.
Next, when $d$ is divided by $p$, we compare in Section \ref{202507212309} the cohomologies associated with $d$ and $d/p$, thereby reducing the computation to the first case.

\subsubsection{1-cocycles and 1-coboundaries of $\ACPN$}

As explained at the beginning of Section \ref{202507041421},
a $1$-cochain $\sum^{d-1}_{i=1} a_i \ \langle i\rangle \in \ACPNk^{1}$,  lies in the group of $1$-cocycles $Z^{1} (\ACPN )$ if and only if $(a_1, \ldots, a_{d-1}) \in (\mathds{Z}/p^{N})^{d-1}$ satisfies the following system of equations:
\begin{align} 
\label{202507081101}
    \binom{k}{r}_{p} a_k = \binom{d-r}{d-k}_{p} a_r, ~ 0<r<k<d .
\end{align}
In particular, if $k$ is coprime to $p$, then, by considering $r= 1$ in (\ref{202507081101}), we obtain
\begin{align} \label{202507081443}
    a_k = \binom{d-1}{d-k}_{p} a_1 ,
\end{align}
since we have $\binom{k}{1}_{p} = 1$.
It is clear that the group $B^1 (\ACPN)$ is generated by 
\begin{align*}
    \left( \binom{d}{1}_{p}, \binom{d}{2}_{p}, \cdots, \binom{d}{d-1}_{p} \right) \in (\mathds{Z} /p^{N})^{d-1} .
\end{align*}

For $d=p$, the computation of $ H^{1} ( \ACPN)$ is quite simple. It is given in the following example as a warm-up for the general case studied in the next sections. 
\begin{Example} \label{202507091151}
    Computation of $H^{1} ( \ACPN)$ for $d = p$.
    Let $(a_1, \cdots, a_{p-1}) \in Z^{1} ( \ACPN)$.
    In this case, all the numbers smaller than $p$ are coprime to $p$, so that, by (\ref{202507081443}), we have $a_k = \binom{p-1}{p-k}_{p} a_1= a_1$ for $0<k<p$.
    Therefore, $Z^{1} ( \ACPN)$ is generated by $(1,1, \ldots, 1)$. 
    
    On the other hand, by definition, $B^{1}( \ACPN)$ is generated by $(\binom{p}{1}_{p}, \binom{p}{2}_{p}, \cdots, \binom{p}{p-1}_{p} ) = (p, p, \cdots, p)$.
    Hence, we see that $H^{1} ( \ACPN ) \cong \mathds{Z}/p$ with the generator $(1,1, \cdots, 1) \in H^{1} ( \ACPN)$.
    This gives an example of Theorem \ref{202507091101}.
\end{Example}

\subsubsection{The case $p \nmid d$}
\label{202507141041}

We compute $H^{1} (\ACPN)$  in the case where $p$ does not divide $d$. The main result of this section is Theorem \ref{202507081603}. Using the results of Appendix \ref{202507071237}, we deduce Corollary \ref{202507081749}.

\begin{Lemma} \label{202506231531}
    Suppose that $p$ does not divide $d$.
    Every element of $H^{1} (\ACPN)$ has a unique representative $(b_1,\cdots, b_{d-1}) \in Z^{1}(\ACPN)$ such that
    \begin{align} \label{202603031605}
        b_{k} =
        \begin{cases}
            \binom{d-1}{d-k}_{p} b_1 & \mathrm{if~} p \nmid k , \\
            0 & \mathrm{otherwise} .
        \end{cases}
    \end{align}
\end{Lemma}
\begin{proof}
    We first prove the existence of the representative.
    By the hypothesis that $p \nmid d$, there exist $r_0, m_0 \in \nat$ such that $0<r_0<p$ and $d = pm_0 + r_0$.
    Let $(a_1,\cdots, a_{d-1}) \in Z^{1} ( \ACPN)$.
    For $0 < k < d$, we set $$b_k {:=} a_k - \binom{d}{k}_{p} a_{pm_0}.$$
    Since the coboundaries $B^1 ( \ACPN)$ are generated by $( \binom{d}{1}_{p}  , \cdots, \binom{d}{d-1}_{p} )$, it is clear that $(b_1, \cdots, b_{d-1})$ is a 1-cocycle of $\ACPN$ which gives the same cohomology class with $(a_1,\cdots, a_{d-1})$.
    Thus, it is sufficient to show that $(b_1, \cdots, b_{d-1})$ satisfies the condition (\ref{202603031605}).
    
    If $p \nmid k$, then we obtain $b_k = \binom{d-1}{d-k}_{p} b_1$ by the general assertion (\ref{202507081443}).
    We now prove $b_{pl} = 0$ for $1\leq l \leq m_0$.
    If $l < m_0$, then the 1-cocycle condition (\ref{202507081101}) on $(a_1,\cdots, a_{d-1})$ implies
    \begin{align*}
        \binom{pm_0}{pl}_{p} a_{pm_0} = \binom{d-pl}{d-pm_0}_{p} a_{pl} .
    \end{align*}
    Furthermore, this equation automatically holds for $l = m_0$.
    By arithmetic results given in Appendix \ref{202507071237}, we can compute these coefficients: by Example \ref{202506251631}, we have $\binom{d-pl}{d-pm_0}_{p} = \binom{d-pl}{r_0}_{p} = 1$; and, by Lemma \ref{202506251622}, we have $\binom{pm_0}{pl}_{p} = \binom{m_0}{l}_{p} = \binom{d}{pl}_{p}$.
    Thus, the above equation leads to $a_{pl} = \binom{d}{pl}_{p} a_{pm_0}$ which implies $b_{pl} = a_{pl} - \binom{d}{pl}_{p} a_{pm_0} = 0$.

    We now prove the uniqueness.
    Suppose that $(b^\prime_1,\cdots, b^\prime_{d-1}) \in Z^{1}(\ACPN)$ is another representative for the cohomology class of $(a_1, \cdots, a_{d-1})$ which satisfies (\ref{202603031605}).
    Then there exists $\lambda \in \mathds{Z} / p^{N}$ such that $b_k - b_k^\prime = \binom{d}{k}_{p} \lambda$ for $0<k<d$.
    Since $b_{pm_0} = 0 = b_{pm_0}^\prime$, we have $b_{pm_0} - b_{pm_0}^\prime = 0 = \binom{d}{pm_0}_{p} \lambda$.
    By Example \ref{202506251631}, we have $\binom{d}{pm_0}_{p} = 1$, so $\lambda = 0$.
\end{proof}

Lemma \ref{202506231531} gives rise to a map 
$ H^1 ( \ACPN) \to Z^{1} (\ACPN)$ which is used in the following definition.

\begin{Defn}
    Suppose that $p$ does not divide $d$.
    We define a map $\psi : H^{1}(\ACPN) \to \mathds{Z}/p^{N}$ as the following composition
    \begin{align*}
        H^1 ( \ACPN) \to Z^{1} (\ACPN) \to \mathds{Z}/p^{N} 
    \end{align*}
    where the first map arises from  Lemma \ref{202506231531} and the second map is the projection $(b_1,\cdots, b_{d-1}) \mapsto b_1$.
\end{Defn}

Lemma \ref{202506231531} implies that the specific representative is uniquely determined by $b_1$.
Hence, the map  $\psi : H^{1}(\ACPN) \to \mathds{Z}/p^{N}$ is {\it injective}. To compute $H^1 ( \ACPN)$, we study the image of $\psi$ after giving some preliminaries lemmas.

\begin{Lemma} \label{202507081528}
    Let $d \in \nat$ such that $p \nmid d$ and $r,k \in \nat$ such that $0 < r< k <d$.
    \begin{enumerate}
        \item If $p \nmid r$, then we have $\binom{d-1}{d-r}_{p} \binom{d-r}{d-k}_{p}  = \binom{d}{k}_{p} \binom{k}{r}_{p}$.
        \item If $p\nmid k$, we have $\binom{d-1}{d-k}_{p} = \binom{d}{k}_{p}$.
    \end{enumerate}
\end{Lemma}
\begin{proof}
    Let $A = \binom{d-1}{d-r} \binom{d-r}{d-k}$, $B = \binom{d-1}{d-k} \binom{k}{r}$ and $C =\binom{d}{k} \binom{k}{r}$.
    Then it is easy to show $\frac{d}{r} A = C = \frac{d}{k}B$.
    If $p \nmid r$, then $dA = r C$ implies that $v_{p}(A) = v_{p}(C)$, equivalently (1) holds, since $d$ and $r$ are coprime to $p$.
    Likewise, if $p \nmid k$, then $kC = dB$ yields $\nu_p (C) = \nu_p (B)$, so that we obtain $\binom{d-1}{d-k}_{p} \binom{k}{r}_{p} = \binom{d}{k}_{p} \binom{k}{r}_{p}$, hence $\binom{d-1}{d-k}_{p} = \binom{d}{k}_{p}$.
\end{proof}

In the following, we use $\theta_{p} (d)$ introduced in Definition \ref{202506251636}.
\begin{Lemma} \label{202507081139}
    Suppose that $p$ does not divide $d$.
    For $x \in \mathds{Z}/p^N$, the following are equivalent:
    \begin{enumerate}
        \item If $p\mid k$ and $p \nmid r$, or if $p\nmid k$ and $p \mid r$, then we have $\binom{d}{k}_{p} \binom{k}{r}_{p} x = 0$.
        \item There exists $(b_1, \cdots, b_{d-1}) \in Z^1 ( \ACPN)$ that satisfies $b_1 =x $ and the conditions in Lemma \ref{202506231531}.
        \item We have $p^{\theta_{p}(d)}  x = 0$.
    \end{enumerate}
\end{Lemma}
\begin{proof}
    We shall prove (1) starting from (2).
    Let $(b_1, \cdots, b_{d-1})$ be the representative given by Lemma \ref{202506231531} such that $b_1 =x$.
    Consider $r,k \in \nat$ such that $0 < r < k < d$.
    We assume that $p \mid k$ and $p \nmid r$.
    Then, by (\ref{202507081443}), we have $b_r = \binom{d-1}{d-r}_{p} b_1$.
    Hence, the 1-cocycle condition (\ref{202507081101}) implies $\binom{k}{r}_{p} b_k = \binom{d-r}{d-k}_{p} b_r =\binom{d-r}{d-k}_{p} \binom{d-1}{d-r}_{p} b_1$.
    By the conditions in Lemma \ref{202506231531}, we have $b_k = 0$, so that we obtain $\binom{d-r}{d-k}_{p} \binom{d-1}{d-r}_{p} b_1 = 0$.
    By Lemma \ref{202507081528}, this leads to $$\binom{d}{k}_{p} \binom{k}{r}_{p} b_1 = \binom{d-r}{d-k}_{p} \binom{d-1}{d-r}_{p} b_1 = 0 . $$ 
    The remaining part under the assumption $p \nmid k$ and $p \mid r$ is proved analogously by using Lemma \ref{202507081528}.

    We now prove (2) from (1).
    We define $(b_1, b_2 \cdots, b_{d-1}) \in (\mathds{Z}/p^{N})^{d-1}$ as follows:
    \begin{align*}
        b_k = 
        \begin{cases}
        \binom{d-1}{d-k}_{p} x & (p\nmid k), \\
        0 & (p\mid k) .
        \end{cases}
    \end{align*}
    Clearly, we have $b_1 = x$.
    To prove that $(b_1, \cdots, b_{d-1}) \in Z^{1} ( \ACPN)$, we verify below the 1-cocycle condition (\ref{202507081101}).
    Let $0 < r < k < d$.
    \begin{itemize}
        \item If $p \mid k$ and $p \mid r$, then the 1-cocycle condition automatically holds by $b_k = b_r = 0$.
        \item If $p\mid k$ and $p \nmid r$, then we have $\binom{k}{r}_{p} b_k = 0$.
        Moreover, Lemma \ref{202507081528} implies $\binom{d-r}{d-k}_{p} b_r = \binom{d-r}{d-k}_{p} \binom{d-1}{d-k}_{p} x = \binom{b}{k}_{p} \binom{k}{r}_{p} x = 0$ by the assumption on $x$.
        \item If $p \nmid k$ and $p \mid r$, then we have $\binom{d-r}{d-k}_{p} b_r = 0$.
        Likewise, by Lemma \ref{202507081528}, we see that $\binom{k}{r}_{p} b_k = \binom{d-1}{d-k}_{p} \binom{k}{r}_{p} x = \binom{d}{k}_{p} \binom{k}{r}_{p} x = 0$.
        \item If $p \nmid k$ and $p \nmid r$, then we have $\binom{k}{r}_{p} b_k = \binom{k}{r}_{p} \binom{d-1}{d-k}_{p} x$ and $\binom{d-r}{d-k}_{p} b_r = \binom{d-r}{d-k}_{p} \binom{d-1}{d-r}_{p} x$.
        These are the same by Lemma \ref{202507081528}.
    \end{itemize}
    
    By definition of $\theta_{p}(d)$, (1) is equivalent to (3).
\end{proof}

\begin{theorem} \label{202507081603}
    If $p$ does not divide $d$, then the map $\psi : H^{1}(\ACPN) \to \mathds{Z}/p^{N}$ induces an isomorphism
    $$
    H^{1}(\ACPN) \cong \{ x \in \mathds{Z}/p^N |~ p^{\theta_{p}(d)}  x = 0 \} .
    $$
    Here, $\theta_{p} (d)$ is defined in Definition \ref{202506251636}.
\end{theorem}
\begin{proof}
    By Lemma \ref{202506231531}, the map $\psi$ is injective.
    By Lemma \ref{202507081139}, the image of $\psi$ is the group in the right hand side.
\end{proof}

\begin{Corollary} \label{202507081749}
    If $p$ does not divide $d$, then we have
    \begin{align*}
        H^1 ( \ACPN ) \cong 
        \begin{cases}
            \mathds{Z}/p & \mathrm{if}\ d = p^{m} + 1 \ \mathrm{for}\  m \in \nat^* , \\
            0 & \mathrm{otherwise} .
        \end{cases}
    \end{align*}
    In the first case, the element $(p^{N-1} ,0, \cdots, 0) \in \ACPNk^{1}$ gives a generator of $H^1 ( \ACPN )$.
\end{Corollary}
\begin{proof}
    It follows from applying Proposition \ref{202507081602} to Theorem \ref{202507081603}.
    If $d = p^m + 1$ for some $m \in \nat^*$, then $\theta_{p}(d) =1$, so that the group in the right hand side in Theorem \ref{202507081603} is generated by $p^{N-1} \in \mathds{Z}/p^{N}$.
    By Lemma \ref{202507081139}, there exists $(b_1, \cdots, b_{d-1}) \in Z^1 ( \ACPN)$ with $b_1 = p^{N-1}$ which is characterized by Lemma \ref{202506231531}.
    If $k >1$ and $p \nmid k$, then we have $b_k = \binom{d-1}{d-k}_{p} b_1 = \binom{p^{m}}{d-k}_{p} p^{N-1} = 0$ since $\binom{p^{m}}{d-k}$ is divided by $p$.
    Hence, $(b_1, \cdots, b_{d-1}) = (p^{N-1}, 0, \cdots, 0)$ induces a generator of $H^{1}(\ACPN)$.
\end{proof}

\subsubsection{Key lemmas for the case $p|d$} \label{202507212309}

We now present some key lemmas for the case where $p$ divides $d$. In this section, we assume that $p \mid d$ unless otherwise specified, and put $d^\prime {:=} d/p \in \nat^*$.

\begin{Defn} \label{202507251508}  
    We define a map $\phi_{d} : \ACPNk^{1} \to \ACPNQk^{1}$ by
    \begin{align*}
        \phi_{d} ( a_1, \cdots, a_{d-1}) {:=} (a_{p} , a_{2p}, \cdots, a_{(d^\prime -1)p}) .
    \end{align*}
    Here, we omit $p,N$ in the notation of $\phi_{d}$ since they are fixed in this section.
\end{Defn}

In the following lemmas, we study some basic properties of $\phi_{d}$.

\begin{Lemma} \label{202507141034}
    We have $\phi_{d} ( Z^{1} ( \ACPN) ) \subset Z^{1}( \ACPNQ)$.
\end{Lemma}
\begin{proof}
    Let $(a_1, \cdots, a_{d-1}) \in Z^{1} ( \ACPN)$.
    To show that $(a_{p} , a_{2p}, \cdots, a_{(d^\prime -1)p}) \in Z^{1}(\ACPNQ)$, we should check the following:
    $$\binom{k}{r}_{p} a_{kp} = \binom{d^\prime -r}{d^\prime -k}_{p} a_{rp}, \quad 0 < r< k < d^\prime.$$
    If we apply the 1-cocycle condition (\ref{202507081101}) to $pk$ and $pr$, we obtain this, since we have $\binom{k}{r}_{p} = \binom{pk}{pr}_{p}$ and $\binom{d^\prime -r}{d^\prime -k}_{p} = \binom{d-pr}{d-pk}_{p}$ by Lemma \ref{202506251622}.
\end{proof}

\begin{Lemma}
    We have $\phi_{d} ( B^{1} (\ACPN) ) \subset B^{1}(\ACPNQ)$.
    Hence, by Lemma \ref{202507141034}, $\phi_{d}$ induces a homomorphism $\bar{\phi_{d}} : H^{1}( \ACPN) \to H^{1}( \ACPNQ)$.
\end{Lemma}
\begin{proof}
    By definitions, $\phi_{d}$ assigns $(\binom{d}{p}_{p}, \binom{d}{2p}_{p}, \cdots, \binom{d}{(d^\prime -1)p}_{p})$ to the generator $(\binom{d}{1}_{p}, \cdots, \binom{d}{d-1}_{p})$ of $B^{1} (\ACPN)$.
    By Lemma \ref{202506251622}, this coincides with the generator of $B^{1} (\ACPNQ)$ given by
    $(\binom{d^\prime}{1}_{p}, \binom{d^\prime}{2}_{p}, \cdots, \binom{d^\prime}{d^\prime-1}_{p})$.
\end{proof}

\begin{Lemma} \label{202507081743}
    If $d > p$, then the map $\bar{\phi_{d}} :H^{1}( \ACPN) \to H^{1} (\ACPNQ)$ is injective.
\end{Lemma}
\begin{proof}
    Let $v \in H^{1}( \ACPN)$ such that $\bar{\phi_{d}} (v) =0$.
    We choose $(a_1 , \cdots , a_{d-1}) \in Z^{1}(\ACPN)$ a representative of $v$.
    By the hypothesis, we have $\phi_{d} (a_1 , \cdots, a_{d-1}) \in B^{1}( \ACPNQ)$.
    In other words, there exists $\lambda\in \mathds{Z}/p^{N}$ such that $a_{pl} = \binom{d^\prime}{l}_{p} \lambda$ for $0 < l < d^\prime$.
    We shall prove that $(a_1 , \cdots, a_{d-1}) \in B^{1}( \ACPN)$.
    By Lemma \ref{202506251622}, we have $\binom{d^\prime}{l}_{p} = \binom{d}{pl}_{p}$, so that we obtain $a_{pl} = \binom{d}{pl}_{p} \lambda$.
    To prove $(a_1, \cdots, a_{d-1}) \in B^{1} ( \ACPN)$, it is sufficient to prove that $a_{k} = \binom{d}{k}_{p} \lambda$ for $p \nmid k$.
    Let $k \in \nat$ such that $0<k<d$ and $p \nmid k$.
    By the hypothesis that $p < d$, the 1-cocycle condition implies
    $$
    \binom{p}{1}_{p} a_p = \binom{d-1}{d-p}_{p}a_1 .
    $$
    Since $p|d$ and $d>p$, the remainder of $d-1$ upon division by $p$ is $p-1$. Therefore, by Example \ref{202506251631}, we have $\binom{d-1}{p-1}_{p} =1$. By symmetry of binomial coefficients, it follows that 
    \begin{equation}\label{202604022213}
    \binom{d-1}{d-p}_{p} = \binom{d-1}{p-1}_{p} =1.
    \end{equation}
    Hence, by the choice of $\lambda$, we obtain 
    $$a_1 = \binom{p}{1}_{p} a_p= \binom{p}{1}_{p} \binom{d}{p}_{p} \lambda.$$
    Consequently, by the first part of Lemma \ref{202507081528} and (\ref{202604022213}), we obtain $$a_1 = \binom{d-1}{d-p}_{p} \lambda =  \lambda .$$
    Applying this to (\ref{202507081443}), we obtain $a_k = \binom{d-1}{d-k}_{p} \lambda$.
    This coincides with $\binom{d}{k}_{p} \lambda$ by the second part of Lemma \ref{202507081528}.
\end{proof}
\subsubsection{Proof of Theorem \ref{202507091101}}

In this section we prove Theorem \ref{202507091101} by applying the previous lemmas. 
The strategy is to reduce the computation of $H^{1}( \ACPN)$ (for general $d$) to that of section \ref{202507141041} by iteratively applying the map $\bar{\phi_{d}} :H^{1}( \ACPN) \to H^{1} (\ACPNQ)$.
Indeed, the map is proved to be an isomorphism by a case analysis, depending on whether $d$ belongs to $A(p)$ or $B(p)$.
Here, the sets $A(p)$ and $B(p)$ are given in Notation \ref{202507141154}. 

\begin{Lemma} \label{202507171051}
    If $d \not\in A(p) \amalg B(p)$, then we have $H^{1}( \ACPN) \cong 0 $
\end{Lemma}
\begin{proof}
    We prove the statement by induction on $v_{p}(d)$.
    If $v_{p}(d) = 0$, then the assumption implies that there does not exist $m \in \nat^*$ such that $d = p^m +1$.
    Hence, the proof for this case follows from Corollary \ref{202507081749}.
    We now assume that the claim is true for $d^\prime \in \nat^*$ such that $v_{p} (d^\prime) = k \in \nat$.
      Let $d \in \nat^*$ such that $d \not\in A(p) \cup B(p)$ and $v_{p} (d) = k+1$.
    If we put $d^\prime = d/p$, then we have $v_{p} (d^\prime)= k$, and $d^\prime$ does not lie in $A(p) \cup B(p)$.
    Hence, the induction hypothesis implies that $H^1 ( \ACPNQ) \cong 0$.
    Furthermore, the assumption $v_{p} (d) = k+1$ implies that $d \geq p$, but $d \neq p$ since $d \not\in B(p)$.
    Hence, we can apply Lemma \ref{202507081743} and prove $H^{1} (\ACPN) \cong 0$ by the injectiveness of $\bar{\phi_{d}}$.
\end{proof}

\begin{Lemma} \label{202507171049}
    If $d \in A(p)$, then we have $H^1 ( \ACPN) \cong \mathds{Z}/p$.
    Furthermore, if $v_{p}(d) = n$, then $H^1 ( \ACPN)$ has a generator represented by $( 0, \cdots, 0, a_{p^n}= p^{N-1} , 0, \cdots, 0) \in \ACPNk^{1}$, where $p^{N-1}$ occurs in the $p^n$-th coordinate.
\end{Lemma}
\begin{proof}
    We fix $m \in \nat^*$, and prove the statement for $d= p^n ( p^m +1)$ following the induction on $n \in \nat$.
    The case $n = 0$ follows from Corollary \ref{202507081749}.
    We now assume that the statement is true for $n= i \in \nat$.
    By applying Lemma \ref{202507081743} to $d = p^{i+1} (p^m+1)$, we obtain an injective map $\bar{\phi_{d}}$.
    By the induction hypothesis, it suffices to prove that this is surjective.
    To do that, we construct a map $f: \mathds{Z}/p \to H^{1}( \ACPN)$ characterized by $f(1) = ( 0, \cdots, p^{N-1} , \cdots, 0)$ where $p^{N-1}$ is in the $p^{i+1}$-th coordinate.
    By the definition of $\bar{\phi_{d}}$, we have $\bar{\phi_{d}} (f (1)) = ( 0, \cdots, p^{N-1} , \cdots, 0)$ where $p^{N-1}$ is in the $p^i$-th coordinate.
    Hence, the composition $\bar{\phi_{d}} \circ f$ is surjective since  $\bar{\phi_{d}} (f (1))$ is a generator of $H^{1}( \ACPNQ)$ by the induction hypothesis.
    Thus, $\bar{\phi_{d}}$ is surjective.
\end{proof}

Let $n \in \nat^*$.
For $0<k<p^n$, we have $\frac{1}{p} \binom{p^n}{k}_{p} \in \mathds{Z}$ since $p$ divides $\binom{p^n}{k}$.
In the following statement, we regard $\frac{1}{p} \binom{p^n}{k}_{p} \in \mathds{Z}/p^{N}$ as usual.

\begin{Lemma} \label{202507171050}
    If $d \in B(p)$, then we have $H^1 ( \ACPN) \cong \mathds{Z}/p$.
    Furthermore, if $v_{p}(d) =n$, then $H^1 ( \ACPN)$ has a generator induced by $(\frac{1}{p} \binom{p^n}{1}_{p}, \frac{1}{p} \binom{p^n}{2}_{p}, \cdots, \frac{1}{p} \binom{p^n}{p^n-1}_{p} )\in \ACPNk^{1}$.
\end{Lemma}
\begin{proof}
    We present an inductive proof for $d= p^n$ with respect to $n \in \nat^\ast$.
    The case $n= 1$ follows from the direct computation in Example \ref{202507091151} since we have $\binom{p}{k}_{p} = p$ for $0<k<p$.
    We now suppose that the statement is true for $n=i \in \nat^\ast$.
    By applying Lemma \ref{202507081743} to $d=p^{i+1}$, we obtain an injective map $\bar{\phi}_{d}$ whose codomain is isomorphic to $\mathds{Z}/p$ by the induction hypothesis.
    Hence, it suffices to prove that $\bar{\phi}_{d}$ is surjective.
    We construct a map $h : \mathds{Z}/p \to H^1 ( \ACPN)$ characterized by 
    $h(1) = (\frac{1}{p} \binom{p^{i+1}}{1}_{p}, \frac{1}{p} \binom{p^{i+1}}{2}_{p}, \cdots, \frac{1}{p} \binom{p^{i+1}}{p^{i+1}-1}_{p} )$.
    Then, by Lemma \ref{202506251622}, one can obtain $(\bar{\phi}_{p^{i+1}} \circ h) (1) = \left( \frac{1}{p} \binom{p^{i}}{1}_{p}, \frac{1}{p} \binom{p^{i}}{2}_{p}, \cdots, \frac{1}{p} \binom{p^{i}}{p^{i}-1}_{p} \right)$, which is a generator of $H^1 ( \ACPN)$ for $d= p^{i}$ by the induction hypothesis.
    Thus, $\bar{\phi}_{p^{i+1}} \circ h$ is surjective, so that $\bar{\phi}_{p^{i+1}}$ is surjective.
\end{proof}

\begin{remark}
    Let $d \in \nat^*$ and $d^\prime = d/p$.
    Note that the map $\phi_{d} : \ACPNk^{1} \to \ACPNQk^{1}$ given in Definition \ref{202507251508} can be reformulated using the basis introduced in section \ref{202507071042} is the following way:
    \begin{align*}
        \phi_{d} ( \langle n \rangle ) =
        \begin{cases}
            \langle n/p \rangle & p \mid n , \\
            0 & p \nmid n .
        \end{cases}
    \end{align*}
    More generally, we can define a cochain map $\phi_{d} : \ACPNk^{\bullet} \to \ACPNQk^{\bullet}$ whose $k$-th component is defined by the formula
    \begin{align*}
        \phi_{d} ( \langle n_1 ~n_2\cdots n_k \rangle ) =
        \begin{cases}
            \langle (n_1/p)~ (n_2/p) \cdots (n_k/p) \rangle & \forall j~ (p \mid n_j) , \\
            0 & \mathrm{otherwise} .
        \end{cases}
    \end{align*}
    Note that Lemma \ref{202507081743} and the lemmas in this section imply that this cochain map induces an isomorphism on the first cohomology if $d > p$.
    We hope that the analogue of Lemma \ref{202507081743} holds in higher degrees.
\end{remark}

\begin{proof}[Proof of Theorem \ref{202507091101}]
Under the assumption $d \in A(p)$ ($d \in B(p)$, resp.), the theorem follows from Lemma \ref{202507171049} (Lemma \ref{202507171050}, resp.).
If $d \not\in A(p) \amalg B(p)$, then it is proved by Lemma \ref{202507171051}.
\end{proof}

\subsubsection{Consequences of Theorem \ref{202507091101}}
As a consequence of Theorem \ref{202507091101} we obtain the following result:

\begin{Corollary} \label{202507211020}
    For $d, N \in \nat^\ast$ and $p \in \prim$, we have
    \begin{align*}
        H^1 ( \AC \otimes \mathds{Z}/p^N) \cong 
        \begin{cases}
            \mathds{Z}/p  & d \in A(p) \amalg B(p) , \\
            0 & \mathrm{otherwise} .
        \end{cases}
    \end{align*}
    Furthermore, for $n = v_{p} (d)$, we have a generator as follows.
    \begin{itemize}
        \item If $d \in A(p)$, then $( 0, \cdots, 0, p^{N-1} , 0, \cdots, 0)$, where $p^{N-1}$ is in the $p^n$-th coordinate, gives a generator.
        \item If $d \in B(p)$, then $(\frac{1}{p} \binom{p^n}{1}, \frac{1}{p} \binom{p^n}{2}, \cdots, \frac{1}{p} \binom{p^n}{p^n-1})$ gives a generator.
    \end{itemize}
\end{Corollary}
\begin{proof}
    The first assertion follows from Theorem \ref{202507091101} and the isomorphism $\varphi^1 : H^1 ( \AC \otimes \mathds{Z}/p^{N}) \to H^1 ( \ACPN)$ given by   (\ref{202507171124}).
    Furthermore, via this isomorphism, the generators of $H^1 ( \ACPN)$ given in Theorem \ref{202507091101} are mapped to $H^1 ( \AC \otimes \mathds{Z}/p^N)$.
    We compute the corresponding generators below:
    \begin{itemize}
        \item Suppose that $d \in A(p)$.
        We choose $n \in \nat$ and $m \in \nat^*$ such that $d = p^n (p^m+1)$, and set $v = ( 0, \cdots, 0, a_{p^n}=p^{N-1} , 0, \cdots, 0) \in (\mathds{Z}/p^N)^{d-1}$.
        By Theorem \ref{202507091101}, the group $H^1 ( \AC \otimes \mathds{Z}/p^N)$ is generated by $(\varphi^1)^{-1}(v) = \lambda^{-1} v$ where $\lambda = \tilde{F}_{p,N} (p^n) \tilde{F}_{p,N} (d-p^n) \mod{p^N}$.
        Since $\lambda^{-1} \in \mathds{Z}/p^N$ is invertible, the group $H^1 ( \AC \otimes \mathds{Z}/p^N)$ is generated by the cohomology class induced by $v = ( 0, \cdots, 0, a_{p^n}=p^{N-1} , 0, \cdots, 0)$.
        
        \item Suppose that $d \in B(p)$.
        Let  $n \in \nat^*$ such that $d= p^n$.
        Let $w_p,w \in (\mathds{Z}/p^N)^{d-1}$ be the elements whose $k$-th entries are $\frac{1}{p} \binom{p^n}{k}_{p}$ and $\frac{1}{p} \binom{p^n}{k}$, respectively.
        By Theorem \ref{202507091101}, the group $H^1 ( \AC \otimes \mathds{Z}/p^N)$ is generated by $(\varphi^1)^{-1}(w_p)$.
        We have $(\varphi^{1})^{-1} (w_p) = \tilde{F}_{p,N} (p^n)^{-1} w$ since, by Theorem \ref{202507021304}, we have
        \begin{align*}
            (\tilde{F}_{p,N}(k) \tilde{F}_{p,N}(p^n -k))^{-1} \frac{1}{p} \binom{p^n}{k}_{p} = \tilde{F}_{p,N} (p^n)^{-1} \frac{1}{p} \binom{p^n}{k} \mod{p^N} .
        \end{align*}
        Since $\tilde{F}_{p,N} (p^n)^{-1}$ is invertible in $\mathds{Z}/p^N$, the group $H^1 ( \AC \otimes \mathds{Z}/p^N)$ is generated by the cohomology class induced by $w$.
    \end{itemize}
\end{proof}

From the corollary, we obtain the following arithmetic result:

\begin{propo} \label{202507211003}
    The solutions  $(a_1, \ldots, a_{d-1}) \in (\mathds{Z}/p^{N})^{d-1}$ of the system of equations:
\begin{align*}
    \binom{k}{r} a_k = \binom{d-r}{d-k} a_r, \quad 0<r<k<d 
\end{align*}
is the $\mathds{Z}/p^{N}$-module generated by :
\begin{itemize}
    \item  \textbf{if $d=p^n(p^m+1)$ for some $n\in \nat$ and $m \in \nat^*$:}$\\
    ( 0, \cdots, 0, a_{p^n}=p^{N-1} , 0, \cdots, 0)$ and $
    \left( \binom{d}{1}, \binom{d}{2}, \cdots, \binom{d}{d-1} \right)$ ;
        \item \textbf{if $d=p^n$ for some $n\in \nat^*$:}\\
        $(\frac{1}{p} \binom{p^n}{1}, \frac{1}{p} \binom{p^n}{2}, \cdots, \frac{1}{p} \binom{p^n}{p^n-1})$;
        \item \textbf{otherwise :}
        $
    \left( \binom{d}{1}, \binom{d}{2}, \cdots, \binom{d}{d-1} \right)$.
\end{itemize}
\end{propo}
\begin{proof}
    As explained in the beginning of section \ref{202507041421}, we identify the $\mathds{Z}/p^{N}$-module generated by the solutions with $Z^1 (\AC \otimes \mathds{Z}/p^{N})$.
    By definition, the latter is generated by 1-coboundaries together with representatives of some generators of the cohomology $H^1 ( \AC \otimes \mathds{Z}/p^N)$.
    By definition of $\AC \otimes \mathds{Z}/p^{N}$ (see (\ref{202507071158})), its 1-coboundaries are generated by $\left( \binom{d}{1}, \binom{d}{2}, \cdots, \binom{d}{d-1} \right) \in (\mathds{Z}/p^N)^{d-1}$  regardless of $d$. 
    Combining this with a generator of $H^1 ( \AC \otimes \mathds{Z}/p^N)$ given by Corollary \ref{202507211020}, we obtain the results.
    In particular, if $d = p^n$ for some $n \in \nat^*$, the obvious relation
    
      $$
    p \left(\frac{1}{p} \binom{p^n}{1}, \frac{1}{p} \binom{p^n}{2}, \cdots, \frac{1}{p} \binom{p^n}{p^n-1} \right) =   \left( \binom{d}{1}, \binom{d}{2}, \cdots, \binom{d}{d-1} \right) 
    $$
    implies that the group $Z^1 ( \AC \otimes \mathds{Z}/p^N)$ is generated by $(\frac{1}{p} \binom{p^n}{1}, \frac{1}{p} \binom{p^n}{2}, \cdots, \frac{1}{p} \binom{p^n}{p^n-1})$.
\end{proof}

\begin{remark}
    One can obtain a similar statement by replacing $\mathds{Z}/p^N$ with $\mathds{Z}$.
    This follows from Theorem \ref{202507171119}.
\end{remark}

\section{Proof of Theorem \ref{202507041404}} \label{202507041422}

\begin{proof}[Proof of Theorem \ref{202507041404}]
We recall $A(p)$ and $B(p)$ introduced in Notation \ref{202507141154}.
Since $H^2 ( \AC) = \mathrm{Ext}^2 ( \mathfrak{a}, S^{d} \circ \mathfrak{a})$ is a finite abelian group by Lemma \ref{202510071626}, it suffices to prove that, for $p \in \prim$, if $d \in A(p)$, then the $p^N$-torsion part of $H^2 ( \AC)$ is $\mathds{Z}/p$ for   any $N \in \nat^*$; and is trivial otherwise.
Equivalently, we will prove the following:
\begin{align*}
    \mathrm{Tor} ( H^2 (\AC), \mathds{Z}/p^{N}) \cong 
    \begin{cases}
        \mathds{Z}/p & d \in A(p) , \\
        0 & \mathrm{otherwise}.
    \end{cases}
\end{align*}
To this end, we apply the exact sequence (\ref{202507171059}).
In the following, we identify $H^1( \AC \otimes \mathds{Z}/p^N)$ with $H^1 (\ACPN)$ based on (\ref{202507171124}).
\begin{itemize}
    \item If $d \in A(p)$, then $H^1 ( \AC ) \cong 0$ by Theorem \ref{202507171119}and $A(p) \cap B(p) = \emptyset$.
    Hence, the exact sequence (\ref{202507171059}) leads to $\mathrm{Tor} ( H^2 (\AC), \mathds{Z}/p^{N}) \cong H^1 ( \AC \otimes \mathds{Z}/p^N) \cong H^1(\ACPN)$.
    By Theorem \ref{202507091101}, we obtain $\mathrm{Tor} ( H^2 (\AC), \mathds{Z}/p^{N}) \cong \mathds{Z}/p$
    \item If $d \not\in A(p)\amalg B(p)$, then $H^1 ( \AC \otimes \mathds{Z}/p^N) \cong H^1 ( \ACPN ) \cong 0$ by Theorem \ref{202507091101}.
    Hence, the exactness of (\ref{202507171059}) yields $\mathrm{Tor} ( H^2 (\AC), \mathds{Z}/p^{N}) \cong 0$.
    \item We consider the case $d \in B(p)$.
    Then Theorem \ref{202507171119}implies $H^1 (\AC) \otimes \mathds{Z}/p^N \cong \mathds{Z}/p$.
    Furthermore, by Theorem \ref{202507091101}, we have $H^1 ( \AC \otimes \mathds{Z}/p^N) \cong H^1 (\ACPN) \cong \mathds{Z}/p$.
    Thus, any injective map from $H^1 (\AC) \otimes \mathds{Z}/p^N$ to $H^1 ( \AC \otimes \mathds{Z}/p^N)$ should be an isomorphism.
    Therefore, by (\ref{202507171059}), we obtain $\mathrm{Tor} ( H^2 (\AC), \mathds{Z}/p^{N}) \cong 0$.
\end{itemize}

All that remain is to give a generator of the $p$-torsion part of $H^2 ( \AC)$ when $d = p^n (p^m+1)$ for some $n \in \nat$ and $m \in \nat^*$.
To do that, we consider the following composite, which we denote by $\iota$:
\begin{align*}
    \iota: H^1(\AC \otimes \mathds{Z}/p) \stackrel{\cong}{\longrightarrow} \mathrm{Tor} ( H^2(\AC), \mathds{Z}/p) \hookrightarrow H^2 ( \AC) .
\end{align*}
where the isomorphism is deduced from the previous discussion for $N=1$, and the inclusion follows from the identification of the $\mathrm{Tor}(-,\mathds{Z}/p)$ with the $p$-torsion subgroup.
In particular, this map is characterized by the following condition: for $x \in Z^1 ( \AC \otimes \mathds{Z}/p)$ and $y \in Z^2 ( \AC)$, 
\begin{align*}
    \iota ( [x] ) = [ y ]  \mathrm{~if~and~only~if~} \delta^1 ( z ) = p y
\end{align*}
where $z$ is a lift of $x$ along the surjection $\ACk^{1} \to \ACk^{1} \otimes \mathds{Z}/p$. 
By Corollary \ref{202507211020}, we already have a generator of $H^1(\AC \otimes \mathds{Z}/p) \cong \mathds{Z}/p$ induced by $(0,\cdots,0,a_{p^n}=1,0,\cdots,0) = \langle p^n \rangle \in \ACk^{1}$.
By definition, we have
$$
 \delta^1 ( \langle p^n \rangle) = -\sum^{p^n-1}_{k=1} \binom{p^n}{k} \langle k ~ p^n \rangle + \sum^{p^{n+m}-1}_{l=1} \binom{p^{n+m}}{l} \langle p^n ~ (p^n+l) \rangle = py
$$
where we set
$$
y = -\sum^{p^n-1}_{k=1} \frac{1}{p}\binom{p^n}{k} \langle k ~ p^n \rangle + \sum^{p^{n+m}-1}_{l=1} \frac{1}{p} \binom{p^{n+m}}{l} \langle p^n ~ (p^n+l) \rangle \in \ACk^{2} .
$$
Note that, since $p \mid \binom{p^r}{k}$ for $r \in \nat^*$ and $0<k<r$, the element $y \in \ACk^2$ is well-defined.
Thus, we obtain $\iota( [\langle p^n \rangle \otimes (1\mod{p})]) = [y] \in H^2 ( \AC)$, and it is a generator of the $p$-torsion subgroup of $H^2 ( \AC)$.
\end{proof}

\section{Computation of $\mathrm{Ext}^{d-1}_{\mathcal{F}(\mathsf{gr};\mathds{Z})} ( \mathfrak{a}, S^{d} \circ \mathfrak{a} )$ and $\mathrm{Ext}^{d-2}_{\mathcal{F}(\mathsf{gr};\mathds{Z})} ( \mathfrak{a}, S^{d} \circ \mathfrak{a} )$} \label{202508081503}

In this section, we calculate $\mathrm{Ext}^{k}_{\mathcal{F}(\mathsf{gr};\mathds{Z})} ( \mathfrak{a}, S^{d} \circ \mathfrak{a})$ for $k \in \{ d-2,d-1\}$.
The content of this Section is inspired by communications with Gregory Arone who sent us a sketch of proof of Theorem \ref{202507311555K}      and Theorem \ref{202507181014}.
By Section \ref{202507071042}, this is equivalent to the computation of the last two nontrivial cohomologies of Arone's complex $H^{d-1} (\AC)$ and $H^{d-2} (\AC)$.
For this, we will write the basis of $\AC$ given in Section \ref{202507071042} in the following more convenient way:
\begin{notation}
    Let $k \in \nat$ such that $0 < k < d$.
    For strictly increasing numbers $m_1, m_2, \cdots, m_k \in \nat^\ast$ such that $0<m_j <d$, we consider the complement 
    $$\{ 1,2, \cdots, (d-1) \} \backslash \{ m_1, m_2, \cdots, m_k \} =  \{ n_1,n_2, \cdots,n_{d-k-1} \}.$$
    We denote the basis $\langle n_1n_2 \cdots n_{d-k-1} \rangle$ of $\ACk^{d-k-1}$ by $\langle m_1m_2\cdots m_k \rangle^{\mathrm{c}}$.
    For instance, $\langle ~\rangle^{\mathrm{c}}  = \langle 1~2~\cdots~(d-1) \rangle$ and $\langle 1~3 \rangle^{\mathrm{c}} = \langle 2~4~\cdots~ (d-1) \rangle$.
\end{notation}

We now return to the last two differentials of Arone's complex:
\begin{align*}
    \AC : \quad \cdots \to \ACk^{d-3} \stackrel{\delta^{d-3}}{\to} \ACk^{d-2} \stackrel{\delta^{d-2}}{\to} \ACk^{d-1} \to 0 \to \cdots .
\end{align*}
By using the above notation, one can obtain the alternative description of the differentials:
\begin{align*}
    \delta^{d-2} ( \langle m \rangle^{\mathrm{c}} ) &= 2 (-1)^{m}  \langle~\rangle^{\mathrm{c}} , \\
    \delta^{d-3} ( \langle m_1~m_2 \rangle^{\mathrm{c}} ) &=
    \begin{cases}
        2 \biggl( (-1)^{m_1}  \langle m_2 \rangle^{\mathrm{c}} - (-1)^{m_2}  \langle m_1 \rangle^{\mathrm{c}} \biggl) & (m_1+1 < m_2), \\
        3 (-1)^{m_1} \biggl( \langle m_1+1 \rangle^{\mathrm{c}} + \langle m_1 \rangle^{\mathrm{c}} \biggl) & (m_2=m_1+1) .
    \end{cases}
\end{align*}

\begin{theorem} \label{202507311555K}
    Let $d \in \nat^*$ such that $d \geq 2$. We have:
    \begin{align*}
        \mathrm{Ext}^{d-1}_{\mathcal{F}(\mathsf{gr};\mathds{Z})} ( \mathfrak{a}, S^{d} \circ \mathfrak{a}) = H^{d-1} ( \AC) \cong \mathds{Z}/2 .
    \end{align*}
    The generator is given by the cohomology class of $\langle ~ \rangle^{\mathrm{c}}$.
\end{theorem}
\begin{proof}
 The image of $\delta^{d-2}$ is generated by $2 \langle~\rangle^{\mathrm{c}} \in \ACk^{d-1}$ and the kernel of $\delta^{d-1}$ is $\ACk^{d-1}$ which is equal to $\mathds{Z}$.
\end{proof}

In order to describe $H^{d-2}( \AC)$, for $d \geq 3$ and $2 \leq m \leq d-1$, we consider the following elements in $\ACk^{d-2}$:
$$u_m {:=} (-1)^m \langle m \rangle^{\mathrm{c}} + \langle 1 \rangle^{\mathrm{c}} .$$
These elements satisfy the following:
\begin{Lemma} \label{202507111139}
    \begin{enumerate}
        \item If $d \geq 3$, then we have $\delta^{d-2}(u_m)=0$.
        \item If $d \geq 3$, then we have $3u_2 = -\delta^{d-3} ( \langle 1~2\rangle^{\mathrm{c}})$.
        \item If $d \geq 4$, then we have $u_3 = - \delta^{d-3} ( \langle 1~2 \rangle^{\mathrm{c}} + \langle 2~3\rangle^{\mathrm{c}} + \langle 1~3 \rangle^{\mathrm{c}})$.
        \item If $m \geq 3$, then we have $$u_{m+1} - u_{m} = (-1)^{m+1} \delta^{d-3} ( \langle 1 ~m \rangle^{\mathrm{c}} + \langle 1 ~(m+1)\rangle^{\mathrm{c}} + (-1)^m \langle m ~(m+1) \rangle^{\mathrm{c}} ).$$
    \end{enumerate}
\end{Lemma}
\begin{proof}
    These are immediate from the definitions.
\end{proof}

\begin{theorem} \label{202507181014}
    Let $d \in \nat^*$ such that $d \geq 2$. We have:
    \begin{align*}
        \mathrm{Ext}^{d-2}_{\mathcal{F}(\mathsf{gr};\mathds{Z})} ( \mathfrak{a}, S^{d} \circ \mathfrak{a}) =H^{d-2}( \AC) \cong
        \begin{cases}
            \mathds{Z}/3 & d \in \{3,4 \}, \\
            0 & \mathrm{otherwise} .
        \end{cases}
    \end{align*}
    In the first case, $u_2 = \langle 2 \rangle^{\mathrm{c}} + \langle 1 \rangle^{\mathrm{c}}$ gives a generator.
\end{theorem}
\begin{proof}
    The case $d=2$ is clear. We assume that $d \geq 3$. By the first point of Lemma \ref{202507111139}, the group $Z^{d-2} ( \AC)$ is generated by $u_m$ for $2 \leq m \leq d-1$. In particular, for $d=3$, $Z^{1} (\AC)$ is generated by $u_2$. By the third and fourth points of Lemma \ref{202507111139}, we have $u_m \in B^{d-2}(\AC)$ for $m \geq 3$. Hence, $H^{d-2} ( \AC)$ is generated by the class of $u_2$ denoted by $[u_2]$.
    
     We first assume that $d \geq 5$.
    By this assumption, we have an element $u_4 \in \ACk^{d-2}$.
    By the first point of Lemma \ref{202507111139}, $u_4 \in Z^{d-2} ( \AC)$.
    It suffices to prove that $[u_2] = 0$.
    By definition, we have $2 (u_2 - u_4) = - \delta^{d-3} ( \langle 2~4 \rangle^{\mathrm{c}})$ which leads to $2 [u_2] = 2[u_2 -u_4] = 0$.
    Furthermore, the second part of Lemma \ref{202507111139} gives $3[u_2] = 0$.
    Thus, we obtain $[u_2] =0$.
    
    For the case $d\in \{3,4\}$, the second part of Lemma \ref{202507111139} implies that the order of the generator $[u_2] \in H^{d-2}(\AC)$ is at most $3$.
    To prove that $[u_2]$ is not zero, we construct a surjective map $H^{d-2}(\AC) \to \mathds{Z}/3$.
    Let $f :Z^{d-2} ( \AC) \to \mathds{Z}/3$ be the map such that $f(u_2) = 1$ and $f(u_{m} ) = 0$ for $m > 2$.
    One can directly verify that this map is well-defined.
    This factors through the quotient map $Z^{d-2} (\AC) \to H^{d-2} (\AC)$.
    In fact, for $1 \leq m \leq d-2$, we have $$\delta^{d-3} ( \langle m ~(m+1) \rangle^{\mathrm{c}}) = 3 (-1)^{m} \biggl( \langle m+1 \rangle^{\mathrm{c}} + \langle m \rangle^{\mathrm{c}} \biggl),$$ so $f ( \delta^{d-3} ( \langle m (m+1) \rangle^{\mathrm{c}}) ) = 0$.
    In the case $d = 4$, we have $$\delta^{d-3} ( \langle 1~3\rangle^{\mathrm{c}} ) = 2 ( - \langle 3 \rangle^{\mathrm{c}} + \langle 1 \rangle^{\mathrm{c}}) = 2u_3 ,$$ 
    so $f ( \delta^{d-3} (\langle 1 ~3 \rangle^{\mathrm{c}})) = 0$.
    Since the map $H^{d-2} ( \AC) \to \mathds{Z}/3$ is surjective, $H^{d-2}(\AC)$ is generated by $[u_2]$ with order $3$.
\end{proof}

\begin{remark}
    Note that Theorem \ref{202507181014} agrees with Theorem \ref{202507171119} when $d=3$; and with
    Theorem \ref{202507041404} when $d = 4$. 
    In particular, the generators given by each theorem are the same.
    Indeed, for $d=3$, the cochain $u_2$ is equal to $(1,1) \in \ACk^{1}$.
    For $d = 4$, we have $[u_2] = [\langle 3 \rangle^{\mathrm{c}} + \langle 2 \rangle^{\mathrm{c}}] \in H^{d-2}(\AC)$ where $\langle 3 \rangle^{\mathrm{c}} + \langle 2 \rangle^{\mathrm{c}} = \langle 1~2 \rangle + \langle 1 ~ 3 \rangle$ is the representative of the generator given by Theorem \ref{202507041404}.
    In fact, we have $u_2 - (\langle 3 \rangle^{\mathrm{c}} + \langle 2 \rangle^{\mathrm{c}}) = u_3 \in B^1 ( \AC)$ by Lemma \ref{202507111139}.
\end{remark}

\section{$\mathrm{Ext}^{*}_{\mathcal{F}(\mathsf{gr};\mathds{Z})} ( T^c \circ \mathfrak{a}, S^{d} \circ \mathfrak{a})$} \label{202508081128}
In this section, we establishes a K\"unneth formula that enables the computation of $\mathrm{Ext}^{*}_{\mathcal{F}(\mathsf{gr};\mathds{Z})} ( T^c \circ \mathfrak{a}, S^{d} \circ \mathfrak{a})$  from $\mathrm{Ext}^{*}_{\mathcal{F}(\mathsf{gr};\mathds{Z})} ( \mathfrak{a}, S^{d} \circ \mathfrak{a})$, and provide some examples to illustrate its application. To obtain this K\"unneth formula, we adapt the method used in \cite{vespa2018extensions} to compute $\mathrm{Ext}^{*}_{\mathcal{F}(\mathsf{gr};\mathds{Q})} ( T^c \circ \mathfrak{a}, T^{d} \circ \mathfrak{a})$  from $\mathrm{Ext}^{*}_{\mathcal{F}(\mathsf{gr};\mathds{Q})} ( \mathfrak{a}, T^{d} \circ \mathfrak{a})$. The main difference lies in the fact that $\mathrm{Ext}^{*}_{\mathcal{F}(\mathsf{gr};\mathds{Z})} ( \mathfrak{a}, S^{d} \circ \mathfrak{a})$ are not torsion free, so we have to consider $\mathrm{Tor}^{\mathds{Z}}_1$.

For $F$ and $G$ in $\mathcal{F}(\mathsf{gr};\mathds{Z})$, their external tensor product $F \boxtimes G$ is the functor on $\gr \times \gr$ sending the object $(n,m)$ of $\gr \times \gr$ to $F(n)\otimes G(m)$. This yields a functor $-\boxtimes -: \mathcal{F}(\mathsf{gr};\mathds{Z}) \times \mathcal{F}(\mathsf{gr};\mathds{Z}) \to \mathcal{F}(\mathsf{gr}\times \mathsf{gr};\mathds{Z})$ where  $\f(\gr \times \gr; \mathds{Z})$ is the category of functors from $\gr \times \gr$ to $\mathsf{Ab}$.
Let $\pi_n: \mathsf{gr}^{\times n} \to \mathsf{gr}$ be the functor obtained by iteration of the free product.

A graded functor $E^\bullet=(E^n)_{n \in \nat}$ of $\mathcal{F}(\mathsf{gr};\mathds{Z})$ is \textit{exponential} if it is a strong symmetric monoidal functor. In other words, it is a sequence $E^\bullet=(E^n)_{n \in \nat}$ of objects of $\mathcal{F}(\mathsf{gr};\mathds{Z})$ together with natural isomorphisms $E^0 \simeq \mathds{Z}$ and 

\begin{align} \label{202508051151}
\pi_n^*( E^d)\cong \underset{i_1+ \ldots +i_n=d}{\bigoplus} E^{i_1}  \boxtimes \ldots  \boxtimes E^{i_n}.
\end{align}

\begin{Lemma} \label{202510201704}
    Let $E^\bullet$ be a graded exponential functor of $\mathcal{F}(\mathsf{gr};\mathds{Z})$, we have:
    $$\mathrm{Ext}^{*}_{\mathcal{F}(\mathsf{gr};\mathds{Z})} ( T^c \circ \mathfrak{a}, E^{d}) \cong\underset{i_1+ \ldots+i_c=d}{\bigoplus}\mathrm{Ext}^*_{\mathcal{F}(\mathsf{gr}\times \ldots \times \mathsf{gr};\mathds{Z})}\big( \mathfrak{a}^{\boxtimes c},  E^{i_1} \boxtimes \ldots \boxtimes E^{i_c} \big) $$
\end{Lemma}

\begin{proof}
    We have the following isomorphisms:
    \begin{align*}
&\mathrm{Ext}^{*}_{\mathcal{F}(\mathsf{gr};\mathds{Z})} ( T^c \circ \mathfrak{a}, E^{d}) \\
\cong& \mathrm{Ext}^*_{\mathcal{F}(\mathsf{gr}\times \ldots \times \mathsf{gr};\mathds{Z})}( \mathfrak{a}^{\boxtimes c}, E^d \circ \pi_c) \text{\ by the sum-diagonal adjunction},\\
\cong &\mathrm{Ext}^*_{\mathcal{F}(\mathsf{gr}\times \ldots \times \mathsf{gr};\mathds{Z})}\big( \mathfrak{a}^{\boxtimes c}, \underset{i_1+ \ldots+i_c=d}{\bigoplus} (E^{i_1}\boxtimes \ldots \boxtimes E^{i_c}) \big)
 \text{\ by (\ref{202508051151})},\\
\cong & \underset{i_1+ \ldots+i_c=d}{\bigoplus}\mathrm{Ext}^*_{\mathcal{F}(\mathsf{gr}\times \ldots \times \mathsf{gr};\mathds{Z})}\big( \mathfrak{a}^{\boxtimes c},  E^{i_1} \boxtimes \ldots \boxtimes E^{i_c}\big).
\end{align*}
\end{proof}

The graded functor $S^{\bullet} \circ \mathfrak{a}=(S^{d} \circ \mathfrak{a})_{d\in \nat}$ is an exponential functor of $\mathcal{F}(\mathsf{gr};\mathds{Z})$.
By Lemma \ref{202510201704}, we have
\begin{equation} \label{202510201716}
\mathrm{Ext}^{*}_{\mathcal{F}(\mathsf{gr};\mathds{Z})} ( T^c \circ \mathfrak{a}, S^{d} \circ \mathfrak{a}) 
\cong  \underset{i_1+ \ldots+i_c=d}{\bigoplus}\mathrm{Ext}^*_{\mathcal{F}(\mathsf{gr}\times \ldots \times \mathsf{gr};\mathds{Z})}\big( \mathfrak{a}^{\boxtimes c},  S^{i_1}\circ \mathfrak{a} \boxtimes \ldots \boxtimes S^{i_c}\circ \mathfrak{a} \big).
\end{equation}
For $c=2$, we obtain the following result, using in an essential way the projective resolution (\ref{202507291730}): 
\begin{propo}[K\"unneth formula] \label{202508211555}
    We have an exact sequence:
\begin{align*}
    0 \to \underset{\alpha+\beta=k}{\bigoplus} \mathrm{Ext}^\alpha_{\mathcal{F}(\mathsf{gr};\mathds{Z})}( \mathfrak{a}, S^{i_1} \circ \mathfrak{a}) \otimes \mathrm{Ext}^\beta_{\mathcal{F}(\mathsf{gr};\mathds{Z})}( \mathfrak{a}, S^{i_2} \circ \mathfrak{a}) \to \mathrm{Ext}^k_{\mathcal{F}(\mathsf{gr}\times \mathsf{gr};\mathds{Z})}( \mathfrak{a}^{\boxtimes 2}, S^{i_1} \circ \mathfrak{a} \boxtimes S^{i_2} \circ \mathfrak{a} ) \\ \to \underset{\alpha+\beta=k+1}{\bigoplus} \mathrm{Tor}^\mathds{Z}_1(\mathrm{Ext}^\alpha_{\mathcal{F}(\mathsf{gr};\mathds{Z})}( \mathfrak{a}, S^{i_1} \circ \mathfrak{a}), \mathrm{Ext}^\beta_{\mathcal{F}(\mathsf{gr};\mathds{Z})}( \mathfrak{a}, S^{i_2} \circ \mathfrak{a})) \to 0
\end{align*}
\end{propo}
\begin{proof}
  By (\ref{202507291730}), we have a projective resolution $P_\bullet \to \mathfrak{a}$ of $\mathfrak{a}$ where $P_k=\mathds{Z}[\mathsf{gr}(k,-)]$.   By the Yoneda lemma, we have $\mathrm{Hom}_{\mathcal{F}(\mathsf{gr};\mathds{Z})}( P_k, S^{i} \circ \mathfrak{a})\cong S^i(\mathds{Z}^k)$, which is a free abelian group. Hence, the complex $C_1^\bullet=\mathrm{Hom}_{\mathcal{F}(\mathsf{gr};\mathds{Z})}(P_\bullet, S^{i_1} \circ \mathfrak{a})$ is flat. We now consider the complexes $C_2^\bullet=\mathrm{Hom}_{\mathcal{F}( \mathsf{gr};\mathds{Z})}(P_\bullet, S^{i_2} \circ \mathfrak{a})$ and $C_1^\bullet \otimes C_2^\bullet$. Note that $P_k \boxtimes P_l \cong \mathds{Z}[(\mathsf{gr} \times \mathsf{gr}) \Big( (k,l),-\Big)]$ so, by the Yoneda lemma, we have $\mathrm{Hom}_{\mathcal{F}( \mathsf{gr}\times \mathsf{gr};\mathds{Z})}(P_k \boxtimes P_l, S^{i_1} \circ \mathfrak{a} \boxtimes S^{i_2} \circ \mathfrak{a} ) \cong  S^{i_1}(\mathds{Z}^k) \otimes S^{i_2}(\mathds{Z}^l)$. 
  Hence, we have:
  \begin{align*}
    &(C_1^\bullet \otimes C_2^\bullet)^n \\
    &=\underset{k+l=n}{\bigoplus} \mathrm{Hom}_{\mathcal{F}( \mathsf{gr};\mathds{Z})}(P_k, S^{i_1} \circ \mathfrak{a}) \otimes \mathrm{Hom}_{\mathcal{F}( \mathsf{gr};\mathds{Z})}(P_l, S^{i_2} \circ \mathfrak{a}), \\
    &\cong \underset{k+l=n}{\bigoplus} S^{i_1}(\mathds{Z}^k) \otimes S^{i_2}(\mathds{Z}^l) \quad &\textrm{by the Yoneda lemma},\\
    &\cong \underset{k+l=n}{\bigoplus} \mathrm{Hom}_{\mathcal{F}( \mathsf{gr}\times \mathsf{gr};\mathds{Z})}(P_k \boxtimes P_l, S^{i_1} \circ \mathfrak{a} \boxtimes S^{i_2} \circ \mathfrak{a} ),\\
    & =\underset{k+l=n}{\prod} \mathrm{Hom}_{\mathcal{F}( \mathsf{gr}\times \mathsf{gr};\mathds{Z})}(P_k \boxtimes P_l, S^{i_1} \circ \mathfrak{a} \boxtimes S^{i_2} \circ \mathfrak{a} ) \quad &\textrm{$P_\bullet$: bounded below},\\
    &\cong \mathrm{Hom}_{\mathcal{F}( \mathsf{gr}\times \mathsf{gr};\mathds{Z})}(\underset{k+l=n}{\bigoplus}P_k \boxtimes P_l, S^{i_1} \circ \mathfrak{a} \boxtimes S^{i_2} \circ \mathfrak{a} ),\\
    &= \mathrm{Hom}_{\mathcal{F}( \mathsf{gr}\times \mathsf{gr};\mathds{Z})}\big((P_\bullet \boxtimes P_\bullet)_n, S^{i_1} \circ \mathfrak{a} \boxtimes S^{i_2} \circ \mathfrak{a} \big),\\
    &= \Big(\mathrm{Hom}_{\mathcal{F}( \mathsf{gr}\times \mathsf{gr};\mathds{Z})}(P_\bullet \boxtimes P_\bullet, S^{i_1} \circ \mathfrak{a} \boxtimes S^{i_2} \circ \mathfrak{a} ) \Big)^n .
  \end{align*}

Since $P_\bullet \boxtimes P_\bullet$ is a projective resolution of $\mathfrak{a} \boxtimes \mathfrak{a}$, we deduce the statement from  the cochain version of the K\"unneth formula (which is dual to the K\"unneth formula for chain complexes given, for example, in \cite[Thm 2.1 p. 172]{Hilton-Stamm}).
\end{proof}
Combining (\ref{202510201716}), Proposition \ref{202508211555}, Theorems \ref{202507171119}, \ref{202507041404} and \ref{202507311555K}, we obtain the following computations:
\begin{Corollary}
We have:
    \begin{align*}
        \mathrm{Ext}^{i}_{\mathcal{F}(\mathsf{gr};\mathds{Z})} ( T^2 \circ \mathfrak{a}, S^{2} \circ \mathfrak{a}) =
        \begin{cases}
            \mathds{Z} & i =0; \\
            0 & \mathrm{otherwise} .
        \end{cases}
    \end{align*}
    \begin{align*}
        \mathrm{Ext}^{i}_{\mathcal{F}(\mathsf{gr};\mathds{Z})} ( T^2 \circ \mathfrak{a}, S^{3} \circ \mathfrak{a}) =
        \begin{cases}
            (\mathds{Z}/2)^{ \oplus 2} & i =1; \\
            0 & \mathrm{otherwise} .
        \end{cases}
    \end{align*}
    \begin{align*}
        \mathrm{Ext}^{i}_{\mathcal{F}(\mathsf{gr};\mathds{Z})} ( T^2 \circ \mathfrak{a}, S^{4} \circ \mathfrak{a}) =
        \begin{cases}
           (\mathds{Z}/3)^{ \oplus 2}  \oplus \mathds{Z}/2& i =1; \\
            (\mathds{Z}/2)^{ \oplus 3}  & i =2; \\
            0 & \mathrm{otherwise} .
        \end{cases}
    \end{align*}
\begin{align*}
        \mathrm{Ext}^{i}_{\mathcal{F}(\mathsf{gr};\mathds{Z})} ( T^2 \circ \mathfrak{a}, S^{5} \circ \mathfrak{a}) =
        \begin{cases}
            (\mathds{Z}/2)^{ \oplus 2} & i =1; \\
            (\mathds{Z}/3)^{ \oplus 2} \oplus (\mathds{Z}/2)^{ \oplus 2} & i =2; \\
            (\mathds{Z}/2)^{\oplus 4} & i =3; \\
            0 & \mathrm{otherwise} .
        \end{cases}
    \end{align*}
\end{Corollary}
\begin{remark}
    Note that $\mathrm{Ext}^{i}_{\mathcal{F}(\mathsf{gr};\mathds{Z})} ( T^2 \circ \mathfrak{a}, S^{4} \circ \mathfrak{a}) $ is computed in \cite[Example 11.11]{arone2025polynomial}.
\end{remark}

\appendix
\section{Arithmetic results on binomial coefficients} \label{202507071237}
In this Appendix we recall some classical arithmetic properties of binomial coefficients used in the paper. The main results used in the paper are 
Proposition \ref{202507081602} which will be crucial in Section \ref{202507141041} and Theorem \ref{202507021304} which will be used in Section \ref{202507071356}. For the completeness of the paper we give the proofs of all the results.

Throughout this appendix, we fix a prime $p$.
For $m \in \nat$, consider the base-$p$ expansion of $m$ i.e. $m = \sum^{l}_{k=0} m_k p^{k}$ where $m_k \in \nat$ and $0 \leq m_k < p$.
We assume that $m_l \neq 0$ unless specified otherwise.
We denote the sum of the digits by 
$$S_{p} (m) {:=} \sum^{l}_{k=0} m_k .$$

\textbf{Kummer's theorem:}
Let $v_{p} (m)$ be the {\it $p$-adic valuation} of $m$ i.e. the exponent of the highest power of $p$ that divides $m$.
The $p$-adic valuation of binomial coefficients can be computed as follows:
\begin{theorem}[Kummer] \label{Kummer}
    Let $n,r \in \nat$ such that $n \geq r$.
    Then we have
    \begin{align*}
        v_{p} \left( \binom{n}{r} \right) = \frac{S_{p}(r) + S_{p}(n-r) - S_{p}(n)}{p-1} .
    \end{align*}
\end{theorem}

For the sake of simplicity, we will use the notation $v_{p}  \binom{n}{r}$ instead of $v_{p} \left( \binom{n}{r} \right)$.

The following example is used in Section \ref{202507141041}.
\begin{Example} \label{202506251631}
    For $n \in \nat$ and $r$ the remainder when $n$ is divided by $p$.
    Let $n=\sum^{l}_{k= 0} n_k p^k$ be the base-$p$ expansion of $n$.
    By the assumption, we have $n_0 = r$, so that Theorem \ref{Kummer} implies $v_{p}\binom{n}{r} = \frac{r + \sum^{l}_{k=1} n_k - \sum^{l}_{k= 0} n_k}{p-1} = 0$.
\end{Example}

For $m \in \nat$, $\lfloor m \rfloor$ denotes the maximal integer which does not exceed $m$.
\begin{Lemma} \label{202506251622}
    Let $d,r \in \nat$ such that $d \geq r$.
    If $p \mid r$, then 
    we have $v_{p} \binom{d}{r} = v_{p}  \binom{\lfloor d/p \rfloor}{r/p}$.
\end{Lemma}
\begin{proof}
    To prove the assertion, it is useful to review the fact that if $n = \sum^{l}_{k=0} n_k p^k$ is the base-$p$ expansion, then so $\lfloor n/p \rfloor = \sum^{l}_{k=1} n_k p^{k-1}$ is.
    Let $d = \sum^{l}_{k= 0} d_k p^{k}$, $r = \sum^{l}_{k= 0} r_k p^{k}$ and $d-r = \sum^{l}_{k= 0} m_k p^{k}$ be the base-$p$ expansions.
    Here, $d_l$, $r_l$ and $m_l$ are allowed to be zero.
    As $p \mid r$, it follows that $r_0 = 0$, and hence $d_0 = m_0$.
    By the previous fact, the base-$p$ expansions of $\lfloor d/p \rfloor$, $\lfloor r/p \rfloor$ and $\lfloor (d-r)/p \rfloor$ are given by $\sum^{l}_{k= 1} d_k p^{k-1}$, $\sum^{l}_{k=1} r_k p^{k-1}$ and $\sum^{l}_{k=1} m_k p^{k-1}$ respectively.
    The assumption $p \mid r$ implies $\lfloor (d-r)/p \rfloor = \lfloor d/p \rfloor - r/p$ and $\lfloor r/p \rfloor = r/p$.
    Hence, we have $S_{p} ( \lfloor d/p \rfloor ) = \sum^{l}_{k=1} d_k$, $S_{p} ( \lfloor d/p \rfloor-r/p ) = \sum^{l}_{k=1} m_k$ and $S_{p} (r/p) = \sum^{l}_{k=1} r_k$.
    By Theorem \ref{Kummer}, we obtain
    {\small \begin{align*}
    v_{p}  \binom{\lfloor d/p \rfloor}{r/p} = \frac{\sum^{l}_{k = 1} r_k + \sum^{l}_{k = 1} m_k - \sum^{l}_{k = 1} d_k}{p-1}
        = \frac{\sum^{l}_{k = 0} r_k + \sum^{l}_{k = 0} m_k - \sum^{l}_{k = 0} d_k}{p-1} 
        = v_{p} \binom{d}{r}.
    \end{align*}}
\end{proof}

\begin{Example}
\label{202506301203}
    Let $m,l\in \nat$  such that $m \leq l$.
    By iteratively applying Lemma \ref{202506251622}, we obtain 
    $v_{p}  \binom{p^{l}}{p^{m}} = v_{p}  \binom{p^{l-1}}{p^{m-1}} = \cdots = v_{p}  \binom{p^{l-m}}{1}= l-m$.
\end{Example}

\begin{Lemma} \label{202506251609}
    Let $d \in \nat$ such that $d \geq 2$ and $p \in \prim$.
    If $d \neq p^l$ for any $l \in \nat$, then there exists $r$ such that $0 < r < d$ and $v_{p} \binom{d}{r} = 0$.
\end{Lemma}
\begin{proof}
    If $d < p$, then one may choose $r = 1$.
    Assume that $d \geq p$, and let $d = \sum^{l}_{k= 0} d_k p^{k}$ be the base-$p$ expansion of $d$ with $d_{l} \neq 0$.
    We consider $r =  p^{l}$.
    By the assumption, it is clear that $0 < r < d$.
    By Lemma \ref{202506251622}, we have
    \begin{align*}
        v_{p}  \binom{d}{p^{l}}  = v_{p}  \binom{\lfloor d/p \rfloor}{p^{l-1}}  .
    \end{align*}
    Note that $\lfloor d/p \rfloor$ has the base-$p$ expansion $\sum^{l}_{k= 1} d_k p^{k-1}$.
    Hence, the iterative applications of Lemma \ref{202506251622} yields $v_{p} \binom{d}{p^{l}} = v_{p}\binom{d_l}{1} = 0$.
\end{proof}

\vspace{3mm}
In this paper, we also need the following definition:
\begin{Defn} \label{202506251636} 
    For $d \in \nat$ such that $d \geq 2$, we define:
    \begin{enumerate}
        \item
        $\mu_{p} (d) {:=} \min \Bigl\{ v_{p}  \binom{d}{k}  \mid 0<k<d \Bigr\}
       \mathrm{\quad  and\quad  }  \mu_{p}(1) {:=} 0$;
        \item
       $\theta_{p} (d)$ to be the minimum of 
        \begin{align*}
            v_{p}  \binom{d}{k}  + v_{p}  \binom{k}{r} 
        \end{align*}
        where $0 < r < k < d$ and $(p\mid k \wedge p\nmid r) \vee (p\mid r \wedge p\nmid k)$.
        If there does not exist such $k$ and $r$, then we define $\theta_{p}(d) {:=} 0$.
    \end{enumerate}
\end{Defn}

\begin{propo}
\label{202506301525}
    Let $d \in \nat^*$.
    We then have 
    \begin{align*}
        \mu_{p} (d) =
        \begin{cases}
            1  & \mathrm{if}\  d = p^l \ \mathrm{for}\  l \in \nat^*  \\
            0 & \mathrm{otherwise}.
        \end{cases}
    \end{align*}
\end{propo}
\begin{proof}
    If $d$ is not a power of $p$, then, by Lemma \ref{202506251609}, we obtain $\mu_{p}(d) = 0$.
    We now suppose that $d = p^{l}$ for some $p \in \prim$ and $l \in \nat^\ast$.
    By the Kummer's theorem (Theorem \ref{Kummer}), we have $v_{p} \binom{p^{l}}{k} \geq 1$.
    Moreover, using Example \ref{202506301203}, we have $v_{p} \binom{p^{l}}{p^{l-1}} = 1$, so we obtain $\mu_{p}(p^{l}) = 1$.
\end{proof}

We can deduce the following well-known result:

\begin{Corollary}
\label{202506301226}
    Let $d \in \nat$.
    \begin{align*}
        \gcd \bigl\{ \binom{d}{k}\mid 0<k<d \bigr\} =
        \begin{cases}
            p & \mathrm{~if~} d = p^l \mathrm{~for~}  p \mathrm{~a~prime~and~} l \in \nat^*\\
             1  & \mathrm{~if~} d \mathrm{~is~not~a~power~of~any~prime.} \\
        \end{cases}
    \end{align*}
\end{Corollary}
\begin{proof}
    Let $g$ be the left hand side.
    By definitions, $\mu_{p}(d)$ coincides with the number of $p$'s in the prime factorization of $g$.
    In other words, we have $v_{p} (g) = \mu_{p}(d)$, which is already computed in Proposition \ref{202506301525}.
    By applying the proposition to any prime $p$, we obtain the statement.
\end{proof}

The following arithmetic result will be crucial in Section \ref{202507141041}
to deduce Corollary \ref{202507081749} from Theorem \ref{202507081603}.
    
\begin{propo} \label{202507081602}
    Suppose that $p$ does not divide $d$.
    Then we have
    \begin{align*}
        \theta_{p}(d) =
        \begin{cases}
            1 &  \mathrm{if}\ d = p^{l} + 1 \ \mathrm{for~some}\  l \in \nat^\ast ,\\
            0 & \mathrm{otherwise} .
        \end{cases}
    \end{align*}
\end{propo}
\begin{proof}
    Let $d = \sum^{l}_{i= 0} d_i p^i$ be the base-$p$ expansion of $d$ such that $d_l \neq 0$.
    By the hypothesis that $p \nmid d$, it is clear that $d_0 \neq 0$.
    If $l = 0$, i.e. $d<p$, then it is trivial that $\theta_{p} (d) = 0$.
    We now assume that $l > 0$ and divide the argument into several cases:
    \begin{enumerate}
        \item
        We assume that $d_l > 1$.
        Consider $k = p^{l} + \sum^{l-1}_{i= 0} d_ip^{i}$ and $r = p^{l} + \sum^{l-1}_{i= 1} d_i p^{i}$.
        It is obvious that $0<r<k<d$, $p \mid r$ and $p \nmid k$.
        Furthermore, one may verify that $v_{p} \binom{d}{k} = v_{p}\binom{k}{r} =0$ by using the Kummer's theorem.
        Hence, we obtain
        $0 \leq \theta_{p}(d) \leq v_{p} \binom{d}{k} + v_{p}\binom{k}{r} = 0$ which leads to $\theta_{p}(d) = 0$.
        \item 
        We assume that $d_0 >1$.
        For $k = \sum^{l}_{i= 1} d_i p^i +1 $ and $r = \sum^{l}_{i= 1} d_i p^i$, we analogously obtain $0 \leq \theta_{p}(d) \leq v_{p} \binom{d}{k} + v_{p}\binom{k}{r} = 0$.
        \item 
        Finally, we consider the case where $d_l = d_0 = 1$.
        \begin{itemize}
            \item We assume that there exists $j$ such that $0<j<l$ and $d_j \neq 0$ (especially, $l>1$).
            If we set $k = \sum^{j}_{i= 0} d_i p^i$ and $r = \sum^{j}_{i= 1} d_i p^i$, then we have $0<r<k<d$, $p \mid r$ and $p \nmid k$.
            By the Kummer's theorem, we obtain $v_{p} \binom{d}{k} = v_{p}\binom{k}{r} = 0$ which proves $0 \leq \theta_{p}(d) \leq v_{p} \binom{d}{k} + v_{p}\binom{k}{r} = 0$.
            \item The only exclusive case is that $d = p^{l} + 1$.
            To prove that $\theta_{p}(d) = 1$, we start with showing that $\theta_{p}(d)$ is positive.
            By the Kummer's theorem, for $0 \leq k \leq p^{l}+1$, the binomial coefficient $\binom{p^l+1}{k}$ is coprime to $p$ if and only if $k \in \{0, 1 , p^l , p^l+1 \}$.
            Thus, we have $v_{p} \binom{d}{k} + v_{p}\binom{k}{r} > 0$ for $0 < r < k < p^l$.
            Note that this is true even if $k = p^l$.
            In fact, $\binom{p^l}{r}$ is divided by $p$ for $0 < r < p^{l}$, so that $v_{p}\binom{k}{r} > 0$.
            Thus, it follows that $\theta_{p}(d)>0$.
            
            ~~All that remains is to prove that there exists $k,r$ for which the conditions in Definition \ref{202506251636} hold and $v_{p} \binom{d}{k} + v_{p}\binom{k}{r} =1$.
            If $l > 1$, then we put $k = p^{l-1}+1, r = p^{l-1}$.
            For these $k$ and $r$, the Kummer's theorem yields $v_{p} ( \binom{p^{l}+1}{k} ) = 1$ and $ v_{p} (\binom{k}{r})= 0$.
            Similarly, if $l =1$, it suffices to consider $k = p$ and $r = 1$.
        \end{itemize}
    \end{enumerate}
\end{proof}

The purpose of the end of this appendix is to compute the factor of $\binom{n}{r}$ which is coprime to $p$ modulo any power of $p$. The result is given in Theorem \ref{202507021304} which requires the introduction of the following definitions.

\begin{Defn}
    For $m \in \nat^*$, we define $F_{p} (m)$ to be the product of positive integers $\leq m$ that are coprime to $p$. 
    We also define $F_{p}(0)=1$.
\end{Defn}
For $n \in \nat$, let $n = \sum^{l}_{k=0} n_k p^k$ be the base-$p$ expansion.
    We also set $n_{k} = 0$ for $k > l$.

\begin{Defn} \label{202507040934}
    Let $N \in \nat^*$.
    We define $\tau^{(N)}_j (n) {:=} \sum^{N-1}_{k=0} n_{k+j} p^k$ for any $j \in \nat$ and \\
    $\alpha_{p,N} (n) {:=} \sum_{j\geq N} \lfloor n/p^{j} \rfloor$.    
    By using these, we define
    $$\tilde{F}_{p,N} (n)  {:=} (\pm 1)^{\alpha_{p,N} (n)} \prod_{j \geq 0} F_{p} ( \tau^{(N)}_{j} (n)) . $$
    Here, $(\pm 1)$ is $-1$ except if $p = 2$ and $N \geq 3$.
    If $j$ is sufficiently large, then we have $\tau^{(N)}_{j}(n) = 0$ so that this definition reduces to a finite product.
\end{Defn}

Note that, by definition, $F_{p}(m)$ is coprime to $p$, hence, so $\tilde{F}_{p,N}(n)$ is.

\begin{Defn} \label{202507040924}
    Let $p \in \prim$.
    For non-negative integers $k,r$ such that $k \geq r$, we define $\binom{k}{r}_{p} {:=} p^{v_{p} \binom{k}{r}}$.
    In other words, $\binom{k}{r}_{p}$ is the exact power of the prime $p$ dividing the binomial coefficient $\binom{k}{r}$.
\end{Defn}

In \cite[Theorem 1]{granville1997arithmetic}, the factor of $\binom{n}{r}$ that is coprime to $p$, i.e. $\binom{n}{r} / \binom{n}{r}_{p}$, is computed modulo any power of $p$.
We give its equivalent statement by using our notation:
\begin{theorem}\label{202507021304}
    Let $p \in \prim$ and $N \in \nat^*$.
    For $n,r \in \nat$ such that $n \geq r$, we have
    \begin{align*}
        \frac{\binom{n}{r}}{\binom{n}{r}_{p}} = \frac{\tilde{F}_{p,N} (n)}{\tilde{F}_{p,N} (r)\tilde{F}_{p,N} (n-r)} \mod{p^{N}}  .
    \end{align*}
    
\end{theorem}

\section{Integral extensions between functors on free abelian groups} \label{202508141211k}

Let $\mathsf{ab}$ be the category of finitely-generated free abelian groups and group homomorphisms. This category is essentially small, with skeleton labeled by $\nat$, where $n\in \nat$ corresponds to the free abelian group $\mathds{Z}^n$. 
We denote by $\f(\mathsf{ab}; \mathds{Z})$  the category of functors from $\mathsf{ab}$ to $\mathsf{Ab}$. In this appendix, we present results concerning Ext-groups between functors in $\f(\mathsf{ab}; \mathds{Z})$, which are related to the results of this paper. 

Let $I:\mathsf{ab}\to \mathsf{Ab}$ be the inclusion functor and $S^d : \mathsf{ab}\to \mathsf{Ab}$ be the functor which assigns the $d$-th symmetric power.
The computation of $ \mathrm{Ext}^{*}_{\mathcal{F}(\mathsf{ab};\mathds{Z})} ( I, S^{d} ) $ is recalled in the two following results:

\begin{theorem}[\mbox{B\"okstedt (see also \cite[Corollary 3.2]{Franjou-Pira})}]
\label{202508150909k}
We have:
\begin{align*}
  \mathrm{Ext}^{i}_{\mathcal{F}(\mathsf{ab};\mathds{Z})} ( I, I ) \cong
        \begin{cases}
            \mathds{Z}/ l & \mathrm{~if~} i=2l \mathrm{~for~some} l\in \nat
             , \\
            0  & \mathrm{otherwise}.
        \end{cases}
    \end{align*}
\end{theorem}

In \cite{Franjou-Pira}, Franjou and Pirashvili obtain the following result:
\begin{theorem}
[\mbox{\cite[Prop. 2.1]{Franjou-Pira}}]
\label{202508111718k}
    For $d\geq 2$, we have:
    \begin{align*}
  \mathrm{Ext}^{i}_{\mathcal{F}(\mathsf{ab};\mathds{Z})} ( I, S^{d} ) \cong
        \begin{cases}
            \mathds{Z}/ p & \mathrm{~if~} i\equiv 1\pmod{2d} \mathrm{~and~}  d = p^l \mathrm{~for~some} p \in \prim , l \in \nat^* 
             , \\
            0  & \mathrm{otherwise}.
        \end{cases}
    \end{align*}
\end{theorem}

\begin {table}[h] 
\[\scalebox{0.88}{
\begin{tabular}{|c|c|c|c|c|c|c|c|c|c|c|c|c|c|}
\hline
$d ~\backslash~ i$ & 0 & 1 & 2 & 3 & 4 & 5 & 6 & 7 & 8 & 9 & 10 & 11 & $\cdots$ \\
\hline
1 & $\mathds{Z}$ & 0 & 0 & 0 & $\mathds{Z}/2$ & 0 & $\mathds{Z}/3$ & 0&$\mathds{Z}/4$& 0& $\mathds{Z}/5$ & 0 & $\cdots$  \\
\hline
2 & 0 & $\mathds{Z}/2$ & 0 & 0 & 0 & $\mathds{Z}/2$ & 0  & 0&0&$\mathds{Z}/2$& 0 & 0 & $\cdots$\\
\hline
3 & 0 & $\mathds{Z}/3$ & 0 & 0 & 0 & 0 & 0 &$\mathds{Z}/3$  & 0 & 0& 0 & 0 & $\cdots$ \\
\hline
4 & 0 & $\mathds{Z}/2$ & 0& 0 & 0 & 0 & 0 & 0 & 0 &$\mathds{Z}/2$ & 0 & 0 & $\cdots$ \\
\hline
5 & 0 & $\mathds{Z}/5$ & 0 & 0 & 0 & 0 & 0 & 0 & 0 &0& 0& $\mathds{Z}/5$& $\cdots$ \\
\hline
6 & 0 & 0 & 0 & 0 & 0 & 0& 0  & 0 & 0 &0& 0 & 0 & $\cdots$\\
\hline
7 & 0 & $\mathds{Z}/7$ & 0 & 0 & 0 & 0 & 0 & 0 & 0 &0& 0& 0& $\cdots$ \\
\hline
$\vdots$ & $\vdots$ & $\vdots$ & $\vdots$ & $\vdots$ & $\vdots$ & $\vdots$ & $\vdots$ & $\vdots$  & $\vdots$ &$\vdots$& $\vdots$ & $\vdots$ & $\ddots$ \\
\hline
\end{tabular}
}
\]
\caption { $\mathrm{Ext}^{i}_{\mathcal{F}(\mathsf{ab};\mathds{Z})} ( I, S^{d} )$} \label{20250801525}
\end{table}

We give, in Table \ref{20250801525}, the computations of $\mathrm{Ext}^{i}_{\mathcal{F}(\mathsf{ab};\mathds{Z})} ( I, S^{d} )$ for small values of $i$ and $d$. It is noteworthy to compare Table \ref{20250801525} with Table \ref{20250801529}.
The aforementioned theorems and the direct computation of $\mathrm{Ext}^{0}_{\mathcal{F}(\mathsf{gr};\mathds{Z})} ( \mathfrak{a}, S^{d} \circ \mathfrak{a}) = H^0 ( \AC)$ imply that, in both tables, the columns of $i=0$ are identical. 
Furthermore, Corollary \ref{202508041451} establishes that their columns of $i=1$ are identical.
In Table \ref{20250801529}, the entries above the diagonal is entirely trivial while, in Table \ref{20250801525}, infinitely many nontrivial entries appear in each row by Theorems \ref{202508150909k} and \ref{202508111718k}.
Analogously, the lower triangle of Table \ref{20250801525}, except the columns $i=0$ and $i=1$, is trivial, whereas that of Table \ref{20250801529} exhibits a variety of nontrivial values, whose pattern is not fully understood.

In the rest of this appendix, we construct an explicit extension that generates $ \mathrm{Ext}^{1}_{\mathcal{F}(\mathsf{ab};\mathds{Z})} (I, S^{p^l} )$ for $p\in \prim$ and $l \in \nat^*$.
To do that, we give a review of a generator-relation description of the category $\mathsf{ab}$.
The category $\mathsf{ab}$ is a \texttt{PROP} with respect to the symmetric strict monoidal structure given by the direct sum, denoted by $\oplus$.
In \cite{Pira}, Pirashvili gives a presentation of $\mathsf{ab}$ as a \texttt{PROP}: the \texttt{PROP} $\mathsf{ab}$ is freely generated by a bicommutative (i.e. commutative and cocommutative) Hopf algebra object $1 \in \mathsf{ab}$.
The structure morphisms of the Hopf algebra object $1$ are given by:
\begin{enumerate}
\item the counit $\epsilon : 1 \to 0$ corresponding to $\mathds{Z} \to \{0\}$;
\item the comultiplication $\Delta: 1 \to 2$ corresponding to the diagonal map $\mathds{Z} \to \mathds{Z}^2$;
\item the unit $\eta : 0 \to 1$ corresponding to $\{0\} \to \mathds{Z}$;
\item the multiplication $\nabla: 2 \to 1$ corresponding to $\mathds{Z}^2 \to \mathds{Z} ; (x,y) \mapsto x+y$.
\item the antipode $S : 1 \to 1$ corresponding to $\mathds{Z} \to \mathds{Z}; x \mapsto -x$;
\end{enumerate}
These satisfy the following:
\begin{itemize}
    \item the commutative algebra axioms of $(\nabla, \eta)$ and the cocommutative coalgebra axioms of $(\Delta, \epsilon)$;
    \item the compatibility of $(\nabla, \eta)$ and $(\Delta, \epsilon)$, i.e.
    $$\Delta \circ \nabla=(\nabla \oplus \nabla)\circ (1 \oplus \tau \oplus 1) \circ (\Delta \oplus \Delta)$$
    where $\tau: 2 \to 2$ is the morphism switching the generators, and
    $$\epsilon \circ \nabla=\epsilon \oplus \epsilon; \qquad \Delta \circ \eta=\eta \oplus \eta; \qquad \epsilon \circ \eta=\mathrm{id};$$
    \item the antipode axiom on $S$.
\end{itemize}

In Proposition \ref{202508111122k}, we define the functor  $E_d: \mathsf{ab}\to \mathsf{Ab}$. Since this functor is not symmetric monoidal, we need to  give a presentation of  $\mathsf{ab}$ as a {\it category}.

\begin{notation} \label{202508111240k}
    For a morphism $f$ in $\mathsf{ab}$ and $n,m \in \nat$, let ${}_{n}[f]_{m} {:=} \mathrm{id}_{n} \oplus f \oplus \mathrm{id}_{m}$.
\end{notation}

\begin{Defn} \label{202508061713k}
    Let $\mathcal{G} {:=} \{\tau,  \nabla, \Delta, \epsilon, \eta ,S \}$.
    We define $\Gamma = (V_{\Gamma} ,E_{\Gamma})$ as a directed graph whose vertex set is $V_{\Gamma} = \nat$ and the edge set is $$E_{\Gamma} = \{ {}_{n}[f]_{m} \mid f \in \mathcal{G} , (n,m) \in \nat^2 \} .$$
    In particular, for $(f:s \to t) \in \mathcal{G}$ and $n,m \in \nat$, ${}_{n}[f]_{m}$ is an edge of $\Gamma$ from $(s+n+m)$ to $(t + n+m)$.
\end{Defn}

\begin{notation}
    Let $\mathcal{C}_{\Gamma}$ be the category freely generated by the directed graph $\Gamma$. 
\end{notation}

We have a functor $\mathcal{C}_{\Gamma} \to \mathsf{ab}$ which assigns the morphism viewed in $\mathsf{ab}$ to each edge of $\Gamma$.
This functor is full by Pirashvili's result above.
Hence, when we construct a functor on $\mathsf{ab}$, we shall start with a construction of that on $\mathcal{C}_{\Gamma}$, and show that it factors through $\mathcal{C}_{\Gamma} \to \mathsf{ab}$.
For this purpose, we give a description of the relation on morphisms in $\mathcal{C}_{\Gamma}$ induced by the functor $\mathcal{C}_{\Gamma} \to \mathsf{ab}$.

\begin{Defn} \label{202508111530}
    We define a relation $\mathcal{R}$ on morphisms of $\mathcal{C}_{\Gamma}$.
    For $g,h \in \mathcal{C}_{\Gamma}$, if $(g,h) \in \mathcal{R}$, then we denote by $g \sim_{\mathcal{R}} h$, or merely $g \sim h$.
    The relation $\mathcal{R}$ consists of the following where $n,m \in \nat$:
    \begin{itemize}
        \item the monoidality:
        \begin{align*} 
            {}_{n}[f_1]_{a+t_2+m} \circ {}_{n+s_1+a}[f_2]_{m} \sim {}_{n+t_1+a}[f_2]_{m} \circ {}_{n}[f_1]_{a+s_2+m} ,
        \end{align*}
        where $f_1: s_1 \to t_1, f_2 : s_2 \to t_2$ lie in $\mathcal{G}$ and $a \in \nat$;
        \item the symmetry relations:
        \begin{align*} 
            {}_{n}[\tau \circ \tau]_{m} \sim \mathrm{id}_{n+2+m} , \quad {}_{n}[\tau]_{m+1} \circ {}_{n+1}[\tau]_{m} \circ {}_{n}[\tau]_{m+1} \sim {}_{n+1}[\tau]_{m} \circ {}_{n}[\tau]_{m+1} \circ {}_{n+1}[\tau]_{m} ;
        \end{align*}
        \item the naturality of the symmetry;
        {\small \begin{align*}
            {}_{n}[\tau]_{m} \circ {}_{n}[\nabla]_{m+1} \sim {}_{n+1}[\nabla]_{m} \circ {}_{n}[\tau]_{m+1} \circ {}_{n+1}[\tau]_{m} , \quad
            {}_{n}[\tau]_{m+1} \circ {}_{n+1}[\tau]_{m} \circ {}_{n}[\Delta]_{m+1} \sim {}_{n+1}[\Delta]_{m} \circ {}_{n}[\tau]_{m} , \\
            {}_{n}[\epsilon]_{m+1} \sim {}_{n+1}[\epsilon]_{m} \circ {}_{n}[\tau]_{m} , \quad {}_{n}[\eta]_{m+1} \sim {}_{n}[\tau]_{m} \circ {}_{n+1}[\eta]_{m}, \quad {}_{n}[\tau]_{m} \circ {}_{n}[S]_{m+1} \sim {}_{n+1}[S]_{m} \circ {}_{n} [\tau]_{m} ;
        \end{align*}}
    \end{itemize}
    and the axioms of a bicommutative Hopf algebra object:
    \begin{itemize}
        \item the commutative algebra axioms of $(\nabla,\eta)$:
        \begin{align*}
            {}_{n}[\nabla]_{m} \circ {}_{n}[\nabla]_{m+1} \sim  {}_{n}[\nabla]_{m} \circ&  {}_{n+1}[\nabla]_{m} , \quad {}_{n}[\nabla]_{m} \circ {}_{n}[\tau]_{m} \sim {}_{n}[\nabla]_{m}, \\
            {}_{n}[\nabla]_{m} \circ {}_{n}[\eta]_{m+1} \sim &\mathrm{id}_{n+1+m} \sim {}_{n}[\nabla]_{m} \circ {}_{n+1}[\eta]_{m} ;
        \end{align*}
        \item the cocommutative coalgebra axioms of $(\Delta,\epsilon)$:
        \begin{align*}
            {}_{n}[\Delta]_{m+1} \circ {}_{n}[\Delta]_{m} \sim {}_{n+1}[\Delta]_{m} &\circ {}_{n}[\Delta]_{m} , \quad {}_{n}[\tau]_{m}\circ {}_{n}[\Delta]_{m} \sim {}_{n}[\Delta]_{m} , \\
            {}_{n} [\epsilon]_{m+1} \circ {}_{n}[\Delta]_{m} \sim &\mathrm{id}_{n+1+m} \sim {}_{n+1}[\epsilon]_{m} \circ {}_{n}[\Delta]_{m} ;
        \end{align*}
        \item the antipode axiom on $S$:
        $$
        {}_{n}[\nabla]_{m} \circ {}_{n}[S]_{m+1} \circ {}_{n}[\Delta]_{m} \sim {}_{n}[\eta]_{m} \circ {}_{n}[\epsilon]_{m} \sim {}_{n}[\nabla]_{m} \circ {}_{n+1}[S]_{m} \circ {}_{n}[\Delta]_{m} ;
        $$
        \item the compatibility of $(\nabla, \eta)$ and $(\Delta, \epsilon)$
        \begin{align} \label{202508072052k}
            {}_{n}[\Delta]_{m} \circ  {}_{n}[\nabla]_{m} \sim {}_{n+1}[\nabla]_{m} \circ {}_{n}[\nabla]_{m+2} \circ {}_{n+1}[\tau]_{m+1} \circ {}_{n+2}[\Delta]_{m} \circ {}_{n}[\Delta]_{m+1},
        \end{align}
        
        \begin{align} \label{202508151051k}
            {}_{n}[\epsilon]_{m} \circ {}_{n}[\eta]_{m} = \mathrm{id}_{n+m} ,
        \end{align}
        and
        \begin{align} \label{202508111506k}
            {}_{n}[\epsilon]_{m} \circ {}_{n}[\nabla]_{m} \sim {}_{n}[\epsilon]_{m} \circ {}_{n}[\epsilon]_{m+1} , \qquad {}_{n}[\Delta]_{m} \circ {}_{n}[\eta]_{m} \sim {}_{n}[\eta]_{m+1} \circ {}_{n}[\eta]_{m} . 
        \end{align}
    \end{itemize}
    We define $\overline{\mathcal{R}}$ as the equivalence relation on morphisms of $\mathcal{C}_{\Gamma}$ which is generated by the above relation $\mathcal{R}$ and compatible with compositions.
\end{Defn}

\begin{remark}
    One may realize that the compatibility of $(\nabla,\eta)$ and $(\Delta,\epsilon)$ is slightly different from what we explained above. 
    In fact, by definition of $\mathcal{C}_{\Gamma}$, the notations $\nabla \oplus \nabla, \Delta \oplus \Delta, \epsilon \oplus \epsilon, \eta \oplus \eta$ do not make sense in $\mathcal{C}_{\Gamma}$.
    The relations (\ref{202508072052k}) and (\ref{202508111506k}) realize the compatibility by taking a representative of them with respect to the full functor $\mathcal{C}_{\Gamma} \to \mathsf{ab}$.
    The choice of such representatives does not affect the relation $\overline{\mathcal{R}}$ since, by the monoidality relation, we have, for example, ${}_{n+1}[\nabla]_{m} \circ {}_{n}[\nabla]_{m+2} \sim {}_{n}[\nabla]_{m+1} \circ {}_{n+2}[\nabla]_{m}$.
\end{remark}

Pirashvili's result \cite{Pira} implies that a functor $F$ defined on the category $\mathcal{C}_{\Gamma}$ factors through $\mathcal{C}_{\Gamma} \to \mathsf{ab}$ if and only if $F$ respects the relation $\overline{\mathcal{R}}$.
In particular, the precomposition of $S^d$ ($I$, resp.) with $\mathcal{C}_{\Gamma} \to \mathsf{ab}$, which is denoted by the same notation $S^d$ ($I$, resp.), respects the relation $\overline{\mathcal{R}}$.

Since $\mathcal{C}_{\Gamma}$ is the category freely generated by $\Gamma$, to define a functor on $\mathcal{C}_{\Gamma}$, it is sufficient to define it on the vertices and the edges of $\Gamma$. 
\begin{Defn} \label{20250801636}
    Let $d = p^{l}$ for some $l \in \nat^*$.
    We define a functor $E_d^\prime : \mathcal{C}_{\Gamma} \to \mathsf{Ab}$.
    For an object $s \in \mathcal{C}_{\Gamma}$, we set $E_{d}^\prime (s) {:=}  S^{d} (s) \oplus I(s)$.
    For an edge $g : s \to t$ of $\Gamma$, we define a map $E_{d}^\prime(g) : E_{d}^\prime(s) \to E_{d}^\prime(t)$ by the following matrix expression
    \begin{align*}
        E_{d}^\prime ( g ) & {:=} 
        \begin{pmatrix}
            S^{d}( g ) & D_{d} (g) \\
            0 & I( g ) 
        \end{pmatrix}
    \end{align*}
    where $D_{d}(g) : I(s) \to  S^{d}(t)$ is the map defined below by using Notation \ref{202508111240k}. 
    For $\{e_1, \ldots, e_{n+m+1}\}$ the canonical basis of $I(n+m+1) = \mathds{Z}^{n+m+1}$, we define: 
\begin{itemize}
    \item
     \begin{align*}
         D_{d} ( {}_{n}[S]_{m} ) (e_i) {:=} 
            \begin{cases}
                \frac{1+(-1)^{d}}{p}e_{n+1}^{d} & \mathrm{if\ } i=n+1, \\
                0 & \mathrm{otherwise} .
            \end{cases}
    \end{align*}
    \item  
    \begin{align*}
        D_{d} ( {}_{n}[\Delta]_{m} )(e_i) {:=}
        \begin{cases}
            \sum^{d-1}_{k=1} \frac{1}{p} \binom{p^{l}}{k} e_{n+1}^{k} e_{n+2}^{d-k} & \mathrm{if\ } i=n+1, \\
            0 & \mathrm{otherwise}.
        \end{cases}
    \end{align*}
    \item for other edges $g$ of $\Gamma$, we define $D_d (g) {:=} 0$. 
\end{itemize}
\end{Defn}

\begin{Lemma} \label{202508111628}
    The functor $E_{d}^\prime$ respects the $\mathcal{R}$-relation given by the monoidality.
    In other words, we have        
    \begin{align*} 
            E_{d}^\prime ({}_{n}[f_1]_{a+t_2+m}) \circ E_{d}^\prime ({}_{n+s_1+a}[f_2]_{m}) = E_{d}^\prime ({}_{n+t_1+a}[f_2]_{m} ) \circ E_{d}^\prime ({}_{n}[f_1]_{a+s_2+m}) ,
    \end{align*}
    where $f_1: s_1 \to t_1, f_2 : s_2 \to t_2$ lie in $\mathcal{G} = \{\tau,  \nabla, \Delta, \epsilon, \eta ,S \}$ and $n,m,a \in \nat$.
\end{Lemma}
\begin{proof}
    Let $A(f,g) {:=} S^d(f) \circ D_d (g) + D_d (f) \circ I(g)$ for composable morphisms $f,g$ in $\mathcal{C}_{\Gamma}$.
    Comparing the entries of matrices, the statement is equivalent with
    \begin{align*}
    A ( {}_{n}[f_1]_{a+t_2+m} ,~ {}_{n+s_1+a}[f_2]_{m}) = A ( {}_{n+t_1+a}[f_2]_{m} ,~ {}_{n}[f_1]_{a+s_2+m}) 
    \end{align*}
    since $S^d$ and $I$ respect the relation $\overline{\mathcal{R}}$.
    This argument is trivial for $f_1, f_2 \in \mathcal{G} \backslash \{ S, \Delta \}$ since $D_d$ is defined to be zero for them.
    One can directly verify this for other $f_1,f_2 \in \mathcal{G}$.
    For instance, consider $f_1 = \Delta$ and $f_2 \in \mathcal{G} \backslash \{ S, \Delta \}$, especially, $s_1 = 1,t_1 = 2$.
    We then have:
    \begin{align*}
        &A ( {}_{n}[\Delta]_{a+t_2+m} ,~ {}_{n+s_1+a}[f_2]_{m}) (e_{i}) =  D_d ({}_{n}[\Delta]_{a+t_2+m}) \circ I({}_{n+s_1+a}[f_2]_{m})(e_{i}) , \\
        =& D_d ({}_{n}[\Delta]_{a+t_2+m})(e_{i}) = 
        \begin{cases}
            \sum^{d-1}_{k=1} \frac{1}{p} \binom{p^{l}}{k} e_{n+1}^{k} e_{n+2}^{d-k} & i= n+1 ,\\
            0 & i \neq n+1
        \end{cases}
    \end{align*}
    On the other hand, we have:
    \begin{align*}
        &A ( {}_{n+t_1+a}[f_2]_{m} ,~ {}_{n}[\Delta]_{a+s_2+m}) (e_i) = S^d ({}_{n+t_1+a}[f_2]_{m}) \circ D_d ({}_{n}[\Delta]_{a+s_2+m}) (e_i) , \\
        =& 
        \begin{cases}
            \sum^{d-1}_{k=1} \frac{1}{p} \binom{p^{l}}{k} e_{n+1}^{k} e_{n+2}^{d-k} & i= n+1 ,\\
            0 & i \neq n+1
        \end{cases}
    \end{align*}
\end{proof}
   
\begin{propo} \label{202508111122k}
The functor $E_d^\prime : \mathcal{C}_{\Gamma} \to \mathsf{Ab}$ factors through $\mathcal{C}_{\Gamma} \to \mathsf{ab}$.
Hence, it induces a functor $E_d \in \mathcal{F}(\mathsf{ab};\mathds{Z})$.
\end{propo}

\begin{proof}
    It suffices to prove that the functor $E_d^\prime$ respects the equivalence relation $\overline{\mathcal{R}}$, equivalently the relation $\mathcal{R}$, introduced in Definition \ref{202508111530}.
    The definition of $\mathcal{R}$ appears to be complicated, but the proof can be simplified by the following strategy:
    \begin{itemize}
        \item The functor $E_d^\prime$ automatically respects the $\mathcal{R}$-relations consisting of $\mathcal{G} \backslash \{ S, \Delta\}$.
        For instance, $E_d^\prime$ respects the $\mathcal{R}$-relations given by the {\it symmetry relations} (see Definition \ref{202508111530}).
        In fact, $D_d ({}_{n}[f]_{m}) = 0$ for $n,m \in \nat$ and $f \in \mathcal{G} \backslash \{ S, \Delta\}$; and both of $S^d$ and $I$ respect the relation $\overline{\mathcal{R}}$.
        Hence, we will prove that $E_d^\prime$ respects the $\mathcal{R}$-relation containing the morphisms $S$ or $\Delta$.
        \item
        Note that, for a $\mathcal{R}$-relation $f_1 \circ \cdots \circ f_k \sim g_1 \circ \cdots \circ g_l : s \to t$ and an element $(x,0) \in E_{d}^\prime (s) {=}  S^{d} (s) \oplus I(s)$, we have 
        \begin{align*}
            &E_d^\prime (f_1) \circ \cdots \circ E_d^\prime  (f_k)(x,0)= (S^d (f_1) \circ \cdots \circ S^d (f_k) (x), 0) ,\\
            =& (S^d (g_1) \circ \cdots \circ S^d(g_l) (x), 0)
            =E_d^\prime (g_1) \circ \cdots \circ E_d^\prime  (g_l)(x,0) ,
        \end{align*}
        since $S^d$ respects the relation $\overline{\mathcal{R}}$. 
        Hence, when we prove $E_d^\prime (f_1) \circ \cdots \circ E_d^\prime  (f_k) = E_d^\prime (g_1) \circ \cdots \circ E_d^\prime  (g_l)$, it is sufficient to compare them by applying to the elements $(0,e_i) \in E_{d}^\prime (s)$ for $1 \leq i\leq s$.
    \end{itemize}
    Since we already considered the case of monoidality in Lemma \ref{202508111628}, we shall prove the other compatibilities below.
    
    {\it Naturality of the symmetry.}
    We prove that the functor $E_d^\prime$ respects the naturality of the symmetry with respect to the morphism $\Delta$, i.e. ${}_{n}[\tau]_{m+1} \circ {}_{n+1}[\tau]_{m} \circ {}_{n}[\Delta]_{m+1} \sim {}_{n+1}[\Delta]_{m} \circ {}_{n}[\tau]_{m}$.
    Let $u = E_d^\prime ({}_{n}[\tau]_{m+1}) \circ E_d^\prime ({}_{n+1}[\tau]_{m}) \circ E_d^\prime ({}_{n}[\Delta]_{m+1})$ and $v = E_d^\prime( {}_{n+1}[\Delta]_{m} )\circ E_d^\prime ( {}_{n}[\tau]_{m})$.
    The following follow from definitions:
    \begin{align*}
        u( 0, e_{n+1}) =& (\sum^{d-1}_{k=1} \frac{1}{p} \binom{p^{l}}{k} e_{n+2}^{k} e_{n+3}^{d-k} , e_{n+2}+e_{n+3} ) = v(0,e_{n+1}) , \\
        u(0,e_{n+2})& = (0,e_{n+1}) = v(0,e_{n+2}) .
    \end{align*}
    Furthermore, it is easy to see that $u( 0, e_i) = (0,e_i) =  v(0,e_i)$ for $i \leq n$; and $u( 0, e_i) = (0,e_{i+1}) =  v(0,e_i)$ for $i \geq n+3$.
    Similarly, one may prove the naturality of the symmetry with respect to the morphism $S$.

    \vspace{1mm}
    All that remains is to prove that $E_d^\prime$ respects the $\mathcal{R}$-relations arising from the axioms of a bicommutative Hopf algebra object.
    By the above strategy, the functor $E_d^\prime$ respects those given by the commutative algebra axioms of $(\nabla , \eta)$, (\ref{202508151051k}) and the first part of (\ref{202508111506k}).
    We shall prove that $E_{d}^\prime$ respects the other $\mathcal{R}$-relations related to the cocommutative coalgebra axioms, the compatibility of $(\nabla,\eta)$ and $(\Delta, \epsilon)$ and the antipode axiom. 
    
    {\it Coassociativity.}
    We prove that $E_{d}^\prime$ respects the $\mathcal{R}$-relation ${}_{n+1}[\Delta]_{m} \circ {}_{n} [\Delta]_{m} \sim {}_{n}[\Delta]_{m+1} \circ {}_{n} [\Delta]_{m}$.
    By definition, we have
    \begin{align*}
        &E_d^\prime  ({}_{n+1}[\Delta]_{m}) \circ E_d^\prime ({}_{n} [\Delta]_{m})(0,e_{n+1})
        =E_d^\prime ({}_{n+1}[\Delta]_{m})\left( \sum^{d-1}_{k=1} \frac{1}{p} \binom{p^{l}}{k} e_{n+1}^{k} e_{n+2}^{d-k}, ~ e_{n+1}+e_{n+2}\right)\\
        =&\left(  \sum^{d-1}_{k=1} \frac{1}{p} \binom{p^{l}}{k} e_{n+1}^{k} (e_{n+2}+e_{n+3})^{d-k}+ \sum^{d-1}_{k=1} \frac{1}{p} \binom{p^{l}}{k} e_{n+2}^{k} e_{n+3}^{d-k}, \quad e_{n+1}+e_{n+2}+e_{n+3}\right) .
    \end{align*}
    On the other hand we have:
    \begin{align*}
       & E_d^\prime  ({}_{n}[\Delta]_{m+1} ) \circ E_d^\prime  ({}_{n} [\Delta]_{m})(0,e_{n+1})=E_d^\prime  ({}_{n}[\Delta]_{m+1} ) \left( \sum^{d-1}_{k=1} \frac{1}{p} \binom{p^{l}}{k} e_{n+1}^{k} e_{n+2}^{d-k}, ~ e_{n+1}+e_{n+2}\right)\\
        =&\left(  \sum^{d-1}_{k=1} \frac{1}{p} \binom{p^{l}}{k} (e_{n+1}+e_{n+2})^{k} e_{n+3}^{d-k}+ \sum^{d-1}_{k=1} \frac{1}{p} \binom{p^{l}}{k} e_{n+1}^{k} e_{n+2}^{d-k},\quad e_{n+1}+e_{n+2}+e_{n+3}\right).
    \end{align*}
    The equality in $E_d^\prime (n+2)$:\\
   \scalebox{.75}{\parbox{.8\linewidth}{%
   \begin{align*}
        \sum^{d-1}_{k=1} \frac{1}{p} \binom{p^{l}}{k} e_{n+1}^{k} (e_{n+2}+e_{n+3})^{d-k}+ \sum^{d-1}_{k=1} \frac{1}{p} \binom{p^{l}}{k} e_{n+2}^{k} e_{n+3}^{d-k} =\sum^{d-1}_{k=1} \frac{1}{p} \binom{p^{l}}{k} (e_{n+1}+e_{n+2})^{k} e_{n+3}^{d-k}+ \sum^{d-1}_{k=1} \frac{1}{p} \binom{p^{l}}{k} e_{n+1}^{k} e_{n+2}^{d-k}
   \end{align*}
   }}\\
    follows from the fact that these two quantities viewed in $\mathbb{Q} \otimes E_d^\prime (n+2)$ are equal to 
    $$\dfrac{1}{p}\left( (e_{n+1}+e_{n+2}+e_{n+3})^d-e_{n+1}^d-e_{n+2}^d-e_{n+3}^d \right);$$
    and the fact that $E_d^\prime (n+2)$ is a free $\mathds{Z}$-module.

    For $i > n+1$, we have:
    \begin{align*}
        &E_d^\prime  ({}_{n+1}[\Delta]_{m}) \circ E_d^\prime ({}_{n} [\Delta]_{m})(0,e_i)=E_d^\prime  ({}_{n+1}[\Delta]_{m})\left( 0, e_{i+1}\right)=\left(  0,e_{i+2}\right)\\
        =& E_d^\prime  ({}_{n}[\Delta]_{m+1})\left( 0, e_{i+1}\right)= E_d^\prime  ({}_{n}[\Delta]_{m+1} ) \circ E_d^\prime  ({}_{n} [\Delta]_{m}) (0,e_i).
    \end{align*}
    For $i < n+1$, one can use a similar argument, so that we obtain:
    $$E_d^\prime  ({}_{n+1}[\Delta]_{m}) \circ E_d^\prime ({}_{n} [\Delta]_{m}) =E_d^\prime  ({}_{n}[\Delta]_{m+1} ) \circ E_d^\prime  ({}_{n} [\Delta]_{m}).$$

    {\it Counit axiom.}
    We prove that $E_d^\prime$ respects the $\mathcal{R}$-relation ${}_{n}[\epsilon]_{m+1} \circ {}_{n} [\Delta]_{m} \sim \mathrm{id}_{n+m+1}$ and ${}_{n+1}[\epsilon]_{m} \circ {}_{n} [\Delta]_{m} \sim \mathrm{id}_{n+m+1}$.
    Here, we only consider the former.
    We have
    \begin{align*}
        &E_d^\prime ({}_{n}[\epsilon]_{m+1}) \circ E_d^\prime ({}_{n} [\Delta]_{m})(0,e_{n+1}) \\
        =& E_d^\prime ({}_{n}[\epsilon]_{m+1})\left( \sum^{d-1}_{k=1} \frac{1}{p} \binom{p^{l}}{k} e_{n+1}^{k} e_{n+2}^{d-k}, e_{n+1}+e_{n+2}\right)=(0,e_{n+1}) .
    \end{align*}
    For $i>n+1$, we have
    $$E_d^\prime ({}_{n}[\epsilon]_{m+1}) \circ E_d^\prime ({}_{n} [\Delta]_{m})(0,e_i)= E_d^\prime ({}_{n}[\epsilon]_{m+1})\left(0, e_{i+1}\right)=(0,e_i).$$
    For $i < n+1$, we have a similar argument, so we obtain $E_d^\prime ({}_{n}[\epsilon]_{m+1}) \circ E_d^\prime ({}_{n} [\Delta]_{m}) = E_d^\prime (\mathrm{id}_{n+m+1})$.

    {\it Cocommutativity.}
    We prove that $E_d^\prime$ respects the $\mathcal{R}$-relation ${}_{n} [\tau]_{m} \circ {}_{n} [\Delta]_{m} = {}_{n} [\Delta]_{m}$.
    By definitions, we have:
    \begin{align*}
        &E_d^\prime ({}_{n} [\tau]_{m}) \circ E_d^\prime ({}_{n} [\Delta]_{m}) (0,e_{n+1}) = E_d^\prime ({}_{n} [\tau]_{m}) \left( \sum^{d-1}_{k=1} \frac{1}{p} \binom{p^{l}}{k} e_{n+1}^{k} e_{n+2}^{d-k}, e_{n+1}+e_{n+2}\right) , \\
        =& \left( \sum^{d-1}_{k=1} \frac{1}{p} \binom{p^{l}}{k} e_{n+2}^{k} e_{n+1}^{d-k}, e_{n+2}+e_{n+1}\right) = \left( \sum^{d-1}_{k=1} \frac{1}{p} \binom{p^{l}}{p^{l}-k} e_{n+1}^{d-k}e_{n+2}^{k} , e_{n+1}+e_{n+2}\right) , \\
        =& \left( \sum^{d-1}_{k=1} \frac{1}{p} \binom{p^{l}}{k} e_{n+1}^{k} e_{n+2}^{d-k}, e_{n+1}+e_{n+2}\right) =  E_d^\prime ({}_{n} [\Delta]_{m}) (0,e_{n+1}) .
    \end{align*}
    For $i > n+1$, we have $$E_d^\prime ({}_{n} [\tau]_{m}) \circ E_d^\prime ({}_{n} [\Delta]_{m}) (0,e_i) = E_d^\prime ({}_{n} [\tau]_{m}) ( 0, e_{i+1}) = ( 0, e_{i+1}) = E_d^\prime ({}_{n} [\Delta]_{m}) (0,e_i).$$
    By a similar argument for $i < n+1$, one can obtain $E_d^\prime ({}_{n} [\tau]_{m}) \circ E_d^\prime ({}_{n} [\Delta]_{m}) = E_d^\prime ({}_{n} [\Delta]_{m})$.

    {\it Compatibility of $(\nabla, \eta)$ and $(\Delta, \epsilon)$.}
    We prove that $E_d^\prime$ respects the $\mathcal{R}$-relation of (\ref{202508072052k}).
    For $r \in \{1,2\}$, we have:
    {\small \begin{align*}
        E_d^\prime ({}_{n}[\Delta]_{m}) \circ E_d^\prime({}_{n}[\nabla]_{m}) (0,e_{n+r}) = E_d^\prime ({}_{n}[\Delta]_{m})(0,e_{n+1}) = \left( \sum^{d-1}_{k=1} \frac{1}{p} \binom{p^{l}}{k} e_{n+1}^{k} e_{n+2}^{d-k}, ~ e_{n+1}+e_{n+2}\right) .
    \end{align*}}
    On the other hand, one can directly verify the following:
    \begin{align*}
        &E_d^\prime ({}_{n+1}[\tau]_{m+1}) \circ E_d^\prime ({}_{n+2}[\Delta]_{m}) \circ E_d^\prime ({}_{n}[\Delta]_{m+1}) (0,e_{n+r}) \\
        =&
        \begin{cases}
            \left( \sum^{d-1}_{k=1} \frac{1}{p} \binom{p^{l}}{k} e_{n +1}^{k} e_{n+3}^{d-k}, e_{n+1} + e_{n+3} \right) & r=1 , \\
            \left( \sum^{d-1}_{k=1} \frac{1}{p} \binom{p^{l}}{k} e_{n +2}^{k} e_{n+4}^{d-k}, e_{n+2} + e_{n+4} \right) & r=2.
        \end{cases} 
    \end{align*}
    Hence, we obtain:
    \begin{align*}
        &E_d^\prime ({}_{n+1}[\nabla]_{m}) \circ E_d^\prime ({}_{n}[\nabla]_{m+2}) \circ E_d^\prime ({}_{n+1}[\tau]_{m+1}) \circ E_d^\prime ({}_{n+2}[\Delta]_{m}) \circ E_d^\prime ({}_{n}[\Delta]_{m+1}) (0,e_{n+r}) , \\
        =& \left( \sum^{d-1}_{k=1} \frac{1}{p} \binom{p^{l}}{k} e_{n+1}^{k} e_{n+2}^{d-k}, ~ e_{n+1}+e_{n+2}\right) .
    \end{align*}
    For $i > n+2$, we have:
    \begin{align*}
        &E_d^\prime ({}_{n+1}[\nabla]_{m}) \circ E_d^\prime ({}_{n}[\nabla]_{m+2}) \circ E_d^\prime ({}_{n+1}[\tau]_{m+1}) \circ E_d^\prime ({}_{n+2}[\Delta]_{m}) \circ E_d^\prime ({}_{n}[\Delta]_{m+1}) (0,e_{i}) , \\
        =& E_d^\prime ({}_{n+1}[\nabla]_{m}) \circ E_d^\prime ({}_{n}[\nabla]_{m+2}) \circ E_d^\prime ({}_{n+1}[\tau]_{m+1}) \circ E_d^\prime ({}_{n+2}[\Delta]_{m})  (0,e_{i+1}) , \\
        =& E_d^\prime ({}_{n+1}[\nabla]_{m}) \circ E_d^\prime ({}_{n}[\nabla]_{m+2}) \circ E_d^\prime ({}_{n+1}[\tau]_{m+1})   (0,e_{i+2}) , \\
        =& E_d^\prime ({}_{n+1}[\nabla]_{m}) \circ E_d^\prime ({}_{n}[\nabla]_{m+2})   (0,e_{i+2}) 
        = E_d^\prime ({}_{n+1}[\nabla]_{m}) (0, e_{i+1}) 
        = (0,e_i) .
    \end{align*}
    which is equal to $E_d^\prime ({}_{n}[\Delta]_{m}) \circ E_d^\prime({}_{n}[\nabla]_{m}) (0,e_{i})$.
    By a similar argument for $i < n+1$, we obtain 
    \begin{align*}
    &E_d^\prime ({}_{n+1}[\nabla]_{m}) \circ E_d^\prime ({}_{n}[\nabla]_{m+2}) \circ E_d^\prime ({}_{n+1}[\tau]_{m+1}) \circ E_d^\prime ({}_{n+2}[\Delta]_{m}) \circ E_d^\prime ({}_{n}[\Delta]_{m+1}) \\
     =& E_d^\prime ({}_{n}[\Delta]_{m}) \circ E_d^\prime({}_{n}[\nabla]_{m}).
    \end{align*}

    We prove that $E_d^\prime$ respects the $\mathcal{R}$-relation in the second part of (\ref{202508111506k}), i.e. ${}_{n}[\Delta]_{m} \circ {}_{n}[\eta]_{m} \sim {}_{n}[\eta]_{m+1} \circ {}_{n}[\eta]_{m}$.
    By definition, one may see that 
    \begin{align*}
        E_d^\prime ( {}_{n}[\Delta]_{m} )\circ E_d^\prime ({}_{n}[\eta]_{m} )(0,e_i) = (0,e_k) 
    \end{align*}
    where $k= i$ if $i \leq n$ and $k = i+2$ otherwise.
    Likewise, we have
    \begin{align*}
        E_d^\prime ( {}_{n}[\eta]_{m+1} ) \circ E_d^\prime ( {}_{n}[\eta]_{m}) (0,e_i) = (0,e_k) .
    \end{align*}
    Hence, we obtain $E_d^\prime ( {}_{n}[\Delta]_{m} )\circ E_d^\prime ({}_{n}[\eta]_{m} ) = E_d^\prime ( {}_{n}[\eta]_{m+1} ) \circ E_d^\prime ( {}_{n}[\eta]_{m})$.

    {\it Antipode axiom.}
    We prove that $E_d^\prime$ respects the $\mathcal{R}$-relation ${}_{n}[\nabla]_{m} \circ {}_{n}[S]_{m+1} \circ {}_{n}[\Delta]_{m} \sim {}_{n}[\eta]_{m} \circ {}_{n} [\epsilon]_{m} \sim {}_{n}[\nabla]_{m} \circ {}_{n+1}[S]_{m} \circ {}_{n}[\Delta]_{m}$.
    Here, we only consider the former case since the other is proved analogously.
    We have:
    \begin{align*}
        &E_d^\prime ({}_{n}[\nabla]_{m}) \circ E_d^\prime ({}_{n}[S]_{m+1})\circ E_d^\prime ({}_{n}[\Delta]_{m})(0,e_{n+1}) \\
        =& E_d^\prime ({}_{n}[\nabla]_{m}) \circ E_d^\prime ({}_{n}[S]_{m+1})\left( \sum^{d-1}_{k=1} \frac{1}{p} \binom{p^{l}}{k} e_{n+1}^{k} e_{n+2}^{d-k}, ~ e_{n+1}+e_{n+2}\right),\\
       =& E_d^\prime (\nabla \ast  \mathrm{id}_{n-1}) \left( \sum^{d-1}_{k=1} \frac{1}{p} \binom{p^{l}}{k} (-e_{n+1})^{k} e_{n+2}^{d-k}+ \frac{1+(-1)^{d}}{p}e_{n+1}^{d} , ~ -e_{n+1}+e_{n+2}\right),\\
       =&\left( ( \sum^{d-1}_{k=1} \frac{1}{p} \binom{p^{l}}{k}(-1)^k ) e_{n+1}^{d}+ \frac{1+(-1)^{d}}{p}e_{n+1}^{d} , ~ 0\right),
    \end{align*}
    where
    \begin{align*}
        \sum^{d-1}_{k=1} \frac{1}{p} \binom{p^{l}}{k}(-1)^k+\frac{1+(-1)^d}{p} = \sum^{d}_{k=0} \frac{1}{p} \binom{p^{l}}{k}(-1)^k =\frac{1}{p}(1-1)^{p^l} = 0 .
    \end{align*}
    On the other hand, we have $E_d^\prime ({}_{n}[\eta]_{m}) \circ E_d^\prime ({}_{n} [\epsilon]_{m})(0,e_{n+1})=(0,0)$.
    
    For $i>n+1$, we have:
    \begin{align*}
        &E_d^\prime ({}_{n}[\nabla]_{m}) \circ E_d^\prime ({}_{n}[S]_{m+1})\circ E_d^\prime ({}_{n}[\Delta]_{m})(0,e_i) 
        = E_d^\prime ({}_{n}[\nabla]_{m}) \circ E_d^\prime ({}_{n}[S]_{m+1})\left( 0, e_{i+1}\right)\\
       =& E_d^\prime ({}_{n}[\nabla]_{m}) \left( 0, e_{i+1}\right)
       =\left( 0, e_{i}\right)
    \end{align*}
    and $E_d^\prime ({}_{n}[\eta]_{m}) \circ E_d^\prime ({}_{n} [\epsilon]_{m}) (0,e_i)=(0,e_i)$.
    By a similar argument for $i < n+1$, one can conclude that $$E_d^\prime ({}_{n}[\nabla]_{m}) \circ E_d^\prime ({}_{n}[S]_{m+1})\circ E_d^\prime ({}_{n}[\Delta]_{m}) =E_d^\prime ({}_{n}[\eta]_{m}) \circ E_d^\prime ({}_{n} [\epsilon]_{m}).$$
\end{proof}

\begin{propo} \label{202508041518}
    For $d = p^{l}$, the extension in $\mathcal{F}(\mathsf{ab};\mathds{Z})$:
\begin{align*}
    0 \to S^{d} \to E_{d} \to I \to 0
\end{align*}
induced by the usual inclusion and projection, is a generator of   $\mathrm{Ext}^{1}_{\mathcal{F}(\mathsf{ab};\mathds{Z})} ( I , S^{d} ) \cong \mathds{Z}/ p$.
\end{propo}
\begin{proof}
Suppose that the epimorphism $E_{d} \to I$ has a section $s : I \to E_{d}$ in $\mathcal{F} ( \mathsf{ab};\mathds{Z})$.
Recall that the category $\mathcal{F}(\mathsf{ab};\mathds{Z})$ is equivalent to the category of additive functors on $\mathds{Z} \mathsf{ab}$ where $\mathds{Z} \mathsf{ab}$ is the $\mathsf{Ab}$-enriched category obtained by $\mathds{Z}$-linearizing $\mathsf{ab}$.
Then, for $\bar{\Delta} {:=} \Delta - (\mathrm{id}_{1} \oplus \eta) -  (\eta \oplus \mathrm{id}_{1}) \in \mathds{Z} \mathsf{ab} (1,2)$, the following diagram commutes:
$$
\begin{tikzcd}
    \mathds{Z} \ar[d, "I(\bar{\Delta})"] \ar[r, "s_1"] & E_{d} (1) \ar[d, "E_{d} (\bar{\Delta})"] \\
    \mathds{Z}^{\oplus 2} \ar[r, "s_2"] & E_{d} (2)
\end{tikzcd}
$$
Since $I( \Delta) (e_1) = e_1+e_2= I (\mathrm{id}_{1} \oplus \eta) (e_1) + I ( \eta \oplus \mathrm{id}_{1}) (e_1)$, we have $I(\bar{\Delta}) (e_1) = 0$, in particular $(s_2 \circ I (\bar{\Delta})) (e_1) = 0$.
By the commutativity of the above diagram, we obtain 
\begin{align} \label{202507211202}
    (E_{d}( \bar{\Delta}) \circ s_1) (e_1) = 0.
\end{align}
Since $s_1$ is assumed to be a section of $E_{d}(1) = S^d (1) \oplus I(1) \to I(1)$, there exists $\lambda \in \mathds{Z}$ such that $s_1 ( e_1) = (\lambda e_1^{d}, e_1)$.
By definition of $E_{d}$, we have
\begin{align*}
    E_{d}(\bar{\Delta}) (\lambda e_1^{d}, e_1) &= ( \lambda \{ (e_1 + e_2)^{d} - e_1^{d} -e_2^{d} \} + \sum^{d-1}_{k=1} \frac{1}{p} \binom{p^{l}}{k} e_{1}^{k} e_{2}^{d-k}, 0) , \\
    &= ( \sum^{d-1}_{k=1} \left( \lambda \binom{p^l}{k} +  \frac{1}{p} \binom{p^{l}}{k} \right) e_{1}^{k} e_{2}^{d-k}, 0) .
\end{align*}
By (\ref{202507211202}), we obtain $ \frac{1}{p} \binom{p^{l}}{k} = - \lambda \binom{p^l}{k}$ for $0<k< d= p^{l}$.
It leads to $\lambda = -\frac{1}{p}$ which contradicts with the condition $\lambda \in \mathds{Z}$.

\end{proof}

\begin{remark} \label{202508281012k}
    Let $n \in \nat^*$.
    The extension in Proposition \ref{202508041518} induces an extension of $\mathrm{GL}_{n}(\mathds{Z})$-modules where $\mathrm{GL}_{n}(\mathds{Z})$ is the general linear group over integers:
    \begin{align*}
        0 \to S^{d} (n) \to E_{d} (n) \to I (n) \to 0
    \end{align*}
    It is notable that, for $d= 2^l$, i.e. $p= 2$, this extension does not split.
    Suppose that $E_d(n) \to I(n)$ has a section $s: I(n) \to E_d(n)$  preserving $\mathrm{GL}_{n}(\mathds{Z})$-actions.
    For $e_n \in I(n)$, we have $s(e_n)=(\sum^{d}_{k=0} v_k e_n^{k} , e_n) \in E_d(n)$ for some $v_k \in S^{d-k}(n-1)$.
    We consider $\mathrm{id}_{n-1} \oplus S \in \mathrm{GL}_{n}(\mathds{Z}) \subset \mathsf{ab}(n,n)$. Since $s$ preserves $\mathrm{GL}_{n}(\mathds{Z})$-action we have :
    $E_d(\mathrm{id}_{n-1} \oplus S)s(e_n)=s(I(\mathrm{id}_{n-1} \oplus S)(e_n)).$
    Hence
    \begin{align*}
        ( \sum^{d}_{k=0} v_k (-e_n)^{k} + e_n^{d}, -e_n) = ( - \sum^{d}_{k=0} v_k e_n^{k} , -e_n) .
    \end{align*}
    By comparing the coefficients of $e_n^{d}$, we obtain $(-1)^d v_d + 1 = - v_d$, equivalently, $-2v_d = 1$, which contradicts with the assumption $v_d \in S^{d-d}(n-1) = \mathds{Z}$.
\end{remark}

\section{On $\mathrm{Ext}^{*}_{\mathcal{F}(\mathsf{gr};\mathds{Z})} (T^c \circ \mathfrak{a}, \Lambda^d \circ \mathfrak{a})$} \label{202508151448}
Let $\Lambda^d: \mathsf{ab} \to \mathsf{Ab}$ be the $d$-th exterior power functor  defined on an object $G$ of $\mathsf{ab}$ by $\Lambda^d(G)=G^{\otimes d}/I$ where $I$ is the submodule generated by all elements of the form $g_1 \otimes \ldots\otimes g_i \otimes g_{i+1} \otimes \ldots \otimes g_d+ g_1 \otimes \ldots\otimes g_{i+1} \otimes g_i \otimes \ldots \otimes g_d$ for all $1 \leq i <d$ and $g_j \in G$. The aim of this appendix is to explain how to recover \cite[Corollary 11.4 (2)]{arone2025polynomial} combining the projective resolution (\ref{202507301119}) and the method used in Section \ref{202508081128}. 
Let $\textrm{Surj}(d,c)$ be the set of surjections from $\{1, \ldots, d\}$ to $\{1, \ldots, c\}$ and $R(d,c){:=}{\  }_{S_d \backslash } \textrm{Surj}(d,c)$ the ordered partitions of $d$ into $c$ parts. We have the following:
\begin{propo}  \label{202508281110}

For $(c,d) \in (\nat^*)^2$, we have:
\begin{align*}
        \mathrm{Ext}^{i}_{\mathcal{F}(\mathsf{gr};\mathds{Z})} ( T^c \circ \mathfrak{a}, \Lambda^d \circ \mathfrak{a}) =
        \begin{cases}
            \mathds{Z}[R(d,c)]& i =d-c ,\\
            0 & \mathrm{otherwise} .
        \end{cases}
    \end{align*}
\end{propo}
\begin{proof}
    We start with calculating $ \mathrm{Ext}^{i}_{\mathcal{F}(\mathsf{gr};\mathds{Z})} (  \mathfrak{a}, \Lambda^d \circ \mathfrak{a})$. By the complex (\ref{202507291740}), these Ext-groups are given by the cohomology of the complex:
    \begin{align*} 
  \overline{\Lambda^d}(\mathds{Z}) \xrightarrow{\delta^0} \text{cr}_2\Lambda^d(\mathds{Z}, \mathds{Z}) \xrightarrow{\delta^1} \ldots \rightarrow \text{cr}_d\Lambda^d(\mathds{Z},\ldots, \mathds{Z}) \xrightarrow{\delta^{d-1}} \text{cr}_{d+1}\Lambda^d(\mathds{Z},\ldots, \mathds{Z}) \rightarrow \dots  
\end{align*}
Since $\Lambda^d(\mathds{Z}^{\oplus d})=\mathds{Z}$ and $\Lambda^d(\mathds{Z}^{\oplus i})=0$ for $i < d$, we obtain, using (\ref{202508071131}), $\text{cr}_d\Lambda^d(\mathds{Z},\ldots, \mathds{Z})=\mathds{Z}$ and $\text{cr}_i\Lambda^d(\mathds{Z},\ldots, \mathds{Z})=0$ for $i \neq d$. Hence, we obtain:
\begin{align} \label{202508071509}
        \mathrm{Ext}^{i}_{\mathcal{F}(\mathsf{gr};\mathds{Z})} ( \mathfrak{a}, \Lambda^d \circ \mathfrak{a}) =
        \begin{cases}
            \mathds{Z}& i =d-1 , \\
            0 & \mathrm{otherwise} .
        \end{cases}
    \end{align}
    Since the graded functor $\Lambda^{\bullet} \circ \mathfrak{a}=(\Lambda^{d} \circ \mathfrak{a})_{d\in \nat}$ is an exponential functor of $\mathcal{F}(\mathsf{gr};\mathds{Z})$, we have:
\begin{eqnarray*}
\mathrm{Ext}^{*}_{\mathcal{F}(\mathsf{gr};\mathds{Z})} ( T^c \circ \mathfrak{a}, \Lambda^{d} \circ \mathfrak{a})
\cong  \underset{i_1+ \ldots+i_c=d}{\bigoplus}\mathrm{Ext}^*_{\mathcal{F}(\mathsf{gr}\times \ldots \times \mathsf{gr};\mathds{Z})}\big( \mathfrak{a}^{\boxtimes c},  \Lambda^{i_1}\circ \mathfrak{a} \boxtimes \ldots \boxtimes \Lambda^{i_c}\circ \mathfrak{a} \big)
\end{eqnarray*}
Since $\mathrm{Ext}^{i}_{\mathcal{F}(\mathsf{gr};\mathds{Z})} ( \mathfrak{a}, \Lambda^d \circ \mathfrak{a})$ is torsion free, using a K\"unneth formula, (similar to Proposition \ref{202508211555} replacing $S^\bullet$ with $\Lambda^\bullet$), we obtain:
\begin{align*}
\mathrm{Ext}^{*}_{\mathcal{F}(\mathsf{gr};\mathds{Z})} ( T^c \circ \mathfrak{a}, \Lambda^{d} \circ \mathfrak{a})
&\cong  \underset{i_1+ \ldots+i_c=d}{\bigoplus} \Big(\underset{1\leq k\leq c}{\bigotimes}\mathrm{Ext}^*_{\mathcal{F}(\mathsf{gr}\times \ldots \times \mathsf{gr};\mathds{Z})}\big( \mathfrak{a},  \Lambda^{i_k}\circ \mathfrak{a}  \big)\Big) , \\
&\cong  \begin{cases}
            \mathds{Z}[R(d,c)]& i =d-c ,\\
            0 & \mathrm{otherwise} .
        \end{cases}\\
\end{align*}
where the last isomorphism is obtained using (\ref{202508071509}).
\end{proof}

\vspace{.5cm}
\textbf{Acknowledgement:} 
The authors would like to express their sincere gratitude to Gregory Arone for providing them useful comments and suggestions.

The first author is supported by a KIAS Individual Grant MG093702 at Korea Institute for Advanced Study.
In particular, this paper was completed during his stay in Aix-Marseille University, with support from KIAS. He is grateful to the Institut de Mathématique de Marseille for their hospitality during this visit.

The second author is partially supported by project ANR-22-CE40-0008 SHoCoS.

\bibliography{reference}{}
\bibliographystyle{plain}

\end{document}